\documentclass[11pt,a4paper,reqno]{amsart}
\usepackage{vmargin,color}
\usepackage[latin1]{inputenc}
\usepackage{amsmath,amsfonts,amsthm,epsfig,graphicx,mathtools,tikz-cd, amssymb,amsbsy}
\usepackage[colorlinks,citecolor=blue,linkcolor=blue]{hyperref}
\usepackage{caption}
\usepackage{autonum}
\usepackage{amssymb} 
\usepackage{amsaddr}
\usepackage{etoolbox}
\patchcmd{\section}{\bfseries}{\bfseries\boldmath }{}{}
\patchcmd{\subsection}{\bfseries}{\bfseries\boldmath }{}{}

\def\\rho{\mathcal{\rho}}
\def\cP{\mathcal{P}}

\def\N{\mathbb N}
\def\Z{\mathbb Z}

\def\R{\mathbb R}
\def\T{\mathbb T}

\def\\rho\rho{\mathbf{\rho}}

\def\eps{\varepsilon}

\def\eff{\mathrm{eff}}
\def\sym{\mathrm{sym}}

\def\Lip{{\rm Lip}}

\def\Id{{I}\,}

\def\pa{\partial}

\def\ds{\displaystyle}

\DeclareMathAlphabet{\mathup}{OT1}{\familydefault}{m}{n}
\newcommand{\dx}[1]{\mathop{}\!\mathup{d} #1}

\newcommand{\intt}[1]{\int_{\T^d} #1 \dx{x}}
\newcommand{\intto}[1]{\int_{\T} #1 \dx{x}}
\DeclarePairedDelimiter{\abs}{\lvert}{\rvert}
\DeclarePairedDelimiter{\norm}{\lVert}{\rVert}
\DeclarePairedDelimiter{\bra}{(}{)}
\DeclarePairedDelimiter{\pra}{[}{]}
\DeclarePairedDelimiter{\set}{\{}{\}}
\DeclarePairedDelimiter{\skp}{\langle}{\rangle}

\newcommand{\Leb}{\ensuremath{{L}}}
\newcommand{\SobH}{\ensuremath{{H}}}

\newtheorem{theorem}{Theorem}
\newtheorem{remark}{Remark}[section]
\newtheorem{definition}{Definition}[section]

\newtheorem{proposition}[theorem]{Proposition}
\newtheorem{lemma}[theorem]{Lemma}
\newtheorem{corollary}[theorem]{Corollary}

\newtheorem{innercustomthm}{Theorem}
\newenvironment{customthm}[1]
  {\renewcommand\theinnercustomthm{#1}\innercustomthm}
  {\endinnercustomthm}

\numberwithin{equation}{section}
\numberwithin{figure}{section}
\numberwithin{theorem}{section} 
\usetikzlibrary{arrows, arrows.meta, decorations.pathmorphing}

\tikzset{snake it/.style={decorate, decoration=snake}}

\title{On the diffusive-mean field limit for weakly interacting diffusions exhibiting phase transitions}
\author{Matias G. Delgadino}
\address{Department of Mathematics, Pontifical Catholic University of Rio de Janeiro}
\email{matias.delgadino@puc-rio.br}
\author{Rishabh S. Gvalani \& Grigorios A. Pavliotis}
\email{rg1314@ic.ac.uk}
\email{g.pavliotis@imperial.ac.uk}
\address{Department of Mathematics, Imperial College London}
\thanks{}
\begin{document}

\begin{abstract}
    The objective of this article is to analyse the statistical behaviour of a large number of weakly interacting diffusion processes evolving under the influence of a periodic interaction potential. We focus our attention on the combined mean field and diffusive (homogenisation) limits. In particular, we show that these two limits do not commute if the mean field system constrained to the torus undergoes a phase transition, that is to say if it admits more than one steady state. A typical example of such a system on the torus is given by the noisy Kuramoto model of mean field plane rotators. As a by-product of our main results, we also analyse the energetic consequences of the central limit theorem for fluctuations around the mean field limit and derive optimal rates of convergence  in relative entropy of the Gibbs measure to the (unique) limit of the mean field energy below the critical temperature.  
\end{abstract}

\maketitle
 {\small
 \tableofcontents
 }

\vspace{0.2cm}

\section{Introduction}
\subsection{Overview}
The study of large systems of  interacting particles in the presence of noise has attracted a large amount of interest in recent years. This is largely due to the fact that they pose challenging mathematical questions and that they appear in several applications, ranging from the theory of random matrices~\cite{Ser15} and the construction of K\"ahler--Einstein metrics~\cite{BO2018} to the design of algorithms for global optimisation~\cite{RV2018,KPP2019}, biological models of chemotaxis~\cite{FJ17}, and models of opinion formation~\cite{GPW17}.

We place ourselves in the setting of a system of weakly interacting diffusion processes as in~\cite{Oel84}. It is well-known that, under appropriate assumptions on the interaction and confining potentials, one can pass to the mean field limit as $N \to \infty$ to obtain the so-called McKean--Vlasov equation~\eqref{eps=1} for the limit of the $N$-particle empirical measure. More precisely, given chaotic initial data, the empirical measure associated to the system of particles converges weakly to the weak solution of the McKean--Vlasov equation. Formally, one can say that the law of the $N$-particle system decouples and converges to $N$ copies of the mean field McKean--Vlasov equation. This corresponds to a strong law of large numbers (LLN) for the the empirical measure. A natural question to ask then is whether one can obtain a second order characterisation of this convergence, i.e. a central limit theorem (CLT). 

Partial results in this direction do exist: Fernandez and M\'el\'eard~\cite{fernandez1997hilbertian} obtained a finite-time horizon version of the CLT. They showed that the fluctuations around the mean field limit are described in the large $N$-limit by a Gaussian random field which itself is the solution of a linear stochastic PDE.  Additionally, Dawson~\cite{dawson1983critical} proved an equilibrium CLT for the empirical measure of  a system of particles in a bistable confining potential and Curie--Weiss interaction. The interesting feature of Dawson's system is that exhibits a phase transition, i.e. for a certain value of the interaction strength the system transitions from having one invariant measure to having multiple. Dawson showed that below the phase transition point equilibrium fluctuations are described by Gaussian random field, similar to the result in~\cite{fernandez1997hilbertian}. However, at the critical temperature the fluctuations become non-Gaussian and are given by the invariant measure of nonlinear SDE. These are non-Gaussian fluctuations are persistent and are characterised by a longer time scale, exhibiting the well known phenomenon of critical slowing down (cf.~\cite{Shi87} for a less rigorous derivation of similar results). We are not aware of any results on the limiting behaviour of the fluctuations that have been obtained ahead of the phase transition.

Fluctuations around the McKean--Vlasov mean field limit for a system of weakly interacting diffusions with an internal degree of freedom were also studied recently in~\cite{BBC19}. Under the assumption of scale separation between the macroscopic and microscopic dynamics, a large deviations principle (LDP) was established for the slow dynamics, valid in the combined limit of infinite scale separation ($\eps \to 0$) and of the number of particles going to infinity ($N \to \infty$). This LDP was then used to deduce information about the fluctuations around the mean field limit and to also offer partial justification for the so-called Dean equation, a stochastic partial differential equation used in dynamical density functional theory which combines, formally, the mean field limit and central limit theorem results for the system of weakly interacting diffusions. Furthermore, the connection between the LDP framework and the Chapman--Enskog approach to the study of the hydrodynamic limit was discussed in detail. The crucial assumption made by the authors was that the microscopic dynamics has a unique stationary state, i.e. that no phase transitions occur.

The prototype of the systems we consider is the following system of $N$ interacting SDEs on $\R$
\begin{align}
\dx{X}_t^{i}= - \frac{1}{N} \sum_{j=1,j \neq 1}^N \sin\bra*{2 \pi \bra*{X_t^{i}- X_t^{j}}} \dx{t} + \sqrt{2\beta^{-1}}dB_t^{i}
\label{R}
\end{align}
where the $B_t^i$ are independent $\R$-valued Wiener processes. The interesting feature about the above system is that the interaction potential is $1$-periodic. As a consequence of this, the behaviour of~\eqref{R} is influenced heavily by the corresponding quotiented process on $\T$ (the one dimensional unit torus). The quotiented system on the torus is in fact the noisy Kuramoto model for mean field plane rotators\footnote{Additionally, its reversible Gibbs measure corresponds to the classical Heisenberg $XY$ model for lattice systems with continuous spins and mean field interaction. This is immediately apparent when one considers the associated Hamiltonian which is given by:
$$
H^N(x_1, \dots, x_N)= -(2N)^{-1} \sum_{i,j}\cos(2 \pi(x_i-x_j))=- (2N)^{-1} \sum_{i,j}S_i \cdot S_j \, ,
$$ where $S_i= (\cos(2\pi x_i), \sin(2 \pi x_i)), $cf.~\cite[Chapter 9]{FV2018} or~\cite{BGP10}.    }~\cite{BGP10,CGPS19}. Indeed (cf. Proposition~\ref{XY}), one can show that the corresponding mean field limit on the torus exhibits a phase transition. A more complete picture of the local bifurcations and phase transitions for the McKean--Vlasov equation on the torus can be found in~\cite{CGPS19}\footnote{In later sections, as a technical requirement, we will consider the same system with an additional confining potential in order the break the translation symmetry of the noisy Kuramoto system which leads to degeneracy of minimisers ahead of the phase transition (cf. Proposition~\ref{XY}). }.

In the spirit of Dawson, our main objective is to study fluctuations in the presence of phase transitions. However, instead of the phase transitions of the system on $\R$, we will be concerned with the phase transitions of the quotiented system on $\T$. Furthermore, we study the diffusive limit which can be thought of as the first step in understanding fluctuations of the $N$-particle system.   Although we do discuss the implications of a full CLT (cf. Section~\ref{ftheorem}), we concern ourselves in this paper mainly with the combined diffusive-mean field limits. 

The problem that we study in this paper is closely related, and simpler, to the one studied in~\cite{BBC19}: scale separation arises naturally in our case due to the disparity between the period of the interaction potential which is the characteristic length scale of the microscopic dynamics, and the long, diffusive length/time scale. The ``hydrodynamics'' in our problem is described by the (homogenised) heat equation, with the effective covariance matrix given by the standard homogenisation formula: compare Equation~\eqref{kipnisvaradhan} below with formulas (3.14) and (3.15) in~\cite{BBC19}. However, in contrast to~\cite{BBC19} our main focus is on the effect of the presence of phase transitions at the microscopic scale on the effective/macroscopic dynamics. We are, in particular, interested on the effect of phase transitions on the (lack of) commutativity between the homogenisation and mean field limits.

Before we discuss what we mean by the combined limit, we remind the reader of what we mean by the diffusive limit. For a fixed number of particles $N>0$ for the system in~\eqref{R}, a natural question to ask is how the law of the system behaves under the diffusive rescaling, i.e. if $\rho^{\eps,N}= \mathrm{Law}\bra*{\eps X_{t/\eps^2}^1 ,\dots ,\eps X_{t/\eps^2}^N}$ then what is the limit as $\eps \to 0$ of $\rho^{\eps,N}$.  The answer to this question can be obtained by using classical arguments from periodic homogenisation~\cite[Chapter 20]{pavliotis2008multiscale}\cite{BLP11}. It turns out that $\rho^{\eps,N}$ converges to $\rho^{N,*}$, the solution of the heat equation with a positive definite effective covariance matrix $A^{\eff,N}$(cf. Section~\ref{epsthenN} and Equation ~\eqref{effectiveeq}), which can be obtained by solving a Poisson equation for the generator of the process on $\T^N$ (cf. Equation~\eqref{Poisson}). Another way of reinterpreting this result is by saying that the system of particles~\eqref{R} converge in law to an $N$-dimensional Brownian motion with covariance $A^{\eff,N}$. A natural next question to ask is how does the covariance matrix $A^{\eff,N}$, and by extension the heat equation, behave in the limit as $N \to \infty$.

One could also ask the question the other way around. As discussed previously, for a fixed $\eps>0$, we can pass to the mean field limit  as $N \to \infty$ in $\rho^{\eps,N}$ to obtain  $N$ copies of the solution of the nonlinear McKean--Vlasov equation, $\rho^{\eps, \otimes N}$. The natural question to ask now is whether we can understand the behaviour of $\rho^{\eps,\otimes N}$ as $\eps \to 0$. This dichotomy is illustrated in Figure~\ref{fig:combined}. Starting from the rescaled law $\rho^{\eps,N}$, we can take the limit $\eps \to 0$ first followed by $N \to \infty$ if we move in the clockwise direction or the other way around in the anti-clockwise direction. Whether these two limits commute depends heavily on the ergodic properties of the quotiented process on $\T^N$ and its behaviour in the mean field limit. Our main result asserts that the two limits commute at high temperatures (small $\beta$) and thus the combined limit is well-defined in this regime. However, at low temperatures (large $\beta$) and in particular, in the presence of a phase transition (cf. Definition~\ref{pt}), we can construct initial data such that the two limits do not commute.

The problem of non-commutativity between the mean field and homogenisation limits was also studied in~\cite{GP18}. In this paper, a system of weakly interacting diffusions in a two-scale, locally periodic confining potential subject to a quadratic, Curie-Weiss, interaction potential was considered. It was shown that, although the combined homogenization-mean field limit leads to coarse-grained McKean-Vlasov dynamics that have the same functional form, the effective diffusion (mobility) tensor and the coarse-grained (Fixman) potential are different, depending on the order with which we consider these two limits (for non-separable two-scale potentials). In particular, the phase diagrams for the effective dynamics can be different, depending on the order with which we take the limits. A more striking manifestation of the non-commutativity between the two limits can be observed at small but finite values of $\eps$, the parameter measuring scale separation: it is easy to construct examples where the mean field PDE, for small, finite $\eps$ can have arbitrarily many stationary states, the homogenised McKean--Vlasov equation (corresponding to the choice of sending first $\eps \to 0$ and then $N \to \infty$) is characterised by a convex free energy functional and, thus, a unique steady state.

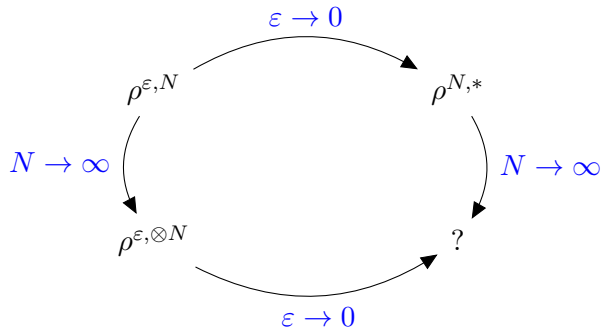
\begin{figure}
\begin{center}
$$
\begin{tikzcd}
\node (X1) at (0,0) {\rho^{\eps,N}};
\node (X2) at (4,0) {\rho^{N,*}};
\draw[>=triangle 45, bend left, ->] (X1) to node [above]{\textcolor{blue}{\eps \to 0}} (X2);
\node (X3) at (0,-2) {\rho^{\eps,\otimes N}};
\draw[>=triangle 45, bend right, ->] (X1) to node [left]{\textcolor{blue}{N \to \infty}} (X3);
\node (X4) at (4,-2) {?};
\draw[>=triangle 45, bend left, ->] (X2) to node [right]{\textcolor{blue}{N \to \infty}} (X4);
\draw[>=triangle 45, bend right, ->] (X3) to node [below]{\textcolor{blue}{\eps \to 0}} (X4);
\end{tikzcd}
$$
\end{center}
\caption{The combined diffusive-mean field limit and the (possible) non-commutativity of the two limits}
\label{fig:combined}
\end{figure}
\subsection{Set up and preliminaries}
We denote by $\T^d$ the $d$ dimensional unit torus (which we identify with $[0,1)^d$) and use the standard notation of $\Leb^p(\T^d)$ and $ \SobH^s(\T^d)$ for the Lebesgue and $\Leb^2$-Sobolev spaces, respectively.  We will use $\SobH^s_0(\T^d)$ to denote the homogeneous $\Leb^2$-Sobolev spaces. We denote by the $C^k(\T^d),C^\infty(\T^d)$ the space of $k$-times ($k \in \N$) continuously differentiable and smooth functions, respectively.    

We denote by $\cP(\Omega)$ the space of all Borel probability measures on $\Omega$ having finite second moment, with $\Omega$ some Polish metric space. We will use $d_1$ and $d_2$ to denote the $1$ and $2$-Wasserstein distances, respectively, on $\cP(\R^d)$ and $\cP(\T^d)$. Similarly we will use $\mathfrak{D}_1$ and $\mathfrak{D}_2$ for the $1$ and $2$-Wasserstein distances, respectively, on $\cP(\cP(\R^d))$ and $\cP(\cP(\T^d))$. In the sequel, any limit of a sequence of measures $\set*{\rho_n}_{n \in \N} \subset \cP(\Omega) $ unless otherwise specified should be understood as a limit in the weak-$*$ topology relative to $C_b(\Omega)$, i.e. tested against bounded, continuous functions. We will often use the same notation for a measure and its density if the density is well-defined.

We consider a large number $N\in\N$ of indistinguishable interacting particles $\{X_t^i\}_{i=1}^N$ in $\R^d$, where both the interaction and confining potentials are periodic and highly oscillatory. In particular, we consider the system
\begin{equation}\label{interacting SDEs}
\begin{cases}
 \dx{X}_t^{\eps,i}= -\eps^{-1}\nabla V(\eps^{-1} X_t^{\eps,i}) - \frac{1}{N} \sum_{j \neq i}^N \eps^{-1} \nabla W(\eps^{-1}(X_t^{\eps,i}-X_t^{\eps,j})) \dx{t} + \sqrt{2 \beta^{-1}}d B_t^i\\
 \mathrm{Law}\; \big((X_0^{\eps,1},...,X_0^{\eps,N})\big)=\rho^{\eps,N}_0,
\end{cases}
\end{equation}
where $W:\R^d\to \R$ and $V:\R^d\to \R$ are smooth $1$-periodic interaction and confining potentials, respectively,  $\eps\ll 1$ is the period size, $\beta>0$ is the inverse temperature, $\rho^{\eps,N}_0\in\mathcal{P}_{\sym}\big((\R^d)^N\big)$ is the initial distribution of the particles which might depend on the period size, and $\{B^i_t\}_{i=1}^N$ are independent Wiener processes. We are interested in understanding the joint limit when the period of oscillations goes to $0$ ($\eps\to 0$) and the number of particles tends to infinity ($N\to\infty$).

We consider the joint law of the particle positions which is given by
$$ 
\rho^{\eps,N}(t)=\,\mathrm{Law}\big((X^{\eps,1}_t,...,X^{\eps,N}_t)\big)\in\mathcal{P}_{\sym}\big((\R^d)^N\big),
$$
where $\cP_{\sym}((\R^d)^N)$ is as defined in~\eqref{symmetric measures}. The law evolves through the following linear forward Kolmogorov or Fokker--Planck equation
\begin{equation}\label{eq:LinearKolmogorov}
\begin{cases}
 \partial_t\rho^{\eps,N}=\beta^{-1} \Delta\rho^{\eps,N}+\nabla\cdot(\nabla H^N_\eps \rho^{\eps,N})&\mbox{on $(0,\infty)\times \big(\R^d\big)^N$}\\
 \rho^{\eps,N}(0)=\rho^{\eps,N}_0(x) &\mbox{on $\big(\R^d\big)^N$}
\end{cases}    
\end{equation}
where $H^N_\eps:(\R^d)^N\to \R$ is given by
\begin{equation}\label{eq:H}
H^N_\eps(x_1,...,x_N)= \sum_{i=1}^N V(\eps^{-1} x_i)+\frac{1}{2N}\sum_{i=1}^N\sum_{\substack{j=1\\j\ne i}}^N W(\eps^{-1}(x_i-x_j)).    
\end{equation}
The main objective of this paper is to study
$$
\lim_{N\to \infty}\lim_{\eps\to 0} \rho^{\eps,N}\qquad\mbox{and}\qquad \lim_{\eps\to 0} \lim_{N\to \infty} \rho^{\eps,N},
$$
and understand under which regimes they coincide or differ.  For the rest of this section we introduce the relevant notions that will play an important role in understanding these limits and present our main results. The result concerning the limit $N \to \infty$ followed by $\eps\to 0$ can be found in Theorem~\ref{thm:variabledata}, while the result concerning the limit $\eps \to 0$ followed by $N \to \infty$ can be found in Theorem~\ref{thm: N then eps}. We discuss the effect of the presence of a phase transition in Section.~\ref{S:explicit}. Finally, in Section~\ref{ftheorem} we discuss the implications of a CLT on the rate of convergence of the Gibbs measure before the phase transition. The proofs of the two main results, Theorems~\ref{thm:variabledata} and~\ref{thm: N then eps}, can be found in Sections~\ref{sec:1thm} and~\ref{sec:2thm}, respectively. The proofs of other useful results related to the phenomenon of phase transitions are relegated to Section~\ref{sec:phase}. Appendix~\ref{ap:coupling}
contains some coupling arguments which are useful for the proof of Theorem~\ref{thm:variabledata}. 

\vspace{0.2cm}

\subsection{The space $\mathcal{P}(\mathcal{P}(\R^d))$ as the limit of $\mathcal{P}_{\sym}\big((\R^d)^N\big)$}

The set up we consider is similar to that in~\cite{carrillo2019proof}. We remark that due to the indistinguishability assumption on the particles their joint law is invariant under relabelling of the particles. In probability this is known as exchangeability, while in analysis this is referred to as symmetry and we denote the set of symmetric probability measures  by $\mathcal{P}_{\sym}\big((\R^d)^N\big)$, i.e.
\begin{align}
\cP_{\sym}((\R^d)^N):= \set*{\rho^N \in \cP((\R^d)^N): \rho^N(A)= \rho^N(\pi(A)), \forall \pi \in\Pi, A \textrm{ measurable}} \, ,
\label{symmetric measures}
\end{align}
where $A$ is any Borel set and $\Pi$ is the set of permutations of the particle positions. Central to our work will be the classical result attributed to de Finetti \cite{deFinetti} and Hewitt--Savage \cite{HewittSavage}, that characterises the limit $N\to\infty$ of $\mathcal{P}_{\sym}\big((\R^d)^N\big)$. Adapted to the set up of this paper, their result can be reformulated as follows:

\begin{definition}\label{def:1}
Given a family $\{\rho^{N}\}_{N\in\N}$ such that $\rho^{N}\in \mathcal{P}_{\sym}((\R^d)^N)$ we say that
$$
\rho^N\to X\in\mathcal{P}(\mathcal{P}(\R^d)) \, , \quad\mbox{as $N \to \infty$} \, ,
$$
if for every $n\in \N$ we have
$$
\rho^N_n\rightharpoonup^* X^n\in \mathcal{P}_{\sym}\big((\R^d)^n\big) \, , \quad\mbox{as $N \to \infty$} \, ,
$$
where $X^n \in \mathcal{P}_{\sym}\big((\R^d)^n\big)$ is defined by duality as follows
$$
\skp*{X^n,\varphi}=\int_{\mathcal{P}(\mathcal{P}(\R^d))}\int \varphi \dx{\rho^{\otimes n}}\;\dx{X}(\rho)  \, ,
$$
for all $\varphi \in C_b((\R^d)^n)$ and
$$
\rho^N_n=\int_{(\R^d)^{N-n}} \rho^N\; \dx{x}_{N-n+1}...dx_{N}\in \mathcal{P}_{\sym}\big((\R^d)^n\big) \, .
$$
We will often suppress the $N \to \infty$ and just write $\rho^N \to X$.
\end{definition}
In particular, we can relate this definition with the usual chaoticity assumption. We will say that $\{\rho^N\}_{N\in \N}$ is chaotic with limit $\rho\in\mathcal{P}(\R^d)$ if 
$$
\rho^N\to \delta_{\rho}\in \mathcal{P}(\mathcal{P}(\R^d)) \, ,
$$
in the sense of Definition~\ref{def:1}.  Additionally, the notion of convergence introduced in Definition~\ref{def:1} can also be interpreted in the following manner:
\begin{definition}[Empirical measure]\label{def:empirical}
Given some $\rho^{N}\in \mathcal{P}_{\sym}((\R^d)^N)$ we define its empirical measure
$\hat{\rho}^{N} \in \cP(\cP(\R^d))$ as follows:
\begin{align}
\hat{\rho}^N := T_N \# \rho^N \, ,
\end{align} 
where $T^N: (\R^d)^N \to \cP(\R^d)$ is the measurable mapping $(x_1, \dots ,x_N)\mapsto N^{-1}\sum_{i=1}^N\delta_{x_i}$.
Furthermore, given a family $\{\rho^{N}\}_{N\in\N}$, we have that $\rho^N \to X \in \cP(\cP(\R^d))$ if and only if $\hat{\rho}^N \rightharpoonup^* X$.
\end{definition}
We conclude this subsection with the following compactness result:
\begin{lemma}[de Finneti--Hewitt--Savage]
Given a sequence $\{\rho^N\}_{N\in\N},$ with $\rho^N\in\mathcal{P}_{\sym}\big((\R^d)^N\big)$ for every $N$, assume that the sequence of the first marginals $\{\rho_1^N\}_{N\in\N}\in\mathcal{P}(\R^d)$ is tight.  Then, up to a subsequence, not relabelled, there exists  $X\in\mathcal{P}(\mathcal{P}(\R^d))$ such that $\rho^N\to X$ in the sense of Definition~\ref{def:1}.
\end{lemma}
For a proof and more details, see \cite{hauray2014kac,rougerie2015finetti,carrillo2019proof}. In the sequel, any limit of a sequence  of symmetric measures $\set{\rho^N}_{N \in \N}$ with $\rho^N \in \cP_{\sym}((\R^d)^N)$ should be understood in the sense of Definition~\ref{def:1}.
\begin{remark}
The above notion of convergence, i.e. Definitions~\ref{def:1} and~\ref{def:empirical}, can be naturally extended to $\cP_{\sym}((\T^d)^N)$. 
\end{remark}
\subsection{Gradient flow formulation and the mean field limit}  \label{gradientflow}

In \cite{carrillo2019proof}, the mean field limit (the limit $N\to\infty$) of the interacting particle system \eqref{eq:LinearKolmogorov} is achieved by passing to the limit in the 2-Wasserstein gradient flow structure. The results of this article will build on this perspective which we briefly recall here:

The evolution of the joint law $\rho^{\eps,N}$ given by \eqref{eq:LinearKolmogorov} is the gradient flow (in the sense of~\cite[Definition 11.1.1]{ambrosio2008gradient}) of the energy $E^N: \cP_{\sym}((\R^d)^N) \to (-\infty,+\infty]$
\begin{equation}\label{Energyperparticle}
E^N[\rho^N]:=\frac{1}{N}\left(\beta^{-1}\int_{(\R^d)^N}\rho^N\log\rho^N\;\dx{x}+\int_{(\R^d)^N}H^N_\eps(x)\;\dx{\rho^N}(x) \right)    \, ,
\end{equation}
under the rescaled 2-Wasserstein distance $\frac{1}{\sqrt{N}}d_2$ on $\mathcal{P}_{\sym}\big((\R^d)^N\big)$. Moreover, we have the following  classical result of Messer and Spohn \cite{messer1982statistical}:

\begin{lemma}\label{lem:gamma}
 The $N$-particle free energy $E^N$ $\Gamma$-converges to $E^\infty:\mathcal{P}(\mathcal{P}(\R^d))\to (-\infty,+\infty]$, where
\begin{equation}\label{}
    E^\infty[X]=\int_{\mathcal{P}(\R^d)}E_{MF}[\rho]\;\dx{X}(\rho) \, ,
\end{equation}
with $E_{MF}:\mathcal{P}(\R^d)\to (-\infty,+\infty]$ given by
\begin{equation}
    E_{MF}[\rho]=\beta^{-1} \int_{\R^d} \rho\log(\rho)\;\dx{x}+\int_{\R^d} V(\eps^{-1}x)\dx{\rho(x)}+\frac{1}{2}\iint_{\R^d \times \R^d}W(\eps^{-1}(x-y))\dx{\rho}(y)\dx{\rho}(x).
\end{equation}
That is to say, for every $X \in \cP(\cP(\R^d))$ there exists a sequence $\set*{\rho^N}_{N \in \N} $, $\rho^N \in \cP_{\sym}((\R^d)^N)$ with $\rho^N \to X$  such that
\begin{align}
\lim_{N \to \infty} E^N[\rho^N]= E^\infty[X] \, .
\end{align}
Additionally, for every $X \in \cP(\cP(\R^d))$ and $\set*{\rho^N}_{N \in \N}$, $\rho^N \in \cP_{\sym}((\R^d)^N)$ with $\rho^N \to X$
it holds that
\begin{align}
E^\infty[X] \leq \liminf_{N \to \infty}E^N[\rho^N] \, .
\end{align}
\end{lemma}
On the other hand we have a similar convergence for the metrics: $\frac{1}{\sqrt{N}}d_2$ the rescaled 2-Wasserstein distance on $\mathcal{P}_{\sym}\big((\R^d)^N\big)$ converges to $\mathfrak{D}_2$ the 2-Wasserstein distance on $\mathcal{P}(\mathcal{P}(\R^d))$. Specifically, given two sequences $\{\mu^N\}_{N\in\N}$ and $\{\nu^N\}_{N\in\N}$ of symmetric probability measures such that $\mu^N\to X_1$ and $\nu^N\to X_2$, then
$$
\frac{1}{N}d_2^2(\mu^N, \nu^N)\to \mathfrak{D}_2^2(X_1,X_2).
$$
We can now state our result concerning the mean field limit, i.e. the limit $N \to \infty$:
\begin{customthm}{A.}[Mean field limit]\label{thma}
Fix some $t>0$, then,
$$
    \lim_{N\to\infty}\rho^{\eps, N}(t)=X^\eps(t) \in \mathcal{P}(\mathcal{P}(\R^d)),
$$
Furthermore, we have that the curve $X^\eps: [0,\infty) \to \cP(\cP(\R^d))$ is a gradient flow of $E^\infty$ under the 2-Wasserstein metric $\mathfrak{D}_2$. Moreover,
\begin{align}\label{mfl}
X^\eps(t)=S_t^\eps\# X^\eps_0,
\end{align}
where $X^\eps_0=\lim_{N\to\infty}\rho^{\eps,N}_0$ and $S_t^\eps:\mathcal{P}(\R^d)\to \mathcal{P}(\R^d)$ is the solution semigroup associated to the nonlinear McKean--Vlasov evolution equation 
\begin{equation}\label{nonlineareq}
    \partial_t \rho^\eps=\beta^{-1} \Delta\rho^\eps +\nabla\cdot(\rho^\eps (\nabla W_\eps\ast\rho^\eps+\nabla V_\eps)),
\end{equation}
with $W_\eps(x)=W(\eps^{-1}x)$ and $V_\eps(x)=V(\eps^{-1}x)$.  
\end{customthm}

\vspace{0.2cm}

\subsection{Scaling and the quotiented process}\label{sandq}
We notice that the Fokker--Planck equation~\eqref{eq:LinearKolmogorov} behaves well under the parabolic scaling, i.e. given a 
solution $\rho^{\eps,N}$ of~\eqref{eq:LinearKolmogorov} we have that
\begin{align}
\nu^{N}(s,y)= \eps^{N d}\rho^{\eps,N}(\eps^{2}s,\eps y) \, ,\label{unscaling}
\end{align}
is the solution to the Fokker--Planck equation at scale $\eps=1$, i.e.
\begin{align}\label{realeps=1}
\partial_s \nu^N= \beta^{-1}\Delta \nu^N + \nabla \cdot \bra*{\nabla H^N_1 \nu^{N}} , (s,y) \in (0,\infty) \times (\R^d)^N.
\end{align}
The above equation naturally describes the evolution of the law of $N$-particle system~\eqref{interacting SDEs} at scale $\eps=1$:
\begin{equation}\label{unscaled interacting SDEs}
\begin{cases}
 \dx{X}_t^{i}= -\nabla V( X_t^{i}) - \frac{1}{N} \sum_{j \neq i}^N  \nabla W(X_t^{i}-X_t^{j}) \dx{t} + \sqrt{2 \beta^{-1}}d B_t^i\\
 \mathrm{Law}\; \big(X_0^{1},...,X_0^{N}\big)=\eps^{Nd}\rho^{\eps,N}_0(\eps x):=\nu^N_0.
\end{cases}
\end{equation}
Since $W$ and $V$ are periodic, in order to that to understand the behaviour of $\rho^{\eps,N}$ in the limit as $\eps \to 0$, we must first understand the behaviour of the quotiented process $\set{\dot{X}_t^i}_{i=1}^N$ of~\eqref{unscaled interacting SDEs} which lives on $(\T^{d})^N$~\cite[Section 9.1]{KLO12}\cite[Section 3.3.2]{BLP11}. Before we introduce the quotiented process, we define the following notion which will play an important role in the rest of the paper:
\begin{definition}\label{rearrangement}
Given a measure $\rho \in \cP(\R^d)$ we define its periodic rearrangement at scale $\eps>0$ to be the measure $\tilde{\rho} \in \cP(\T^d)$, such that for any measurable $A\subset \T^d$ it holds that
\begin{align}
\tilde{\rho}(A):= \eps^d\sum_{k \in \Z^d}\rho(\eps(A +k)) \, .
\end{align}  
We will often just use the words periodic rearrangement when $\eps=1$.
\end{definition}
Given the above notion, we have that quotient process $\set{\dot{X}_t^i}_{i=1}^N$ satisfies the following set of SDEs posed on the torus:
\begin{equation}\label{periodic interacting SDEs}
\begin{cases}
 \dx{\dot{X}}_t^{i}= -\nabla V( \dot{X}_t^{i}) - \frac{1}{N} \sum_{j \neq i}^N  \nabla W(\dot{X}_t^{i}-\dot{X}_t^{j}) \dx{t} + \sqrt{2 \beta^{-1}}d \dot{B}_t^i\\
 \mathrm{Law}\; \big(\dot{X}_0^{1},...,\dot{X}_0^{N}\big)=\tilde{\nu}^N_0 \, ,
\end{cases}
\end{equation}
where $\dot{B}^i_t$ are independent $\T^d$-valued Brownian motions and $\tilde{\nu}^N_0$ is the periodic rearrangement of $\nu^N_0$ in the sense of Definition~\ref{rearrangement}. One can check that the process $\set{\dot{X}_t^i}_{i=1}^N$ is a reversible
ergodic diffusion process with its unique invariant or Gibbs measure $M_N \in \cP_{\sym}((\T^d)^N)$ given by
\begin{align}
M_N(x)=\frac{e^{-H_1^N(x)}}{\int_{\bra*{\mathbb{T}^d}^N}e^{-H_1^N(y)}\;\dx{y}} \, .
\end{align}
As expected, the law $\tilde{\nu}^N(t)$ of the quotiented system~\eqref{periodic interacting SDEs} can be obtained by considering the periodic rearrangement of $\nu^N(t)$, the solution of~\eqref{realeps=1} and it evolves according to the following PDE:
\begin{align}\label{periodicNeps=1}
\partial_s \tilde{\nu}^N= \beta^{-1}\Delta \tilde{\nu}^N + \nabla \cdot \bra*{\nabla H^N_1 \tilde{\nu}^N} , \quad (s,y) \in (0,\infty) \times (\R^d)^N  \, .
\end{align}
In analogy to the discussion in Section~\ref{gradientflow}, the above PDE is the gradient flow of the following $N$-particle periodic free energy:
\begin{equation}\label{Energyperparticleperiodic}
\tilde{E}^N[\tilde{\nu}^N]:=\frac{1}{N}\left(\beta^{-1}\int_{(\T^d)^N}\tilde{\nu}^N\log\tilde{\nu}^N\;\dx{x}+\int_{(\T^d)^N}H^N_1(x)\;\dx{\tilde{\nu}^N}(x) \right)  \, ,
\end{equation}
under the rescaled $2$-Wasserstein distance $\frac{1}{\sqrt{N}}d_2$ on $\cP_{\sym}((\T^d)^N)$. Furthermore, the Gibbs measure $M_N$ of the process~\eqref{periodic interacting SDEs} is the unique minimiser of $\tilde{E}^N$.

Similarly, we also notice that the nonlinear McKean--Vlasov equation \eqref{nonlineareq} behaves well under the parabolic scaling. Specifically, given $\rho^\eps$ a solution to \eqref{nonlineareq}, then
$$
\nu(s,y)=\eps^d\rho^\eps(\eps^{2}s,\eps y)
$$
is a solution to the McKean--Vlasov equation at scale $\eps=1$,
\begin{equation}\label{eps=1}
 \partial_s \nu=\beta^{-1} \Delta\nu +\nabla\cdot(\nu (\nabla W\ast\nu+\nabla V))\qquad\mbox{on $(0,\infty)\times\R^d.$}   
\end{equation}
It is well known that this describes the law of the corresponding mean field McKean SDE which is given by
\begin{align}
\begin{cases} \label{eq:mckeanSDEintro}
\dx{Y}_t^\eps &= -\nabla V(Y_t^\eps)\dx{t} - \nabla (W \ast \nu (t))(Y_t^\eps)\dx{t} + \sqrt{2 \beta^{-1}} dB_t \\
\mathrm{Law}(Y_0^\eps)&= \nu_0^\eps=\eps^d\rho_0^\eps(\eps x) \in \cP(\R^d) \, .
\end{cases}
\end{align}
Again, we notice that all the coefficients in~\eqref{eps=1} are 1-periodic. Therefore, the nonlinearity $\nabla W\ast\nu$ only depends on the law of the quotiented process. We can thus understand the behaviour of the nonlinearity $\nabla W \ast \nu$ by considering the evolution  of the periodic rearrangement $\tilde{\nu}(t)$ of $\nu(t)$, which solves the periodic nonlinear McKean--Vlasov equation:
\begin{equation}\label{periodiceps=1}
 \partial_s \tilde{\nu}=\beta^{-1} \Delta\tilde{\nu} +\nabla\cdot(\tilde{\nu} (\nabla W\ast\tilde{\nu}+\nabla V))  \qquad\mbox{on $(0,\infty)\times\T^d.$}
\end{equation}
An important role is thus played by the limiting behaviour of solutions $\tilde{\nu}(t)$ of the above equation and its steady states. As in Section~\ref{gradientflow}, the equation~\eqref{periodiceps=1} is the gradient flow of the periodic mean field free energy
\begin{equation}\label{periodicmeanfieldenergy}
     \tilde{E}_{MF}[\tilde{\nu}]=\beta^{-1}\int_{\T^d} \tilde{\nu}(x)\log(\tilde{\nu}(x))\;\dx{x}+\int_{\T^d} V(x)\;\dx{\tilde{\nu}}(x)+\frac{1}{2}\int_{\T^d}W\ast\tilde{\nu}(x)\;d{\tilde{\nu}(x)} \, ,
\end{equation}
with respect to the the $2$-Wasserstein metric on $\T^d$ and the energies $\tilde{E}^N$ and $\tilde{E}^{MF}$ are related in the same way as the energies $E^N$ and $E_{MF}$, i.e. through the result of Messer and Spohn~\cite{messer1982statistical}:
\begin{lemma}\label{lem:gammaperiodic}
 The $N$-particle periodic free energy $\tilde{E}^N$ $\Gamma$-converges (in the sense of Lemma~\ref{lem:gamma}) to $\tilde{E}^\infty:\mathcal{P}(\mathcal{P}(\T^d))\to (-\infty,+\infty]$, where
\begin{equation}\label{}
    \tilde{E}^\infty[X]=\int_{\mathcal{P}(\T^d)}\tilde{E}_{MF}[\tilde{\nu}]\;\dx{X}(\tilde{\nu}) \, .
\end{equation}
As a consequence, if $\{M_N\}_{N\in\N}$ is the sequence of minimisers of $\tilde{E}^N$, then any accumulation point $X\in\mathcal{P}(\mathcal{P}(\T^d))$ of this sequence is a minimiser of $\tilde{E}^\infty$.
\end{lemma}

We can use the gradient flow structure to provide a useful characterisation of the steady states of the periodic McKean--Vlasov system~\eqref{periodiceps=1}.
\begin{proposition}\label{tfae}
Let $\tilde{\nu} \in \cP(\T^d)$. Then, the following statements are equivalent:
\begin{enumerate}
\item $\tilde{\nu}$ is a steady state of~\eqref{periodiceps=1}.
\item $\tilde{\nu}$ is a critical point of the mean field free energy, $\tilde{E}_{MF}$, i.e. the metric slope (cf.~\cite[Definition 1.2.4]{ambrosio2008gradient}) $\abs*{\partial \tilde{E}_{MF}}(\tilde{\nu})=0$.
\item $\tilde{\nu}$ is a zero of the dissipation functional $D: \cP(\T^d) \to (-\infty,+\infty]$, i.e.
\begin{align}
D(\tilde{\nu})= \intt{\abs*{\nabla\log\frac{\tilde{\nu}}{e^{-\beta\bra*{ W \ast \nu + V}}}}^2\tilde{\nu}}=0 \,.
\end{align}
\item $\tilde{\nu}$ satisfies the self-consistency equation
\begin{align}\label{eq:critical point}
    \tilde{\nu} =\frac{e^{-\beta (V + W\ast\tilde{\nu})}}{Z},     
\end{align}
with the partition function given by
\begin{align}\label{partitionfunction}
    Z=\int_{\T^d}e^{-\beta (V + W\ast\tilde{\nu}(y))}\dx{y}.
\end{align}
\end{enumerate}
\end{proposition}
A proof of this result can be found, for example, in~\cite[Proposition 2.4]{CGPS19} or in~\cite{Tam84}. It is evident from this characterisation that the behaviour of the system~\eqref{periodiceps=1} on the torus will affect the distinguished limits (either $N \to \infty$ or $\eps \to 0$) of the system~\eqref{nonlineareq} on $\R^d$. In particular, if~\eqref{periodiceps=1} has multiple steady states then the distinguished limits will be influenced by steady states attained in the long-time dynamics. We refer to the phenomenon of nonuniqueness of steady states as a phase transition and discuss its effect on the limits in Section~\ref{S:explicit}. 

To conclude this subsection, for the reader's convenience, we include Figure~\ref{fig:notation} which provides a useful schematic of the notation that will be used for the rest of this paper. Starting with $\rho^{\eps,N}$ the solution of~\eqref{eq:LinearKolmogorov}, one can obtain $\tilde{\nu}^N$, the solution of~\eqref{realeps=1}, by using the scaling in~\eqref{unscaling}. One can then pass to to the limit $N \to \infty $ in $\rho^{\eps,N}$ and $\nu^N$, to obtain the McKean--Vlasov equation at scale $\eps$~\eqref{nonlineareq} or scale $1$~\eqref{eps=1}, respectively. Alternatively one can consider the periodic rearrangement $\tilde{\nu}^N$ of $\nu^N$ which solves~\eqref{periodicNeps=1} and pass to the limit $N \to \infty$ to obtain a solution of the periodic McKean--Vlasov equation~\eqref{periodiceps=1}. The rest of the figure follows in a similar fashion. 

\begin{figure}
\begin{center}
$$
\begin{tikzcd}
\node (X1) at (0,0) {\rho^{\eps,N}};
\node (X2) at (4,0) {\nu^{N}};
\draw[>=triangle 45, bend left,->] (X1) to node [above]{\textcolor{blue}{\eps^{N d} \rho^{\eps,N}(\eps^2 s,\eps y)}} (X2);
\draw[>=triangle 45, bend left, ->] (X2) to node [below]{\textcolor{blue}{\eps^{-N d} \nu^{N}(\eps^{-2}t,\eps^{-1} x)}} (X1);
\node (X3) at (0,-4) {\rho^{\eps,\otimes N}};
\draw[>=triangle 45,  right, ->] (X1) to node [left]{\textcolor{blue}{N \to \infty}} (X3);
\node (X4) at (4,-4) {\nu^{\otimes N}};
\draw[>=triangle 45, left, ->] (X2) to node [right]{\textcolor{blue}{N \to \infty}} (X4);
\draw[>=triangle 45,  bend left, ->] (X3) to node [above]{\textcolor{blue}{\eps^{d} \rho^{\eps}(\eps^2 s,\eps y)}} (X4);
\draw[>=triangle 45,  bend left, ->] (X4) to node [below]{\textcolor{blue}{\eps^{-d} \tilde{\nu}(\eps^{-2} t,\eps^{-1} x)}} (X3);
\node (X5) at (7,0) {\tilde{\nu}^{ N}};
\draw[>=triangle 45, bend left, ->] (X2) to node [above]{\textrm{\textcolor{blue}{P.R.}}} (X5);
\node (X6) at (7,-4) {\tilde{\nu}^{\otimes N}};
\draw[>=triangle 45, left, ->] (X5) to node [right]{\textcolor{blue}{N \to \infty}} (X6);
\draw[>=triangle 45, bend right, ->] (X4) to node [below]{\textrm{\textcolor{blue}{P.R.}}} (X6);
\end{tikzcd}
$$
\end{center}
\caption{A schematic of the notation. The P.R. denotes periodic rearrangement in the sense of Definition~\ref{rearrangement}.}
\label{fig:notation}
\end{figure}
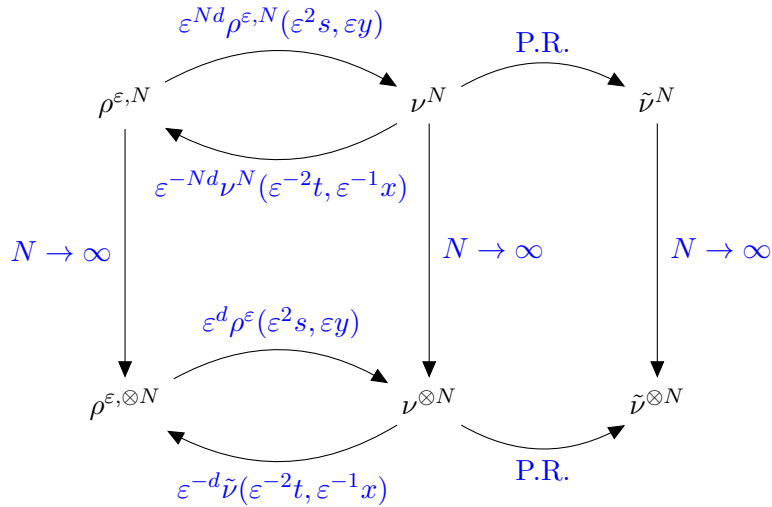
\subsection{The diffusive limit}\label{sec:diffusivelimit}

We have already discussed the limit $N \to \infty$ in Section~\ref{gradientflow}. Here, we discuss the diffusive limit, i.e. $\eps \to 0$. For a fixed number of particles $N$, we can use techniques from the theory of periodic homogenisation to pass to the limit $\eps\to0$ in \eqref{eq:LinearKolmogorov}, see for instance \cite[Chapter 20]{pavliotis2008multiscale}\cite{KV86,DFGW89,BLP11}. In particular, we have the following result:
\begin{customthm}{B}[The diffusive limit]\label{diffusivelimit}
Consider $\rho^{\eps,N}$ the solution to~\eqref{eq:LinearKolmogorov} with initial data $\rho_0^{\eps,N} \in \cP_{\sym}((\R^d)^N)$. Then, for all $t>0$ the limit
$$
\rho^{N,*}(t)=\lim_{\eps\to 0}\rho^{\eps,N}(t)
$$
exists. Furthermore, the curve of measures $\rho^{N,*}:[0,\infty) \to \cP_{\sym}((\R^d)^N)$ satisfies the heat equation
\begin{equation}\label{effectiveeq}
    \partial_t\rho^{N,*}=\nabla\cdot(A^{\eff,N} \nabla \rho^{N,*}) \, ,
\end{equation}
with initial data $\rho^{N,*}(0)=\lim_{\eps \to 0} \rho_0^{\eps,N}$ and where the covariance matrix is given by the  formula
\begin{align}
A^{\eff,N}=\beta^{-1}\int_{(\mathbb{T}^d)^N} (I+\nabla \Psi^N(y))\;M_N(y)\;\dx{y} \, ,
\label{kipnisvaradhan}
\end{align}
with 
$$
M_N(x)=\frac{e^{-H_1^N(x)}}{\int_{\bra*{\mathbb{T}^d}^N}e^{-H_1^N(y)}\;\dx{y}} \, ,
$$
the Gibbs measure of the quotiented $N$-particle system~\eqref{periodic interacting SDEs} and $\Psi^N:\big(\T^d\big)^N\to \big(\R^d\big)^N$ the unique mean zero solution to the associated corrector problem
$$
\nabla\cdot(M_N\nabla \Psi^N)=-\nabla M_N \label{Poisson} \, .
$$
Here, $H^N_1$ is the Hamiltonian of the associated particle system and is as defined in~\eqref{eq:H}.
\end{customthm}

\subsection{The limit $N \to \infty$ followed by $\eps \to 0$}\label{Nfirst}
We have discussed the mean field limit $N \to \infty$ in Section~\ref{gradientflow}. Now, we are ready to state our first result that characterises the limit $\lim_{\eps\to 0}\lim_{N\to \infty}\rho^{\eps,N}$:
\begin{theorem}\label{thm:variabledata}
Consider the set of initial data given by $\{\rho_0^\eps\}_{\eps>0}\subset\mathcal{P}(\R^d)$, and consider the periodic rearrangement at scale $\eps>0$ , i.e.
$$
\tilde{\nu}_0^\eps(A)=\eps^d \sum_{k\in \Z^d} \rho_0^\eps(\eps(A+ k))\qquad\mbox{ for $\eps>0$ }.
$$
Assume that there exists $C>0$, $p>1$ and a steady state $\tilde{\nu}^*\in\mathcal{P}(\T^d)$ such that $\tilde{\nu}^\eps(t)$, the solution to the $\eps=1$ periodic nonlinear evolution \eqref{periodiceps=1} with initial data $\tilde{\nu}_0^\eps(x)$, satisfies
\begin{equation}\label{exponential convergence}
\sup_{\eps>0}d_2^2(\tilde{\nu}^\eps (t),\tilde{\nu}^*) \le C t^{-p}\, . \tag{A1}
\end{equation} 
Then,
\begin{equation}\label{weak convergence}
\lim_{\eps\to 0}d_2^2(S_t^\eps\rho^\eps_0,S_t^* \rho^*_0)=0,    
\end{equation}
where $S^\eps_t$ is the solution semigroup associated to \eqref{nonlineareq}, $\rho^*_0\in\mathcal{P}(\R^d)$ is the weak-$*$ limit of $\rho_0^\eps$, and $S_t^*$ is the solution semigroup of the heat equation
\begin{equation}\label{limit eps to 0}
    \partial_t \rho=\nabla\cdot(A_*^{\eff}\nabla \rho), 
\end{equation}
where the covariance matrix
\begin{equation}\label{covariance matrix}
    A_*^{\eff} = \beta^{-1} \int_{\T^d} (I+ \nabla \Psi^*(y))\;\dx{\tilde{\nu}}^*(y),
\end{equation}
with $\Psi^*:\T^d\to \R^d$, $\Psi^*_i \in \SobH^1(\T^d)$ for $i=1, \dots,d$, is the unique mean zero solution to the associated corrector problem
\begin{equation}\label{limiting corrector}
    \nabla \cdot(\tilde{\nu}^* \nabla \Psi^*)=-\nabla \tilde{\nu}^*.
\end{equation}
Furthermore, assume that $X(t)^\eps$ is as defined in~\eqref{mfl} and that $\lim_{N \to \infty}\rho_0^{\eps,N}= X_0^\eps=\delta_{\rho_0^\eps}$. Then it holds that:
\begin{align}
\lim_{\eps \to 0} \lim_{N \to \infty} \rho^{\eps,N}= \lim_{\eps \to 0} X(t)^\eps= S_t^*\# X_0 \, ,
\end{align}
where $X_0=\delta_{\rho_0^*}$. 
\end{theorem}
In particular, we can apply this theorem to obtain the following result.
\begin{corollary}\label{cor:beforephasetransition}
Assume that the periodic mean field energy \eqref{periodicmeanfieldenergy} admits a unique minimiser (and hence critical point) $\tilde{\nu}^{\min}$ and that it is an exponential attractor for arbitrary initial data of the evolution of \eqref{periodiceps=1}, i.e. $d_2(\tilde{\nu}(t),\tilde{\nu}^{\min}) \leq d_2(\tilde{\nu}_0^\eps,\tilde{\nu}_{\min})e^{-Ct}$ for some fixed constant $C>0$. Then, the conclusions of Theorem~\ref{thm:variabledata} are valid for arbitrary initial data.
\end{corollary}

\begin{proof}
The proof of this follows from the fact that $d_2(\tilde{\nu}(t),\tilde{\nu}^{\min}) \leq d_2(\tilde{\nu}_0^\eps,\tilde{\nu}^{\min})e^{-Ct}$ implies that assumption~\eqref{exponential convergence} holds. 
\end{proof}

\begin{remark}[Non-chaotic initial data]
Although Theorem~\ref{thm:variabledata} requires that the initial data be chaotic, we can deal with non-chaotic initial data by tweaking assumption~\eqref{exponential convergence} to read as follows:
\begin{align}
\sup_{\eps>0} \sup\limits_{\rho \in \mathrm{supp} \; X_0^\eps} d_2(\tilde{\nu}_\rho(t), \tilde{\nu}^*) \leq Ct^{-p} \, ,
\end{align}
for $p>1$ and $C>0$ and $\tilde{\nu}_\rho(t)$ the solution of~\eqref{periodiceps=1} starting with initial data $\tilde{\nu_0}$
which is the periodic rearrangement of $\rho$.
\end{remark}
\begin{remark}
We cannot expect convergence $S_t^\eps \rho_0^\eps$ to $S_t^*\rho_0^*$ in a strong sense. By performing a formal multiscale expansion, we expect that
$$
S^\eps_t\rho^\eps_0(x)=S^*_t\rho^*_0(x) \tilde{\nu}^*(x/\eps)+\eps S^*_t(\nabla\rho^*_0)(x)\cdot\Phi(x/\eps) + \mathcal{O}(\eps^2).
$$
In particular, whenever  $\tilde{\nu}^*$ is not trivial, the leading term  $S^*_t\rho^*_0(x) \tilde{\nu}^*(x/\eps)\rightharpoonup S^*_t\rho^*_0 $ converges only weakly to its limit. 
\end{remark}

\begin{remark}\label{rem:nondegenerate}
The effective covariance matrix $A^{\eff}_*$ is strictly positive definite and we have the following bound on the ellipticity of the effective covariance matrix
$$
    \frac{\beta^{-1}}{Z_* Z^-_*}I\le A_*^{\eff}\le \beta^{-1} I,
$$
where
$$
Z_*=\int_{\T^d}e^{-\beta (V+W\ast\tilde{\nu}^*(y))}\;\dx{y}\qquad\mbox{and}\qquad Z^-_*=\int_{\T^d}e^{\beta (V+W\ast\tilde{\nu}^*(y))}\;\dx{y},
$$
see \cite[Theorem 13.12]{pavliotis2008multiscale}.
\end{remark}
\begin{remark}
If we consider rapidly varying initial data, that is to say, if there exists $\rho_{in}\in\mathcal{P}(\R^d)$ such that
$$
\rho_0^\eps(x)=\eps^{-d}\rho_{in}(\eps^{-1}x).
$$
Then, the hypothesis of Theorem~\ref{thm:variabledata} reduces to checking the speed of convergence to $\nu^*$ of the solution to \eqref{periodiceps=1} with the periodic rearrangement of $\rho_{in}$ as initial data, and $\rho^*_0=\delta_0$.

Here we can easily see how the phase transition matters for the limiting behaviour. If the evolution \eqref{periodiceps=1} admits more that one steady state $\tilde{\nu}^{*}_1$ and $\tilde{\nu}^{*}_2$, then the diffusive limit will be different if we consider $\rho_{in}=\tilde{\nu}^{*}_1$ or $\rho_{in}=\tilde{\nu}^{*}_2$, see Corollary~\ref{cor:example2} for an explicit example.
\end{remark}

\vspace{0.2cm}

\subsection{The limit $\eps \to 0$ followed by $N \to \infty$}\label{epsthenN}
 Now that we have discussed the diffusive limit in Section~\ref{sec:diffusivelimit}, we characterise the limit $N\to\infty$ of $\rho^{N,*}(t)$:
\begin{theorem}\label{thm: N then eps}
Assume that the periodic mean field energy $\tilde{E}_{MF}$ \eqref{periodicmeanfieldenergy} admits a unique minimiser $\tilde{\nu}^{\min}$, then we have that $\rho^{N,*}$ the solution of~\eqref{effectiveeq} satisfies, for any fixed $t>0$,
$$
\lim_{N\to\infty}\rho^{N,*}(t)=X(t)=S_t^{\min}\#X_0,
$$
where $X_0 \in \cP(\cP(\R^d))$ is the limit of $\rho^{N,*}(0)$ in the sense of Definition~\eqref{def:1}, and $S_t^{\min}:\mathcal{P}(\R^d)\to\mathcal{P}(\R^d)$ is the solution semigroup of the heat equation
\begin{equation}\label{limit eps to 0 bis}
    \partial_t \rho=\nabla\cdot(A_{\min}^{\eff}\nabla \rho), 
\end{equation}
where the covariance matrix
\begin{equation}\label{covariance matrix bis}
    A_{\min}^{\eff} = \beta^{-1}\int_{\T^d} (I+ \nabla \Psi^{\min}(y))\;\dx{\tilde{\nu}^{\min}}(y) \, ,
\end{equation}
with $\Psi^{\min}:\T^d\to \R^d$, $\Psi^{\min}_i \in \SobH^1\bra{\T^d}$ for $i=1, \dots, d$, the unique mean zero solution to the associated corrector problem
\begin{equation}\label{min corrector}
    \nabla \cdot(\tilde{\nu}^{\min} \nabla \Psi^{\min})=-\nabla \tilde{\nu}^{\min} \, .
\end{equation}
It follows then, that for any fixed $t>0$, the solution $\rho^{\eps,N}(t)$ of~\eqref{eq:LinearKolmogorov} satisfies
\begin{align}
\lim_{N \to \infty} \lim_{\eps \to 0} \rho^{\eps,N}(t)= \lim_{N\to \infty}\rho^{N,*}(t)=S_t^{\min}\#X_0 \, .
\end{align}
\end{theorem}
\begin{remark}\label{rmk:chaoticity}
By $\Gamma$-convergence, the assumption that the periodic mean field energy $\tilde{E}_{MF}$ defined in~\eqref{periodicmeanfieldenergy} admits a unique minimiser implies chaoticity of the Gibbs measure, that is to say $M_N\to \delta_{\tilde{\nu}^{\min}}\in \mathcal{P}(\mathcal{P}(\T^d))$, see Lemma~\ref{lem:gammaperiodic}. We note that the assumption that $\tilde{E}_{MF}$ admits a unique minimiser can be replaced by the weaker chaoticity assumption on $M_N$, i.e. $M_N\to \delta_{\tilde{\nu}^{\min}_0}$ for some specific minimiser $\tilde{\nu}^{\min}_0$.
\end{remark}

\vspace{0.2cm}

\subsection{The effect of phase transitions}\label{S:explicit}
As mentioned in Section~\ref{sandq}, we expect the presence of phase transition to affect the commutativity of the limits , especially since the results of Theorems~\ref{thm:variabledata} and~\ref{thm: N then eps} depend on the steady states of~\eqref{periodiceps=1} and the minimisers of the periodic mean field energy $\tilde{E}_{MF}$. Before proceeding any further, we define what we mean by a phase transition:
\begin{definition}[Phase transition]\label{pt}
The periodic mean field system~\eqref{periodiceps=1} is said to undergo a phase transition at some $0<\beta_c<\infty$, if:
\begin{enumerate}
\item For all $\beta<\beta_c$, there exists a unique steady state of~\eqref{periodiceps=1}.
\item For $\beta>\beta_c$, there exist at least two steady states of~\eqref{periodiceps=1}.
\end{enumerate}
The temperature $\beta_c$ is referred to as the point of phase transition or the critical temperature.
\end{definition}
The above definition would not make sense without the following result:
\begin{proposition}[Uniqueness at high temperature]\label{uniqueness}
For all $0<\beta<\infty$, the periodic mean field system~\eqref{periodiceps=1} has at least one steady state, which is a minimiser of the periodic mean field energy $\tilde{E}_{MF}$. Furthermore, for $\beta$ small enough, there exists a unique steady state $\tilde{\nu}^{\min}$ of~\eqref{periodiceps=1}, which corresponds to 
the unique minimiser of $\tilde{E}_{MF}$.
\end{proposition} 
The proof of this result follows from standard fixed point and compactness arguments and can be found in~\cite[Theorem 2.3 and Proposition 2.8]{CGPS19} or~\cite[Theorem 3]{messer1982statistical}.
\begin{remark}
The reader may have noticed that in Definition~\ref{pt} we do not discuss what happens at $\beta=\beta_c$. This is due to the fact that this depends on the nature of the phase transition, i.e. whether it is continuous or discontinuous. A detailed discussion of these phenomena and the conditions under which they arise can be found in~\cite{CP10,CGPS19}. 
\end{remark}

In the absence of a confining potential, i.e. for $V=0$, the existence and properties of phase transitions were studied in detail in~\cite{CP10,CGPS19}. It turns out that a key role in understanding this phenomenon is played by the notion of $H$-stability. We refer to an interaction potential $W$ as $H$-stable, denoted by $W \in \mathbf{H}_s$, if its Fourier coefficients are nonnegative, i.e.
\begin{align}
\hat{W}(k):= \intt{W(x) e^{i2 \pi k x}}\geq 0 \quad \forall k \in \Z^d, k \not \equiv 0 \, .
\end{align}
This notion of $H$-stability is closely related to a similar concept used in the statistical mechanics of lattice spin systems (cf.~\cite{Rue69}). Indeed, it provides us with a sharp criterion for the existence of a phase transition in the absence of the term $V$:
\begin{proposition}[Existence of phase transitions,~\cite{CGPS19,CP10}]\label{expt}
Assume $V=0$. Then the periodic mean field system~\eqref{periodiceps=1} undergoes a phase transition in the sense of Definition~\ref{pt} if and only if $W \notin \mathbf{H}_s$.
\end{proposition}
As discussed in the introduction, a prototypical example of a system that exhibits a phase transition is given by the potentials $V=0,\; W=-\cos(2 \pi x)$. The corresponding particle system is referred to the noisy Kuramoto model. The structure of phase transitions for this system is remarkably simple and is discussed in the following proposition:
\begin{proposition}\label{XY}
Consider the quotiented periodic mean field system~\eqref{periodiceps=1} with $d=1$, $W=-\cos(2 \pi x)$, and $V= 0$. Then for $\beta\leq 2$, $\tilde{\nu}_\infty \equiv 1$ is the unique minimiser and steady state of~\eqref{periodiceps=1}. For $\beta>2$, the steady states of~\eqref{periodiceps=1} are given by $\tilde{\nu}_\infty\equiv 1$ and the family of translates of the measure $\tilde{\nu}^{\min}_\beta$ which is given by the following expression:
\begin{align}
\tilde{\nu}^{\min}_\beta= Z^{-1}e^{a \cos(2 \pi x)}, \qquad Z= \intto{e^{a \cos (2 \pi x)}} \, ,
\end{align}
with $a=a(\beta)$ the solution of the following nonlinear equation $a= \beta I_1(a)/I_0(a)$, where $I_1(a),\;I_0(a)$ are the modified Bessel functions of the first and zeroth kind respectively. Moreover for $\beta>2$, $\tilde{\nu}^{\min}_\beta$ (and its translates) are the only minimisers of the periodic mean field energy $\tilde{E}_{MF}$. Thus, $\beta_c=2$ is the critical temperature of~\eqref{periodiceps=1}.
\end{proposition}
A proof of the above result can be found in~\cite[Proposition 6.1]{CGPS19}. A depiction of the bifurcation diagram of the noisy Kuramoto system can be found in Figure~\ref{fig:kurbif}.
\begin{figure}[ht]
\centering
\begin{minipage}[c]{0.45\textwidth}
\centering
    \includegraphics[width=\linewidth]{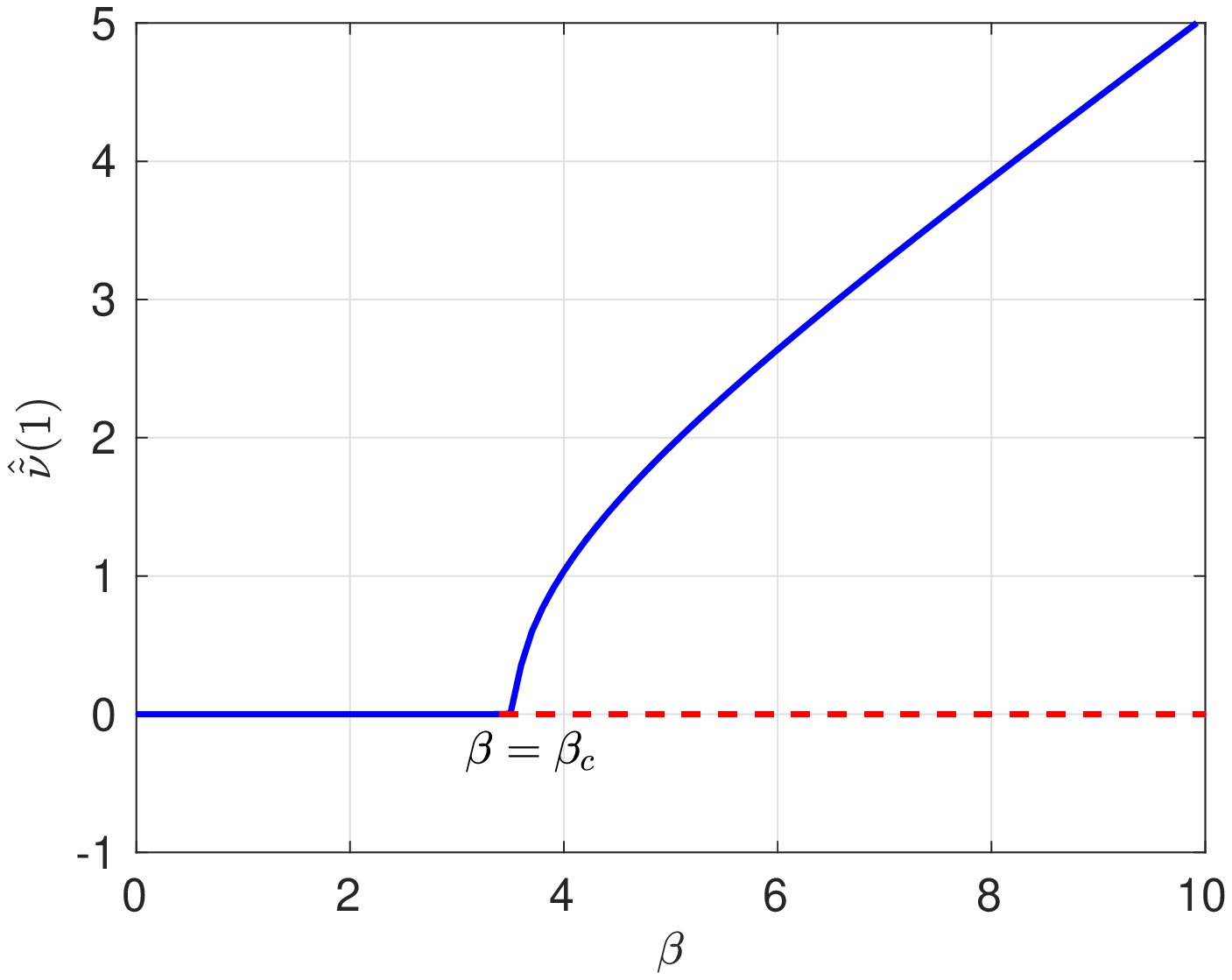}
    \caption*{(a)}
\end{minipage}
\begin{minipage}[c]{0.45\textwidth}
\centering
    \includegraphics[width=\linewidth]{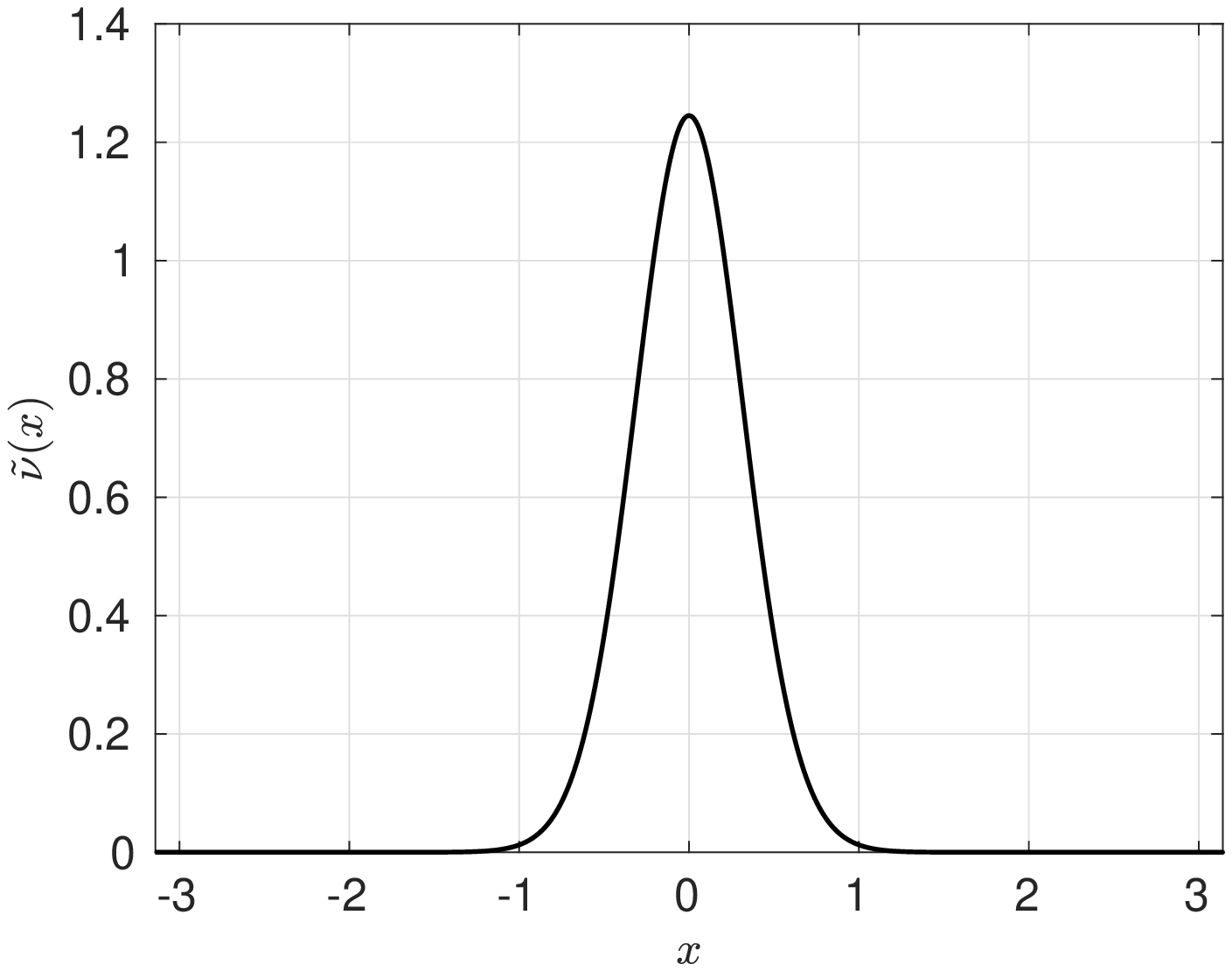}
     \caption*{(b)}
\end{minipage}
\caption{(a). The bifurcation diagram for the noisy Kuramoto system: the solid blue line denotes the stable branch
of solutions while the dotted red line denotes the unstable branch of solutions (b). An example of a clustered steady state $\tilde{\nu}_\beta^{\min}$ representing phase synchronisation of the oscillators}
\label{fig:kurbif}
\end{figure}

We can now start stating our results concerning the effect of the presence of a phase transition on the combined diffusive-mean field limit. In general, we have that for the large temperature regime the limits commute:
\begin{corollary}\label{cor:2}
Assume that $\rho_0^{\eps,N}= \bra*{\rho_0^{\eps}}^{\otimes N}$ for some $\rho_0^\eps \in \cP(\R^d)$ and that
$$
\lim_{N\to\infty}\lim_{\eps\to 0}\rho^{\eps,N}_0= \lim_{\eps\to 0} \lim_{N\to\infty} \rho^{\eps,N}_0=X_0 = \delta_{\rho_0^*}\in\mathcal{P}(\mathcal{P}(\R^d)),
$$
where $\rho_0^* \in \cP(\R^d)$ is the weak-$*$ limit of $\rho_0^\eps$. Then, there exists an explicit $\beta_0\in(0,\beta_c]$ depending on $\norm{V}_{C^2(\T^d)}$ and $\norm{W}_{C^2(\T^d)}$ such that for $\beta<\beta_0$ the limits commute:
$$
\lim_{N\to\infty}\lim_{\eps\to 0} \rho^{\eps,N}(t)=\lim_{\eps\to 0} \lim_{N\to\infty} \rho^{\eps,N}(t)= S_t^{\min}\# X_0.
$$
\end{corollary}
Moreover, for rapidly varying initial data and $V=0$, we can show that the limits commute all the way up to the phase transition. We have the following result:
\begin{corollary}\label{cor:1}
Assume $V=0$ and $\beta<\beta_c$, the critical temperature. Assume further that 
\begin{align}
\rho_0^{\eps,N}= \eps^{-d}\rho_0(\eps^{-1}x_1) \otimes \dots \otimes \eps^{-d}\rho_0(\eps^{-1}x_N) \in \cP_{\sym}\bra*{(\R^{d})^{\times N}}  
\end{align}
for some fixed $\rho_0 \in \cP(\R^d)$. Then the limits commute, i.e. 
$$
\lim_{N\to\infty}\lim_{\eps\to 0} \rho^{\eps,N}(t)=\lim_{\eps\to 0} \lim_{N\to\infty} \rho^{\eps,N}(t)= S_t^{\min}\# X_0.
$$
where $X_0= \delta_{\delta_0} \in \cP(\cP(\R^d))$. If $W$ is $H$-stable, this result holds for all $0<\beta<\infty$ and  arbitrary chaotic initial data.
\end{corollary}
The proof of Corollaries \ref{cor:2} and \ref{cor:1} can be found in Section~\ref{sec:phase}.
\begin{remark}
The results of the preceding corollaries apply to the noisy Kuramoto model.
\end{remark}

We are now ready to present our results above the critical temperature. As we are interested in illustrating our results in a clear way, we consider a simple system that undergoes a phase transition and show that the limits do not commute ahead of the phase transition. We do not consider the noisy Kuramoto model because, as demonstrated in Proposition~\ref{XY}, the minimisers of the $\tilde{E}_{MF}$ are not unique ahead of the phase transition; the entire family of translates of $\tilde{\nu}_{\beta}^{\min}$ are minimisers. Thus we cannot apply the results of Theorem~\ref{thm: N then eps} directly. Indeed, applying Lemma~\ref{lem:gammaperiodic}, one can show that the $N$-particle Gibbs measure $M_N$ converges, in the sense of Definitions~\ref{def:1} and~\ref{def:empirical}, to $X \in \cP(\cP(\T))$, where $X$ is supported uniformly on the set of translates of $\tilde{\nu}_{\beta}^{\min}$.

The alternative is to work in a quotient space as in~\cite{Rey2018,malrieu2003convergence} or to add a small confinement to break the translation invariance of the problem. We choose to do the latter. However, we do expect our results to hold true even in the translation-invariant setting but we do not deal with what we feel is essentially a technical issue in this paper. 

In particular, we consider in 1 space dimension the dynamics generated by the potentials $V=- \eta \cos(2\pi x)$ and $W=-\cos(2\pi x)$ with $0<\eta<1$. In this case, we  have the following characterisation of phase transitions:

\begin{lemma}~\label{phase transition}
Consider the quotiented periodic mean field system~\eqref{periodiceps=1} with $d=1$, $W=-\cos(2 \pi x)$, and $V= -\eta \cos(2 \pi x)$ for a fixed $\eta \in(0,1)$. Then there exists a value of the parameter $\beta=\beta_c$ such that:
\begin{itemize}
    \item For $\beta<\beta_c$, there exists a unique steady state of the quotiented periodic system \eqref{periodiceps=1} given by
\begin{align}
\tilde{\nu}^{\min}(x)= Z_{\min}^{-1}e^{a^{\min}  \cos(2 \pi x)} \label{min1}\, , \qquad Z_{\min}= \intto{e^{a^{\min} \cos(2 \pi x)}} \, ,
\end{align}
for some $a^{\min}=a^{\min}(\beta), a^{\min}>0$, which is the unique minimiser of the periodic mean field energy $\tilde{E}_{MF}$ \eqref{periodicmeanfieldenergy}. 

\item For $\beta>\beta_c$, there exist at least 2 steady states of the quotiented periodic system \eqref{periodiceps=1} given by

\begin{align}
\tilde{\nu}^{\min}(x)&= Z_{\min}^{-1}e^{a^{\min}  \cos(2 \pi x)}\, , && \qquad Z_{\min}= \intto{e^{a^{\min} \cos(2 \pi x)}} \, , \\
\tilde{\nu}^{*}(x)&= Z_*^{-1}e^{ a^{*}  \cos(2 \pi x)} \, ,&& \qquad Z_*^{-1}= \intto{e^{a^{*} \cos(2 \pi x)}} \label{min2} \, ,
\end{align}
where $a^{*} < 0< a^{\min}$ and both constants depend on $\beta$. Here $\tilde{\nu}^{\min}$ is the unique minimiser and $\tilde{\nu}^*$ is a non-minimising critical points of the periodic mean field energy $\tilde{E}_{MF}$  \eqref{periodicmeanfieldenergy}. Moreover, $a^*\ne - a^{\min}$.
\end{itemize}
\end{lemma}
The proof of Lemma~\ref{phase transition} can be found in Section~\ref{sec:phase}.

Now, we are ready to state our results in this specific case, i.e. above the phase transition we can choose specific initial data for which the limits do not commute.
\begin{corollary}\label{cor:example2}
Assume that $V=-\eta\cos(2\pi x)$, $W=\cos(2\pi x)$ for a fixed $\eta \in (0,1)$, and that we are above the phase transition $\beta>\beta_c$. As in Proposition~\ref{phase transition}, we denote by $\tilde{\nu}^{\min}$ and  $\tilde{\nu}^*$ the minimiser and the nonminimising critical point of $\tilde{E}_{MF}$. We choose the following initial data:
$$
\rho^{\eps,N}_0(x)= \big(\eps^{-1} \rho_0^*(\eps^{-1}x_1)\big)\otimes ...\otimes  \big(\eps^{-1} \rho_0^*(\eps^{-1}x_N)\big)\in \mathcal{P}_{\sym}\big((\R^d)^N\big)\, ,
$$
where $\rho_0^* \in \cP(\R)$ satisfies
\begin{align}
\tilde{\nu}^*(A)=\sum_{k\in\Z^d} \rho_0^* (A+k) \,,
\label{dataquotient}
\end{align}  
for any measurable $A$, i.e. its periodic rearrangement is $\tilde{\nu}^*$.
Then, for every $t>0$, $\rho^{\eps,N}(t)$ the solution to \eqref{eq:LinearKolmogorov}, satisfies
$$
\lim_{N\to\infty}\lim_{\eps\to 0} \rho^{\eps,N}(t)= \delta_{\rho^{\min}(t)}\in \mathcal{P}(\mathcal{P}(\R)),
$$
where
\begin{align}
\rho^{\min}(t)= \sqrt{\beta} I_0(a^{\min})\frac{e^{-\frac{\beta I_0(a^{\min})^2 |x|^2}{2t}}}{ \sqrt{2\pi t}}\,.
\label{lim1}
\end{align}
On the other hand, we have that
$$
\lim_{\eps\to0}\lim_{N\to\infty}  \rho^{\eps,N}= \delta_{\rho^{*}(t)}\in \mathcal{P}(\mathcal{P}(\R))
$$
where 
\begin{align}
\rho^{*}(t)= \sqrt{\beta} I_0(-a^{*})\frac{e^{-\beta I_0(-a^{*})^2|x|^2}{2t}}{ \sqrt{2\pi t}}\,.
\label{lim2}
\end{align}
Finally, by Lemma~\ref{phase transition} $a^*\ne -a^{\min}$, and therefore by the strict monotonicity of the modified zeroth Bessel function $I_0$ we obtain that
$$
\rho^{\min}(t)\ne \rho^*(t)\qquad\mbox{for any $t>0$.}
$$
\end{corollary}
\begin{proof}
We first note that
\begin{align}
\lim_{N \to \infty} \lim_{\eps \to 0} \rho^{\eps,N}_0 =  \lim_{\eps \to 0} \lim_{N \to \infty} \rho^{\eps,N}_0 = \delta_{\delta_0}=: X_0 \in \cP(\cP(\R)) \, . 
\end{align}
For the limit $\eps \to 0$ followed by $N \to \infty$, we use that by Proposition~\ref{phase transition} $\tilde{\nu}^{\min}$ is the unique minimiser of $\tilde{E}_{MF}$, hence we can apply Theorem~\ref{thm: N then eps} to obtain that
$$
\lim_{N\to\infty}\lim_{\eps\to 0} \rho^{\eps,N}(t)= \delta_{\rho^{\min}(t)}
$$
with $\rho^{\min}$ satisfying
$$
\begin{cases}
\partial_t \rho^{\min}=\pa_x(A^{\eff}_{\min}\pa_x \rho^{\min})\\
\rho^{\min}(0)=\delta_0,
\end{cases}
$$
where
$$
A^{\eff}_{\min}=\frac{\beta^{-1}}{Z\hat{Z}}=\frac{\beta^{-1}}{\int_{\T}e^{a^{\min}\cos(2\pi y)}\;dy\int_{\T}e^{-a^{\min}\cos(2\pi y)}\;dy}=\frac{\beta^{-1}}{I_0(a^{\min})^2}.
$$
To obtain this formula, we have used that in 1-D we can solve the corrector problem~\eqref{min corrector} explicitly, see for instance \cite[Equation (13.6.13)]{pavliotis2008multiscale}. The explicit expression for $\rho^{\min}(t)$ now follows.

Now we turn to the other limit. As discussed in Section~\ref{gradientflow}, passing to the limit $N \to \infty$, we obtain that for a fixed $t>0$
\begin{align}
\lim_{N \to \infty} \rho^{\eps,N}(t)= X^\eps(t)= S_t^\eps \# X_0^\eps \, ,
\end{align}
with $S_t^\eps$ the solution semigroup of~\eqref{nonlineareq} and $X_0^\eps= \delta_{\eps^{-1}\rho_0^*(\eps^{-1}x)}$. Using~\eqref{dataquotient}, we have  that the initial data for the $\eps=1$ periodic mean field equation~\eqref{periodiceps=1} is
given by
\begin{align}
\tilde{\nu}^\eps_0= \sum_{k\in\Z^d} \rho_0^* (x+k)= \tilde{\nu}^* \, .
\end{align}
We know from Proposition~\ref{phase transition} that $\tilde{\nu}^*$ is steady state of~\eqref{periodiceps=1}, thus the hypothesis \eqref{exponential convergence} is trivially satisfied. Therefore, we can pass to the limit as $\eps \to 0$ using Theorem~\ref{thm:variabledata} and obtain for a fixed $t>0$
\begin{align}
\lim_{\eps \to 0} X^\eps(t) = \delta_{\rho^*(t)} \, ,
\end{align}
with $\rho^{*}$ satisfying
$$
\begin{cases}
\partial_t \rho^{*}=\pa_x(A^{\eff}_{*}\pa_x \rho^{*})\\
\rho^{*}(0)=\delta_0,
\end{cases}
$$
where
$$
A^{\eff}_{*}=\frac{\beta^{-1}}{Z\hat{Z}}=\frac{\beta^{-1}}{\int_{\T}e^{a^{*}\cos(2\pi y)}\;dy\int_{\T}e^{-a^{*}\cos(2\pi y)}\;dy}=\frac{\beta^{-1}}{I_0(-a^{*})^2},
$$
thus proving~\eqref{lim2} and completing the proof of the result.
\end{proof}

\begin{figure}[ht]
\centering
    \includegraphics[scale=0.5]{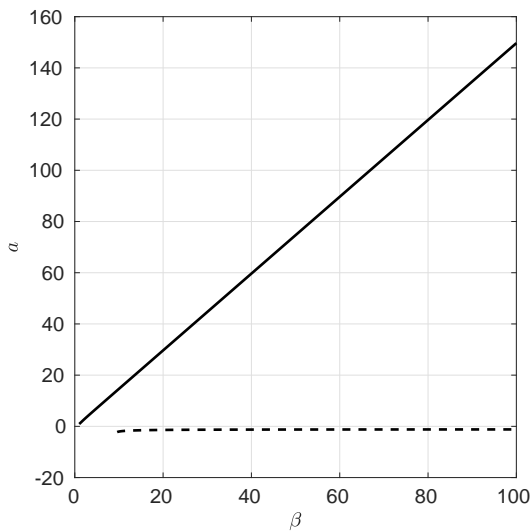}

\caption{The two solutions of Equation~\eqref{simplified}, i.e. $a^{\min}$ (the solid line) and $a^*$ (the dotted line) for $\eta=0.5$. The figure demonstrates how stark the difference between the two solutions and hence between the two effective solutions, $\rho^{\min}$ and $\rho^*$ of Corollary~\ref{cor:example2}, is.}
\label{fig:adiff}
\end{figure}

\begin{remark}
The result of Corollary~\ref{cor:example2} can be generalized to other rapidly varying initial data that is exponentially attracted to $\tilde{\nu}^{*}$.
\end{remark}
\begin{remark}
A simple choice of initial data which satisfies~\eqref{dataquotient} is $\rho_0^*= \chi_{[0,1]}\tilde{\nu}^*$, with $\chi_A$
the indicator function of the set $A$.
\end{remark}

\vspace{0.2cm}

\subsection{Application of the fluctuation theorem}\label{ftheorem}

In this subsection will assume without proof that we have a characterisation of the fluctuations around the mean field limit, in the spirit of~\cite{dawson1983critical,fernandez1997hilbertian}, as the solution to a linear SPDE  and use this together the energy minimisation property of the Gibbs measure to obtain a rate of convergence in relative entropy of the Gibbs measure  to the minimiser of the periodic mean field energy \eqref{periodiceps=1}. We also characterise the asymptotic behaviour of the partition function. At the end of the subsection we present a provisional result in which we show that this rate of convergence does hold at high temperatures without using the central limit theorem (the characterisation of fluctuations) but instead conditional on a certain rate of convergence in a weaker topology (cf.~\eqref{propchaos}).

We start by restating the classical result by Messer--Spohn~\cite{messer1982statistical} (cf. Lemma~\ref{lem:gammaperiodic}). We consider the unique minimiser of $\tilde{E}^N:\mathcal{P}_{\sym}((\T^d)^N)\to (-\infty,+\infty]$
$$
\tilde{E}^N[\tilde{\nu}^N]=\frac{1}{N}\left(\beta^{-1}\int_{(\T^d)^N}\tilde{\nu}^N(x)\log\tilde{\nu}^N(x)\;\dx{x}+\int_{(\T^d)^N}\frac{1}{2N}\sum_{i=1}^N\sum_{\substack{j=1\\j\ne i}}^N W(x_i-x_j)+\sum_{i=1}^N V(x_i)\;\dx{\tilde{\nu}^N}(x) \right)
$$
which is given by the Gibbs measure
$$
M_N(x)=\frac{e^{-\beta\left(\frac{1}{2N}\sum_{i=1}^N\sum_{j\ne i}^N W(x_i-x_j)+\sum_{i=1}^N V(x_i)\right)}}{Z_N},
$$
with the partition function
$$
Z_N=\int_{(\mathbb{T}^d)^N}e^{-\beta\left(\frac{1}{2N}\sum_{i=1}^N\sum_{j\ne i}^N W(y_i-y_j)+\sum_{i=1}^N V(y_i)\right)}\;\dx{y}.
$$
Then any accumulation point $X^\infty\in \mathcal{P}(\mathcal{P}(\T^d))$ of the sequence of minimisers $M_N$ is a minimiser of 
\begin{equation}
    \tilde{E}^\infty[X]=\int_{\mathcal{P}(\mathcal{P}(\T^d))}\tilde{E}_{MF}[\tilde{\nu}]\;\dx{X}(\tilde{\nu})
\end{equation}
with $\tilde{E}_{MF}:\mathcal{P}(\T^d)\to (-\infty,+\infty]$ given by
\begin{equation}\label{eq:MFenergy}
    \tilde{E}_{MF}[\tilde{\nu}]=\beta^{-1} \int_{\T^d} \tilde{\nu}\log(\tilde{\nu})\;\dx{x}+\frac{1}{2}\int_{\T^d}\int_{\T^d}W(x-y)\tilde{\nu} (x)\tilde{\nu}(y)\;\dx{x}\dx{y}+\int_{\T^d} V(x)\tilde{\nu}(x)\;\dx{x},
\end{equation}
which implies that
$$
\textrm{supp}\;X^\infty \subset \{\tilde{\nu}\in \mathcal{P}(\T^d)\;:\; \tilde{E}_{MF}[\tilde{\nu}]=\inf \tilde{E}_{MF} \}
$$
In particular, if we are below the phase transition $\beta<\beta_c$, we have, by Definition~\ref{pt} and Proposition~\ref{tfae}, that $\tilde{E}_{MF}$ admits a unique minimiser, which we denote by $\tilde{\nu}^{\min}\in\mathcal{P}(\T^d)$ and thus $X^\infty=\delta_{\tilde{\nu}^{\min}}\in\mathcal{P}(\mathcal{P}(\T^d))$. In the subsequent calculations, we will use $(M_N)_n$ to refer to the $n$\textsuperscript{th} marginal of the $N$-particle Gibbs measure $M_N$.

A natural next step is to consider the next order of convergence:
\begin{align}
\ds N(\tilde{E}^N(M_N)-\ds\tilde{E}_{MF}(\tilde{\nu}^{\min})) &= \beta^{-1}\int_{\T^{Nd}} M_N\log M_N \;\dx{x}-\beta^{-1}\int_{\T^d} (\tilde{\nu}^{\min})^{\otimes N}\log (\tilde{\nu}^{\min})^{\otimes N} \;\dx{x}\\
\ds\qquad&+\frac{N-1}{2}\int_{\T^{2d}}W(x-y)(M_N)_2\;\dx{x}\dx{y}+N\int_{\T^d}V(x)(M_N)_1\;\dx{x}\\
\ds\qquad& -\frac{N}{2}\int_{\T^{2d}}W(x-y)\tilde{\nu}^{\min}(x)\tilde{\nu}^{\min}(y)\;\dx{x}\dx{y}-N\int_{\T^d}V(x)\tilde{\nu}^{\min}\;\dx{x}. \label{eq:Nenergy}
\end{align}
The idea is to massage the previous expression to obtain something we can control with the fluctuations. To do this we first need to use the empirical measure $\hat{M}_N\in\mathcal{P}(\mathcal{P}(\T^d))$ associated to $M_N\in\mathcal{P}_{\sym}((\T^d)^N)$, as defined in Definition~\ref{def:empirical}.
We can compare the second marginal $(M_N)_2$ of $M_N$ with the products of the empirical measure. We notice that for any test function $\varphi\in C^\infty (\T^{2d})$, we have
\begin{equation}\label{eq:empirical}
\left(1-\frac{1}{N}\right)\int_{\T^{2d}}\varphi(x,y)(M_N)_2\;\dx{x}\dx{y}+\frac{1}{N}\int_{\T^d}\varphi(x,x)(M_N)_1\;\dx{x}= \mathbb{E} \left\langle \varphi, \left(\frac{1}{N}\sum_{i=1}^N\delta_{x_i}\right)^{\otimes 2}\right\rangle,
\end{equation}
where the expectation is taken with respect to the law $\hat{M}_N$ (for more details on these type identities for higher order marginals see \cite{diaconis1980finite}). We know from Proposition~\ref{tfae} that the minimiser of the mean field energy must satisfy the following condition
\begin{equation}\label{eq:firstorder}
\beta^{-1}\log \tilde{\nu}^{\min}=-W\ast\tilde{\nu}^{\min}-V+C.
\end{equation}
Putting \eqref{eq:Nenergy}, \eqref{eq:empirical}, and \eqref{eq:firstorder}  together, adding and subtracting
$$
\beta^{-1}\int_{\T^d}M_N \log \tilde{\nu}^{\min}\;\dx{x} \, ,
$$
and completing the square, we obtain 
\begin{equation}\label{eq:boundfrombelow}
 N(\tilde{E}^N(M_N)-\tilde{E}_{MF}(\tilde{\nu}^{\min}))=\ds \beta^{-1}\mathcal{H}(M_N|(\tilde{\nu}^{\min})^{\otimes N}) -\frac{1}{2}\mathbb{E}\left\langle W(x-y), (\mathcal{G}^N)^{\otimes 2}\right\rangle -\frac{W(0)}{2}  
\end{equation}
where $\mathcal{H}(\cdot|\cdot)$ denotes the relative entropy or Kullback--Leibler divergence and
$$
\mathcal{G}^N:= \sqrt{N}\left(\frac{1}{N} \sum_{i=1}^N\delta_{x_i}-\tilde{\nu}^{\min}\right)
$$
is a Radon measure-valued random variable defined on the probability space $((\T^d)^N,M_N)$. We refer to $\mathcal{G}^N$ as the fluctuations around the mean field limit. Using the fact that
\begin{equation}\label{eq:minimizing}
\begin{array}{rl}
\ds N\tilde{E}_{MF}(\tilde{\nu}^{\min})&\ds=N\tilde{E}^N((\tilde{\nu}^{\min})^{\otimes N} )-\frac{1}{2}\int_{\T^{2d}}W(x-y)\tilde{\nu}^{\min}(x)\tilde{\nu}^{\min}(y)\;\dx{x}\dx{y}\\
&\ds\ge N \tilde{E}^N(M_N )     -\frac{1}{2}\int_{\T^{2d}}W(x-y)\tilde{\nu}^{\min}(x)\tilde{\nu}^{\min}(y)\;\dx{x}\dx{y}
\end{array}
\end{equation}
we obtain the bound
\begin{equation}\label{eq:boundfromabove}
0\le \mathcal{H}(M_N|(\tilde{\nu}^{\min})^{\otimes N})\le \frac{W(0)}{2}+ \frac{1}{2}\mathbb{E}\left\langle W(x-y), (\mathcal{G}^N)^{\otimes 2}\right\rangle+\frac{1}{2}\int_{\T^{2d}}W(x-y)\tilde{\nu}^{\min}(x)\tilde{\nu}^{\min}(y)\;\dx{x}\dx{y}.    
\end{equation}
In a similar way, we can also obtain the bound
\begin{equation}\label{eq:partitionfunction}
\begin{array}{rl}
\ds-\frac{1}{2}\int_{\T^{2d}}W(x-y)\tilde{\nu}^{\min}(x)\tilde{\nu}^{\min}(y)\;\dx{x}\dx{y}&\ds\ge N\log\left(\frac{Z_N}{Z_{\infty}}\right)=N\left(\tilde{E}^N[M_N]- \tilde{E}_{MF}[\tilde{\nu}^{\min}]\right)\\
&\ds\ge-\frac{W(0)}{2} -\frac{1}{2}\mathbb{E}\left\langle W(x-y), (\mathcal{G}^N)^{\otimes 2}\right\rangle,
\end{array}
\end{equation}
where we have used the positivity of the relative entropy. Therefore, to obtain useful information from \eqref{eq:boundfromabove} and \eqref{eq:partitionfunction}, we need to show that
\begin{equation}
\limsup_{N\to\infty}\mathbb{E}\left\langle W(x-y), (\mathcal{G}^N)^{\otimes 2}\right\rangle<\infty.    
\end{equation}

To simplify the discussion and obtain sharp bounds all the way up to the phase transition, we consider the specific example of $d=1$, $V=0$ and $W=-\cos(2\pi x)$, which undergoes a phase transition at $\beta_c=2$ (cf. Proposition~\ref{XY}). We now make our main assumption that we have an equilibrium version of the central limit theorem before the phase transition., i.e. $\mathcal{G}^N$ converges in law to $\mathcal{G}^\infty$ whose law is the unique invariant measure of the following linear stochastic PDE
\begin{equation}\label{eq:SPDE}
\dot{\mathcal{G}}^\infty=\beta^{-1}\partial_{xx} \mathcal{G}^\infty+(2\pi)^2 \cos(2\pi x)*\mathcal{G}^\infty+\sqrt{2\beta^{-1}}\xi,    
\end{equation}
where we have used that $\tilde{\nu}^{\min}=d\mathcal{L}$ and that $W=-\cos(2\pi x)$ has zero average  to simplify the linearisation of the nonlinear PDE \eqref{nonlineareq} and $\xi$ is the space and time derivative of the cylindrical Wiener process. More specifically, if we consider $\{e_k\}_{k\in\Z}$ the standard orthonormal Fourier basis of $L^2(\T)$ given by
$$
e_k(x)=\begin{cases}
\sqrt{2}\sin(2\pi k x)&k>0\\
1&k=0\\
\sqrt{2}\cos(2\pi k x)& k<0,
\end{cases}
$$
then we can express
$$
\xi(x,t)=\sum_{k\in \Z} 2 \pi k e_k(x) \dot{B}_k(t) 
$$
where $\{\dot{B}_k\}_{k\in \Z}$ is a countable family of independent $\T$-valued Wiener processes. In particular, we can decompose \eqref{eq:SPDE} by projecting it onto each mode to obtain a family of uncoupled SDEs given by
$$
\langle \dot{\mathcal{G}}^\infty,e_k\rangle=\begin{cases}
-\beta^{-1}(2\pi)^2\langle\mathcal{G}^\infty,e_k\rangle+\frac{1}{2}(2\pi)^2 \langle\mathcal{G}^\infty,e_{k}\rangle+\sqrt{2\beta^{-1}}2\pi k \dot{B}_{k}&|k|=1\\
-\beta^{-1}(2\pi k)^2\langle\mathcal{G}^\infty,e_k\rangle+\sqrt{2\beta^{-1}}2\pi k\dot{B}_{k}&|k|\ne 1,
\end{cases}
$$
where we have used the trigonometric identity
\begin{equation}\label{eq:trigidentity}
\cos(2\pi x)*\mathcal{G}^\infty= \frac{1}{2}\left(e_1 \langle \mathcal{G}^\infty,e_1\rangle +e_{-1} \langle \mathcal{G}^\infty,e_{-1}\rangle\right).    
\end{equation}
In particular, we can find the invariant measure explicitly for each mode
\begin{equation}\label{eq:law}
\mathrm{Law}(\langle\mathcal{G}^\infty,e_k\rangle)=\begin{cases}
\mathcal{N}\left(0,\frac{2}{2-\beta}\right)&|k|=1\\
\mathcal{N}(0,1)&|k|\ne 1,
\end{cases}
\end{equation}
where $\mathcal{N}$ is the normal distribution. From \eqref{eq:law} we can clearly identify the phase transition $\beta_c=2$ when the SPDE \eqref{eq:SPDE} does no longer support an invariant measure.

Taking limits in \eqref{eq:boundfromabove} and using that $W(0)=-1$ we obtain that for this specific system we have the bound
$$
\limsup_{N\to\infty}\mathcal{H}(M_N|(\tilde{\nu}^{\min})^{\otimes N})\le  -\frac{W(0)}{2}+\mathbb{E}\frac{|\langle e_1, \mathcal{G}^\infty\rangle|^2+|\langle e_{-1}, \mathcal{G}^\infty\rangle|^2}{4}=\frac{\beta}{2(2-\beta)} ,
$$
where we have used the trigonometric identity \eqref{eq:trigidentity} and the law \eqref{eq:law} of the projections of $\mathcal{G}^\infty$. Decomposing the $M_N$ into its marginals, we can use the subadditivity of the relative entropy to conclude that
$$
\mathcal{H}((M_N)_n,(\tilde{\nu}^{\min})^n)\le\left\lfloor \frac{n}{N}\right\rfloor \frac{\beta}{2(2-\beta)},
$$
where $\left\lfloor n/N\right\rfloor $ is the largest integer less than $n/N$. We note that this estimate holds all the way up to the phase transition for this system $\beta_c=2$. Similarly, using \eqref{eq:partitionfunction} we obtain that for every $\delta>0$, we have the estimate
$$
1\ge \frac{Z_N}{Z_{\min}}\ge e^{-\frac{1}{N}\left(\frac{\beta}{2(2-\beta)}+\delta\right)}\stackrel{N\to\infty}{\to} 1
$$
for $N$ large enough. To conclude this subsection, we rewrite these bounds into a general provisional theorem (cf. Remark~\ref{provisional}).
\begin{theorem}
Consider $\tilde{\nu}^{\min}$, the unique minimiser of the periodic mean field energy $\tilde{E}_{MF}$~\eqref{eq:MFenergy} and $Z_{\min}$, its associated partition function. Assume that there exists a constant $C_1>0$, such that $\norm{W}_{C^2(\T^d)}\le C_1$ and that for $N$ large enough we have the estimate
\begin{equation}\label{propchaos}
\mathfrak{D}_2^2(\hat{M}_N,\delta_{\nu^{\min}})\le \frac{C_1}{N} \,.
\end{equation}
Then, there exists $C>0$ such that the following estimates hold
$$
\mathcal{H}\left((M_N)_n,(\tilde{\nu}^{\min})^{\otimes n}\right)\le C \left\lfloor \frac{n}{N}\right\rfloor 
$$
and
$$
\frac{Z_N}{Z_{\min}}\ge e^{-\frac{C}{N}}.
$$
\end{theorem}
\begin{remark}\label{provisional}
The bound \eqref{propchaos} has been shown in the case of convex potentials in \cite{malrieu2003convergence}. In the large temperature regime $\beta\gg 1$, the slightly weaker bound
$$
\mathfrak{D}_1^2(\hat{M}_N,\delta_{\nu^{\min}})\le \frac{C_1}{N},
$$
with $\mathfrak{D}_1$ the 1-Wasserstein distance, can be found in \cite{durmus2018elementary} by employing the coupling technique outlined in Appendix~\ref{ap:coupling}; trying to upgrade this bound to \eqref{propchaos} is an interesting open question. Note that if the formal central limit theorem discussed at the start of the subsection (in the spirit of~\cite{fernandez1997hilbertian}) could be proved rigorously then~\eqref{propchaos} would hold. For the case $\beta=\beta_c$, we can not expect \eqref{propchaos} to hold (cf. \cite{dawson1983critical}).
\end{remark}
\begin{proof}
By \eqref{eq:boundfromabove} and \eqref{partitionfunction}, we need to show that there exists a $C$ depending on $\beta$, $V$, and $W$ such that
$$
\limsup_{N\to\infty} \mathbb{E}\langle W(x-y),(\mathcal{G}^N)^{\otimes 2}\rangle\le C.  
$$
By using the dual formulation of the Wasserstein distance and using the definition of $\mathcal{G}^N$, we obtain the following estimate:
\begin{align}
|\mathbb{E}\langle W(x-y),(\mathcal{G}^N)^{\otimes 2}\rangle| & \leq N\|D^2W\|_{\Leb^\infty(\T^d)}\mathbb{E} \left(d^2_2\left(\frac{1}{N}\sum_{i=1}^N\delta_{x_i},\tilde{\nu}^{\min}\right)\right)\\&=N\|D^2W\|_{\Leb^\infty(\T^d)}\mathfrak{D}_2^2(\hat{M}_N,\delta_{\tilde{\nu}^{\min}}),
\end{align}
where the expectation is taken with respect to the empirical measure $\hat{M}_N$ and $\mathfrak{D}_2$ is the 2-Wasserstein distance on the metric space $(\mathcal{P}(\T^d),d_2)$. The result now follows by applying hypothesis \eqref{propchaos}. 

\end{proof}




\section{Proof of Theorem~\ref{thm:variabledata}}\label{sec:1thm}

We start the proof of Theorem~\ref{thm:variabledata} with some basic elliptic estimates on a time-dependent corrector problem.
\begin{lemma}\label{lem:exreg}
Consider the following elliptic equations
\begin{align}
\nabla \cdot \bra*{\tilde{\mu}^\eps \nabla \chi}= - \nabla \tilde{\mu}^\eps \qquad \mbox{on $\T^d\times[0,\infty)$},
\label{eq:tdcorrector}
\end{align}
where 
\begin{align}
\tilde{\mu}^\eps(x,t)=Z^{-1} \exp\bra*{-\beta( W \ast \tilde{\nu}^\eps +V)}\, , \qquad Z=\intt{\exp\bra*{-\beta( W \ast \tilde{\nu}^\eps +V)}} \, ,
\end{align}
where $\tilde{\nu}^\eps(x,t)$ is a solution to the evolution \eqref{periodiceps=1} with initial data $\tilde{\nu}_0^\eps$. Then, there exists a unique (up to an additive constant) smooth solution $\chi:\T^d \times[0,\infty) \to \R^d$, $\chi_i \in \SobH^1(\T^d)$ to~\eqref{eq:tdcorrector}. Additionally, it satisfies the following estimates
\begin{align}
\norm{\chi_i}_{\SobH^m (\T^d)} &\leq C_1  \label{eq:Hmbound}  \\
\norm{\partial_t \chi_i(t)}_{\SobH^m(\T^d)} &\leq \sum_{k=1}^m c_k \|\partial_t\tilde{\nu}^\eps\|_{C^{-3}(\T^d)}^k(t) \label{eq:Hmboundt}
\end{align} 
for all $i = 1, \dots, d$ and $t>0$, where $C^{-3}(\T^d)$ is the dual of $C^3(\T^d)$, and the constants $C_1\;,c_k>0$ depend only on $m$, $d$, $\norm{V}_{C^m(\T^d)}$, and $\norm{W}_{C^m(\T^d)}$.
\end{lemma}
\begin{proof}
\item
\paragraph{\bf Existence and uniqueness}
We consider the equation component-wise for any $i= 1 \dots ,d$:
\begin{align}
\nabla \cdot \bra*{\tilde{\mu}^\eps \nabla \chi_i}= - \partial_{x_i} \tilde{\mu}^\eps \, .
\label{eq:tdcorrectori}
\end{align}
Note that $\tilde{\mu}^\eps$ is smooth and is bounded above and below uniformly in time: 
\begin{align}
e^{-2 \beta \bra*{\norm{W}_{\Leb^\infty(\T^d)}+ \norm{V}_{\Leb^\infty(\T^d)} }} \leq \mu^{\eps}(x,t) \leq e^{2 \beta \bra*{\norm{W}_{\Leb^\infty(\T^d)}+ \norm{V}_{\Leb^\infty(\T^d)} }} \, .
\label{eq:mubounds}
\end{align}
Thus, by standard elliptic theory, for each $t \geq 0$ and $i = 1, \dots, d$, there exists a unique smooth solution $\chi_i \in \SobH^1_0(\T^d)$ to~\eqref{eq:tdcorrectori}. We can check that $\chi_i$ is continuously differentiable in time, as $\xi_i=\partial_t \chi_i$ satisfies
\begin{align}
\nabla \cdot \bra*{\tilde{\mu}^\eps \nabla \xi_i} = - \partial_{x_i} \partial_t \tilde{\mu}^\eps -\nabla \cdot \bra*{\partial_t\tilde{\mu}^\eps \nabla \chi_i} \, .
\label{eq:tdcorrectorit}
\end{align} 
Similar arguments imply that there exists a unique smooth solution of the above equation $\xi_i \in \SobH^1_0(\T^d)$. 
\item
\paragraph{\bf Regularity} We note that is it is sufficient to prove the bounds~\eqref{eq:Hmbound} and~\eqref{eq:Hmboundt} in the weighted space $\SobH^m(\tilde{\mu}^\eps)$ since by~\eqref{eq:mubounds} these norms are equivalent to the flat space up to a time-independent multiplicative constant. We deal first with the regularity of~\eqref{eq:tdcorrectori}. Testing against
$\chi_i$ we obtain,
\begin{align}
\intt{\abs{\nabla\chi_i}^2 \tilde{\mu}^\eps}&= -\intt{\partial_{x_i} \chi_i \tilde{\mu}^\eps} \\
& \leq \norm{\partial_{x_i} \chi_i}_{\Leb^2(\tilde{\mu}^\eps)} \leq \norm{\nabla \chi_i}_{\Leb^2(\tilde{\mu}^\eps)} \, .
\end{align} 
It follows then that
\begin{align}
\norm{\chi_i}_{\SobH^1_0(\tilde{\mu}^\eps)} \leq 1 \,.
\label{eq:H1bound}
\end{align}
Now let $\alpha \in \N^d$ be any multi-index of order $m-1$ for some $m \geq 1$. Testing~\eqref{eq:tdcorrectori} against $\partial_{2 \alpha} \chi_i$, we obtain for the left hand side
\begin{align}
\intt{(\partial_{2 \alpha} \chi_{i}) \nabla \cdot \bra*{ {\tilde{\mu}^\eps \nabla \chi_{i}}}}&=
(-1)^{m}\intt{\bra*{\nabla \partial_{\alpha} \chi_{i}} \cdot \partial_{\alpha} (\tilde{\mu}^\eps \nabla \chi_{i})} \\
&= (-1)^{m}\intt{\bra*{\nabla \partial_{\alpha} \chi_{i}} \cdot  \sum_{\gamma\le \alpha}C_{\alpha, \gamma} \bra*{\partial_\gamma \tilde{\mu}^\eps} \bra*{\partial_{\alpha-\gamma} \nabla \chi_i}} \\
&= (-1)^{m}\intt{\abs*{\nabla \partial_\alpha \chi_i}^2 \tilde{\mu}^\eps} \\&+(-1)^{m}\intt{\bra*{\nabla \partial_{\alpha} \chi_{i}} \cdot  \sum_{\substack{\gamma\le \alpha\\
\gamma\ne 0}} C_{\alpha,k}\bra*{\partial_\gamma \tilde{\mu}^\eps} \bra*{\partial_{\alpha-\gamma} \nabla \chi_i}},
\end{align}
where the coefficients are given by
$$
C_{\alpha,\gamma}=\prod_{l=1}^d{\alpha_l \choose \gamma_l}.
$$

Similarly for the right hand side of \eqref{eq:tdcorrectori} we obtain
\begin{align}
-\intt{\bra*{\partial_{2 \alpha}\chi_i} \bra*{  \partial_{x_i} \tilde{\mu}^\eps}} = (-1)^{m-1} \intt{(\partial_{x_i} \partial_{\alpha} \chi_i ) \partial_\alpha \tilde{\mu}^\eps} \, .
\end{align}
Putting the previous two equations together and multiplying by $(-1)^{-m}$ we have
\begin{align}
\intt{\abs*{\nabla \partial_\alpha \chi_i}^2 \tilde{\mu}^\eps} = - \intt{(\partial_{x_i} \partial_{\alpha} \chi_i)  \partial_\alpha \tilde{\mu}^\eps} -\intt{\bra*{\nabla \partial_{\alpha} \chi_{i}} \cdot \sum_{\substack{\gamma\le \alpha\\
\gamma\ne 0}} C_{\alpha,\gamma}\bra*{\partial_\gamma \tilde{\mu}^\eps} \bra*{\partial_{\alpha-\gamma} \nabla \chi_i}} \, .
\label{eq:Hmestimate}
\end{align}
Using the exponential form of $\mu^\eps$, we note that for any multi-index $\alpha \in \N^d$ we have $\partial_\alpha \tilde{\mu}^\eps= f^\alpha \tilde{\mu}^\eps$, where $f^\alpha$ is a smooth function which is a linear combination of $\partial_\gamma W \ast \tilde{\nu}^\eps +\partial_\gamma V$ for $\gamma \leq \alpha$. This implies that we can obtain the bound
\begin{align}
\norm{f^\alpha}_{\Leb^\infty(\T^d)} \leq C_\alpha
\end{align}
where $C_\alpha$ depends only on $\norm{W}_{C^{m-1}(\T^d)} , \;\norm{V}_{C^{m-1}(\T^d)}$ .
 Applying H\"older's inequality and bounding in \eqref{eq:Hmestimate} we obtain
\begin{align}
\intt{\abs*{\nabla \partial_\alpha \chi_i}^2 \tilde{\mu}^\eps} &\leq \norm{\partial_{x_i} \partial_{\alpha} \chi_i}_{\Leb^2(\tilde{\mu}^\eps)} \norm{f^\alpha}_{\Leb^\infty(\T^d)} 
\\&+  \norm{\nabla \partial_{\alpha} \chi_{i}}_{\Leb^2(\tilde{\mu}^\eps)}\sum_{\substack{\gamma\le \alpha\\
\gamma\ne 0}} C_{\alpha,\gamma} \norm{f^\gamma}_{\Leb^\infty(\T^d)}  \norm{\partial_{\alpha-\gamma} \nabla \chi_i}_{\Leb^2(\tilde{\mu}^\eps)} 
\end{align}
Simplifying, we obtain,
\begin{align}
\norm{\nabla \partial_\alpha \chi_i}_{\Leb^2(\tilde{\mu}^\eps)} &\leq  \norm{f^\alpha}_{\Leb^\infty(\T^d)}
+ \sum_{\substack{\gamma\le \alpha\\
\gamma\ne 0}} C_{\alpha,\gamma} \norm{f^\gamma}_{\Leb^\infty(\T^d)}  \norm{\partial_{\alpha-\gamma} \nabla \chi_i}_{\Leb^2(\tilde{\mu}^\eps)} \\
& \leq C_\alpha + \sum_{\substack{\gamma\le \alpha\\
\gamma\ne 0}} C_{\alpha,\gamma} C_\gamma  \norm{  \chi_i}_{\SobH^{m-1}(\tilde{\mu}^\eps)}.
\end{align}
We can sum over all such $\alpha$ and recursively apply this bound along with~\eqref{eq:H1bound} to obtain~\eqref{eq:Hmbound}.
Note that the $\norm{\chi_i}_{\Leb^2(\T^d)}$ norm can be controlled by the Poincar\'e inequality since $\chi_i$ is mean zero. 

Before we turn to the regularity of~\eqref{eq:tdcorrectorit}, we derive the following estimates
\begin{align}\label{eq:aux111}
|\partial_t \tilde{\mu}^\eps| &= \beta  (W\ast \partial_t\tilde{\nu}^\eps)  \tilde{\mu}^\eps
 \leq  \beta \norm{W}_{C^3(\T^d)}\|\partial_t \tilde{\nu}^\eps\|_{C^{-3}(\T^d)} \tilde{\mu}^\eps \,,
\end{align}
where we denote by $C^{-3}(\T^d)$ the dual of $C^3(\T^d)$ and equip it with the norm $\norm{g}_{C^{-3}(\T^d)}=\sup\limits_{\norm{f}_{C^3(\T^d)}\leq 1 }\skp{f,g}$.  Similarly, for $\alpha \in \N^d$ the following estimate holds
\begin{align}
|\partial_{\alpha} \partial_t \tilde{\mu}^\eps| \leq \beta \norm{W}_{C^{3+|\alpha|}}\|\partial_t \tilde{\nu}^\eps\|_{C^{-3}(\T^d)} \tilde{\mu}^\eps
\end{align}

Next, we test~\eqref{eq:tdcorrectorit} against $\xi_i$ to obtain
\begin{align}
\intt{\abs{\nabla \xi_i}^2 \dx{\tilde{\mu}^\eps}} &= -\intt{\partial_{x_i} \xi_i \partial_t \tilde{\mu}^\eps} -\intt{\nabla \xi_i \cdot \nabla \chi_i \partial_t \tilde{\mu}^\eps} \\
& \leq \beta\norm{W}_{C^3(\T^d)} \norm{\nabla\xi_i}_{\Leb^2(\tilde{\mu}^\eps)}\|\partial_t \tilde{\nu}^\eps\|_{C^{-3}(\T^d)}(t) \\&+ \beta\norm{W}_{C^3(\T^d)} \norm{\nabla\xi_i}_{\Leb^2(\tilde{\mu}^\eps)} \norm{\nabla\chi_i}_{\Leb^2(\tilde{\mu}^\eps)} \|\partial_t \tilde{\nu}^\eps\|_{C^{-3}(\T^d)}(t) \, ,
\end{align}
where we have simply used~\eqref{eq:aux111} and applied the Cauchy--Schwartz inequality. It follows that
\begin{align}
\norm{\xi_i}_{\SobH^1(\tilde{\mu}^\eps)} \leq C \|\partial_t \tilde{\nu}^\eps\|_{C^{-3}(\T^d)}(t) \, ,
\end{align}
where the constant $C$ is independent of $t$ and depends on $\norm{\chi_i}_{\SobH^1(\tilde{\mu}^\eps)}$, $V$, and $W$. We omit the details but an essentially similar argument to the one used for~\eqref{eq:tdcorrectori} will give us an estimate of the form
\begin{align}
\norm{\nabla \partial_\alpha \xi_i}_{L^2(\tilde{\mu}^\eps)} \leq C' \|\partial_t \tilde{\nu}^\eps\|_{C^{-3}(\T^d)}(t) +\sum_{l=1}^m
C_l \norm{\xi_i}_{\SobH^{m-l}(\tilde{\mu}^\eps)}\|\partial_t \tilde{\nu}^\eps\|_{C^{-3}(\T^d)}(t) \, ,
\end{align}
where $\abs{\alpha}=m-1$, and the constants $C',\;C_l$ are independent of $t$ and depend on the norms of $\chi_i$, $W$, $V$, and their derivatives. Recursively applying these bounds one obtains~\eqref{eq:Hmboundt}.
\end{proof}
Next, we bound $\|\partial_t \tilde{\nu}^\eps\|_{C^{-3}(\T^d)}$ by $d_2(\tilde{\nu}^\eps,\tilde{\nu}^*)$.
\begin{lemma}\label{lem:boundoftimederivative}
Assume that $\tilde{\nu}^\eps$ and $\tilde{\nu}^*$ are a solution and a steady state to \eqref{periodiceps=1} respectively, then
$$
\|\partial_t \tilde{\nu}^\eps\|_{C^{-3}(\T^d)}\le C d_2(\tilde{\nu}^\eps,\tilde{\nu}^*),
$$
where the constant $C$ depends on dimension, $\beta$, $\norm{W}_{C^2(\T^d)}$, and $\norm{V}_{C^{2}(\T^d)}$.
\end{lemma}
\begin{proof}
Using \eqref{periodiceps=1} and that $\tilde{\nu}^*$ is a steady state, we obtain that for any test function $\varphi$
$$
\int_{\T^d} \partial_t \tilde{\nu}^\eps\varphi\;\dx{x}=\beta\int_{\T^d}\Delta\varphi  \tilde{\nu}^\eps-\nabla\varphi\cdot( \nabla W\ast\tilde{\nu}^\eps+\nabla V)\tilde{\nu}^\eps\;\dx{x}
$$
and
$$
0=\int_{\T^d} \partial_t \tilde{\nu}^*\varphi\;\dx{x}=\beta\int_{\T^d}\Delta\varphi \tilde{\nu}^*-\nabla\varphi\cdot( \nabla W\ast\tilde{\nu}^*+\nabla V)\tilde{\nu}^*\;\dx{x}.
$$
Therefore,
$$
\int_{\T^d} \partial_t \tilde{\nu}^\eps\varphi\;\dx{x}=\beta\int_{\T^d}\Delta\varphi  (\tilde{\nu}^\eps-\nu^*)-\nabla\varphi\cdot( \nabla W\ast\tilde{\nu}+\nabla V)(\tilde{\nu}^\eps-\tilde{\nu}^*)+ \nabla\varphi\cdot \nabla W\ast(\tilde{\nu}^\eps-\tilde{\nu}^*)\tilde{\nu}^*\;\dx{x}
$$
Using the dual formulation of the 1-Wasserstein distance we can obtain the following bound
$$
\left|\int_{\T^d} \partial_t \tilde{\nu}^\eps\varphi\;\dx{x}\right|=\beta (\|\varphi\|_{C^3(\T^d)}+ \|\varphi\|_{C^2(\T^d)}(\norm{W}_{C^2(\T^d)}+\norm{V}_{C^2(\T^d)})+\|\varphi\|_{C^1}\norm{W}_{C^2(\T^d)})d_1(\tilde{\nu}^\eps,\tilde{\nu}^*).
$$
Finally, bounding the 1-Wasserstein distance by the 2-Wasserstein distance we obtain
$$
d_1(\tilde{\nu}^\eps,\tilde{\nu}^*)\le d_2(\tilde{\nu}^\eps,\tilde{\nu}^*) \, .
$$
Thus we have the desired estimate.
\end{proof}

We now study the behaviour of the underlying SDE associated to~\eqref{eps=1}.
\begin{lemma}\label{mCLT}
Consider the mean field SDE
\begin{align}
\begin{cases} \label{eq:mckeanSDE}
\dx{Y}_t^\eps &= -\nabla V(Y_t^\eps)\dx{t} - \nabla (W \ast \nu (t))(Y_t^\eps)\dx{t} + \sqrt{2 \beta^{-1}} dB_t \\
\mathrm{Law}(Y_0^\eps)&= \nu_0^\eps=\eps^d\rho_0^\eps(\eps x) \in \cP(\R^d) \, .
\end{cases}
\end{align}
where $\nu(t)$ is a solution of~\eqref{eps=1} with initial data $\nu_0^\eps$ and  $B_t$ is a standard $d$-dimensional Wiener process Then for fixed $\eps>0$, the random variables $\set{t^{-1/2} Y_t^\eps}_{t > 0}$ converge in law (specifically in $d_2$) as $t \to \infty$ to a mean zero Gaussian random variable $Y$ with covariance matrix $2 A^{\eff}_* \in \R^{d \times d}$. 
\end{lemma}
\begin{proof}
Consider $\nu^\eps(t)$ the solution to the mean field PDE \eqref{eps=1} with initial data given by $\nu^\eps_0$. (we add the $\eps$ superscript to $\nu(t)$ to emphasise the dependence of the initial data on $\eps$). As $V$ and $W$ are smooth $1$-periodic functions, it follows that $(W\ast\nu(t))(x)$ is also $1$-periodic and is equal to $(W\ast \tilde{\nu}^\eps(t))(x)$, where $\tilde{\nu}^\eps$ is the periodic rearrangement of $\nu^\eps(t)$. Thus the SDE in~\eqref{eq:mckeanSDE} can be rewritten as
\begin{align}
\begin{cases} \label{eq:mckeanSDEqr}
\dx{Y}_t^\eps &= -\nabla V(\dot{Y}_t^\eps)\dx{t} - \nabla (W \ast \tilde{\nu}^\eps (t))(\dot{Y}_t^\eps)\dx{t} + \sqrt{2 \beta^{-1}} dB_t \\
\mathrm{Law}(Y_0^\eps)&= \nu_0^\eps=\eps^d\rho_0^\eps(\eps x) \in \cP(\R^d) \, ,
\end{cases}
\end{align}
where $\dot{Y}_t^\eps$ is the quotient process, i.e. $(\dot{Y}_t^\eps)_j=(Y_t^\eps)_j\; (\mbox{mod} \;  1)$ for all $j= 1,\dots,d$. Furthermore, $\dot{Y}_t^\eps$ satisfies the following SDE
\begin{align}
\begin{cases} \label{eq:mckeanSDEq}
\dx{\dot{Y}}_t^\eps &= -\nabla V(\dot{Y}_t^\eps)\dx{t} - \nabla (W \ast \tilde{\nu}^\eps (t))(\dot{Y}_t^\eps)\dx{t} + \sqrt{2 \beta^{-1}} d\dot{B}_t \\
\mathrm{Law}(\dot{Y}_0^\eps)&= \tilde{\nu}_0^\eps \in \cP(\T^d) \, ,
\end{cases}
\end{align}
where $\dot{B}_t$ is a $\T^d$-valued Wiener process. 
Now consider the unique solution $\chi(\cdot,t) \in \SobH^1(\tilde{\mu}^\eps)$ of the time-dependent corrector problem in~\eqref{eq:tdcorrector} given by Lemma~\ref{lem:exreg}. Applying Ito's lemma to
$\chi(Y_t^\eps,t)$ we obtain the following 
\begin{align}
\chi(Y_t^\eps,t)&= \chi(Y_0^\eps,0) + \int_0^t \partial_s \chi(\dot{Y}_s^\eps,s) \dx{s} \\& + \int_0^t \bra*{- \bra*{\nabla V(\dot{Y}_s^\eps) +\nabla (W \ast \tilde{\nu}^\eps (s))(\dot{Y}_s^\eps)} \cdot \nabla   + \beta^{-1} \Delta } \chi(\dot{Y}_s^\eps,s) \dx{s} \\
&+ \sqrt{2 \beta^{-1}} \int_0^t \nabla \chi(\dot{Y}_s^\eps,s) dB_s 
\\&=\chi(Y_0^\eps,0) + \int_0^t \partial_s \chi(\dot{Y}_s^\eps,s) \dx{s} + \int_0^t\beta^{-1} \bra*{\tilde{\mu}^\eps}^{-1}\bra*{\nabla \cdot\bra*{\tilde{\mu}^\eps \nabla \chi} } (\dot{Y}_s^\eps,s)\dx{s} \\
&+ \sqrt{2 \beta^{-1}} \int_0^t \nabla \chi(\dot{Y}_s^\eps,s) dB_s  \, ,
\end{align}
where we have used the fact that $f(Y_t^\eps)=f(\dot{Y}_t^\eps)$ for any $1$-periodic function $f$ and the equation for $\tilde{\mu}^\eps$. Using the fact that $\chi$ satisfies~\eqref{eq:tdcorrector}, the above expression simplifies to
\begin{align}
\chi(Y_t^\eps,t)&= \chi(Y_0^\eps,0) + \int_0^t \partial_s \chi(\dot{Y}_s^\eps,s) \dx{s}  - \int_0^t\beta^{-1} \bra*{\tilde{\mu}^\eps}^{-1} \nabla \tilde{\mu}^\eps(\dot{Y}_s^\eps,s)  \dx{s} \\
&+ \sqrt{2 \beta^{-1}} \int_0^t \nabla \chi(\dot{Y}_s^\eps,s) dB_s   \\
&= \chi(\dot{Y}_0^\eps,0) + \int_0^t \partial_s \chi(\dot{Y}_s^\eps,s) \dx{s}   +\int_0^t\nabla V(\dot{Y}_s^\eps)+ \nabla (W \ast \tilde{\nu}^\eps (s))(\dot{Y}_s^\eps)\dx{s}\\
&+ \sqrt{2 \beta^{-1}} \int_0^t \nabla \chi(\dot{Y}_s^\eps,s) dB_s \, .
\end{align}
Integrating~\eqref{eq:mckeanSDEq} from $0$ to $t$ and adding the above expression we obtain
\begin{align}
Y_t&= Y_0+ \chi(Y_0^\eps,0) - \chi(Y_t^\eps,t)  + \int_0^t \partial_s \chi(\dot{Y}_s^\eps,s) \dx{s}  \\
& + \sqrt{2 \beta^{-1}} \int_0^t \bra*{I+ \nabla \chi(\dot{Y}_s^\eps,s)} dB_s \, .
\end{align}
Multiplying by $t^{-1/2}$ we obtain
\begin{align}
t^{-1/2}Y_t^\eps &= t^{-1/2} \bra*{\chi(Y_0^\eps,0) - \chi(Y_t^\eps,t)   + \int_0^t \partial_s \chi(\dot{Y}_s^\eps,s) \dx{s} } + t^{-1/2}Y_0^\eps \\
& +t^{-1/2} \sqrt{2 \beta^{-1}} \int_0^t \bra*{I+ \nabla \chi(\dot{Y}_s^\eps,s)} dB_s \, .
\end{align}
To analyse the limit of $t^{-1/2}Y_t^\eps$ we start by showing that the first three terms on the RHS of the above expression go to zero in $\Leb^\infty(\mathbb{P})$ as $t \to \infty$. Picking $m>d/2$ and applying the results of Lemma~\ref{lem:exreg} along with Morrey's inequality we have
\begin{align}
&t^{-1/2} \bra*{\chi(Y_0^\eps,0) - \chi(Y_t^\eps,t)   + \int_0^t \partial_s \chi(\dot{Y}_s^\eps,s) \dx{s} } \\&\leq 
t^{-1/2} C_d \bra*{2\norm{\chi(\cdot,t)}_{\SobH^m(\T^d)}   + \int_0^t \norm{\partial_s \chi(\cdot,t)}_{\SobH^m(\T^d)} \dx{s} }\\
& \leq t^{-1/2} C_d \bra*{2C_1   + \int_0^t \sum_{k=1}^m c_k \|\partial_t\tilde{\nu}^\eps\|_{C^{-3}(\T^d)}^k(s) \dx{s} } 
\stackrel{t \to \infty}{\to} 0 \, ,
\label{e1}
\end{align}
where in the last step we have used Lemma~\ref{lem:boundoftimederivative} and applied assumption~\eqref{exponential convergence}. For the fourth term we simply use the fact that $Y_0$ has finite second moment to argue that it goes to zero in $\Leb^2(\mathbb{P})$. Thus studying the behaviour of $t^{-1/2} Y_t$, in law, as $t \to \infty$
is equivalent to studying the asymptotic behaviour of the martingale term $Z_t$, where
\begin{align}
Z_t:=  t^{-1/2} \sqrt{2 \beta^{-1}} \int_0^t \bra*{\Id+ \nabla \chi(\dot{Y}_s^\eps,s)} dB_s \, .
\end{align}

We will proceed in steps: In \textit{Step 1}, we will argue that the $\chi$ in the above expression can be replaced by $\Psi^*$, where $\Psi^*$ solves~\eqref{limiting corrector}. In \textit{Step 2}, we will compute the limiting covariance matrix of $Z_t$ as $t \to \infty$ and show that it is precisely $2 A^{\eff}_*$. Finally, in \textit{Step 3}, we will argue that the limiting random variable is a mean zero Gaussian.

\textit{Step 1.} First note that $\mu^\eps(t) \to \tilde{\nu}^*$ in $L^\infty$ as $t \to \infty$. Indeed, we have that
\begin{align}
&\abs*{Z^{-1}(\tilde{\mu}^\eps) \exp\bra*{-\beta( W \ast \tilde{\nu}^\eps +V)} -Z^{-1}(\tilde{\nu}^*)\exp\bra*{-\beta( W \ast \tilde{\nu}^* +V)} } 
\\
& \leq Z^{-1}(\tilde{\mu}^\eps) \abs*{ \exp\bra*{-\beta( W \ast \tilde{\nu}^\eps +V)} -\exp\bra*{-\beta( W \ast \tilde{\nu}^* +V)} }\\
& +  \abs*{Z^{-1}(\tilde{\mu}^\eps)- Z^{-1}(\tilde{\nu}^*)} e^{\beta(\norm{W}_{\Leb^\infty(\T^d)} + \norm{V}_{\Leb^\infty(\T^d)})} \\
& \leq e^{2\beta(\norm{W}_{\Leb^\infty(\T^d)} + \norm{V}_{\Leb^\infty(\T^d)})} \beta\abs*{ W \ast \tilde{\nu}^\eps-  W \ast \tilde{\nu}^*} +  e^{3\beta(\norm{W}_{\Leb^\infty(\T^d)} + \norm{V}_{\Leb^\infty(\T^d)})} \beta\abs*{ W \ast \tilde{\nu}^\eps-  W \ast \tilde{\nu}^*} \\
& \leq Cd_2(\tilde{\nu}^\eps, \tilde{\nu}^*) \stackrel{t \to \infty}{\to}0  \, , 
\label{eq:munuconvergence}
\end{align}
where we have used~\eqref{exponential convergence}. We now argue that $\nabla \chi$ converges to $\nabla \Psi^*$ in $\Leb^2(\T^d; \R^d)$. We perform the proof component-wise using the weak formulations of\eqref{eq:tdcorrector} and~\eqref{limiting corrector}
\begin{align}
&\intt{\nabla \chi_i \cdot \nabla \phi \tilde{\mu}^{\eps}(t)} - \intt{\nabla \Psi^*_i \cdot \nabla \phi \tilde{\mu}^{\eps}(t)} \\&= \intt{\partial_{x_i}(\tilde{\mu}^\eps-\tilde{\nu^*} ) \phi } +\intt{\nabla \Psi^*_i \cdot \nabla \phi \tilde{\nu}^{*}(t)} - \intt{\nabla \Psi^*_i \cdot \nabla \phi \tilde{\mu}^{\eps}(t)}  \\
&=-\intt{(\tilde{\mu}^\eps-\tilde{\nu^*} ) \partial_{x_i} \phi } +\intt{\nabla \Psi^*_i \cdot \nabla \phi \bra*{\tilde{\nu}^{*}(t) -\tilde{\mu}^{\eps}(t) }}  \\
& \leq  \norm{\tilde{\mu}^\eps-\tilde{\nu^*} }_{\Leb^\infty(\T^d)}(1+ \norm{\nabla\Psi^*_i}_{\Leb^2(\T^d)})\norm{\nabla\phi }_{\Leb^2(\T^d)} \, .
\end{align}
Choosing $\phi=\chi_i- \Psi^*_i$ and using the uniform lower bound from~\eqref{eq:mubounds}, we obtain that
\begin{align}
\norm{\nabla(\chi_i- \Psi^*_i)}_{\Leb^2(\T^d)} \leq C\norm{\tilde{\mu}^\eps-\tilde{\nu^*} }_{\Leb^\infty(\T^d)}(1+ \norm{\nabla\Psi^*_i}_{\Leb^2(\T^d)})  \stackrel{t \to \infty}{\to}0 \, ,
\label{eq:chitopsi}
\end{align}
using~\eqref{eq:munuconvergence}. Thus we can now simply apply Ito's isometry as follows
\begin{align}
 & \mathbb{E} \pra*{t^{-1} \abs*{\sqrt{2 \beta^{-1}} \int_0^t \bra*{\Id+ \nabla \chi(\dot{Y}_s^\eps,s)} dB_s -   \sqrt{2 \beta^{-1}} \int_0^t \bra*{\Id+ \nabla \Psi^*(\dot{Y}_s^\eps,s)} dB_s}^2} \\
 &=2 \beta^{-1}t^{-1}\mathbb{E} \pra*{ \abs*{   \int_0^t \bra*{\nabla(\chi- \Psi^*)}(\dot{Y}_s^\eps,s) dB_s}^2}  \\
 &=2 \beta^{-1}t^{-1}\mathbb{E} \pra*{    \int_0^t \sum_{i=1}^d\abs*{\nabla(\chi_i- \Psi^*_i)}^2(\dot{Y}_s^\eps,s) \dx{s}}   \\
 & \leq 2 \beta^{-1}t^{-1}  \int_0^t \sum_{i=1}^d\norm*{\nabla(\chi_i- \Psi^*_i)}_{\Leb^\infty(\T^d)}^2(s) \dx{s} \, . 
\end{align}
Picking some $m >d/2$ and applying Morrey's inequality we obtain
\begin{align}
 & \mathbb{E} \pra*{t^{-1} \abs*{\sqrt{2 \beta^{-1}} \int_0^t \bra*{\Id+ \nabla \chi(\dot{Y}_s^\eps,s)} dB_s -   \sqrt{2 \beta^{-1}} \int_0^t \bra*{\Id+ \nabla \Psi^*(\dot{Y}_s^\eps,s)} dB_s}^2} \\
 & \leq 2 \beta^{-1}t^{-1} C_d^2 \int_0^t \sum_{i=1}^d\norm*{\nabla(\chi_i- \Psi^*_i)}_{\SobH^m_0(\T^d)}^2(s) \dx{s} \\
 &\leq 2 \beta^{-1}t^{-1} C_d^2 C^2 \int_0^t \sum_{i=1}^d\norm*{\nabla(\chi_i- \Psi^*_i)}_{\SobH^{m+1}_0(\T^d)}^{2 \alpha}(s) \norm*{\nabla(\chi_i- \Psi^*_i)}_{\Leb^{2}(\T^d)}^{2-2 \alpha}(s) \dx{s}  \, ,
\end{align}
where we have applied the Gagliardo--Nirenberg--Sobolev inequality and $\alpha=m/(m+1)$. We bound the $\SobH^{m+1}$-norm in the above expression by a uniform constant using Lemma~\ref{lem:exreg} and the fact that $\Psi^*$ is the solution of a uniformly elliptic PDE with smooth coefficients. Hence, using~\eqref{eq:chitopsi} we obtain
 \begin{align}
 \mathbb{E} \pra*{t^{-1} \abs*{\sqrt{2 \beta^{-1}} \int_0^t \bra*{\Id+ \nabla \chi(\dot{Y}_s^\eps,s)} dB_s -   \sqrt{2 \beta^{-1}} \int_0^t \bra*{\Id+ \nabla \Psi^*(\dot{Y}_s^\eps,s)} dB_s}^2}\stackrel{t \to \infty}{\to} 0 \, . \label{e3}
\end{align}

\textit{Step 2.} In this step, we compute the limiting covariance as $t \to \infty$ of the following term
\begin{align}
G_t^\eps:=  t^{-1/2} \sqrt{2 \beta^{-1}} \int_0^t \bra*{\Id+ \nabla \Psi^*(\dot{Y}_s^\eps)} dB_s \, .
\label{intermediate variable}
\end{align}
Applying Ito's isometry again we have
\begin{align}
\mathbb{E} \pra*{(G_t^\eps)_i (G_t^\eps)_j}& = 2 \beta^{-1} t^{-1} \int_0^t \mathbb{E}\pra*{\bra*{\sum_{k=1}^d (\delta_{ik} + \partial_{x_k} \Psi^*_i)(\delta_{jk} + \partial_{x_k} \Psi^*_j)}(\dot{Y}_s^\eps)}  \dx{s} 
\\
&= 2 \beta^{-1} t^{-1} \int_0^t \intt{\bra*{\sum_{k=1}^d (\delta_{ik} + \partial_{x_i} \Psi^*_k)(\delta_{jk} + \partial_{x_j} \Psi^*_k)}(x) \tilde{\nu}^\eps(x,s)}  \dx{s} \\
&= 2 \beta^{-1} t^{-1} \int_0^t \intt{\bra*{\sum_{k=1}^d (\delta_{ik} + \partial_{x_i} \Psi^*_k)(\delta_{jk} + \partial_{x_j} \Psi^*_k)}(x) (\tilde{\nu}^\eps-\tilde{\nu}^*)(x,s)}  \dx{s} \\
&+2 \beta^{-1} t^{-1} \int_0^t  \dx{s} \bra*{\intt{\bra*{\sum_{k=1}^d (\delta_{ik} + \partial_{x_i} \Psi^*_k)(\delta_{jk} + \partial_{x_j} \Psi^*_k)}(x) \tilde{\nu}^*(x)}} \, .
\label{limiting covariance}
\end{align} 
We can bound the first term as follows
\begin{align}
&2 \beta^{-1} t^{-1} \int_0^t \intt{\bra*{\sum_{k=1}^d (\delta_{ik} + \partial_{x_i} \Psi^*_k)(\delta_{jk} + \partial_{x_j} \Psi^*_k)}(x) (\tilde{\nu}^\eps-\tilde{\nu}^*)(x,s)}  \dx{s}\\
&\qquad\qquad \lesssim t^{-1}
\int_0^t d_2(\tilde{\nu}^\eps(s),\tilde{\nu}^*) \dx{s} \stackrel{t \to \infty}{\to} 0 \,, \label{e4}
\end{align}
where we have used~\eqref{exponential convergence} and the fact $\nabla \Psi^*$ is Lipschitz. Then from~\eqref{limiting covariance} it follows that
\begin{align}
\lim_{t \to \infty}\mathbb{E} \pra*{(G_t^\eps)_i (G_t^\eps)_j}&=2 \beta^{-1}\bra*{\intt{\bra*{\sum_{k=1}^d (\delta_{ik} + \partial_{x_i} \Psi^*_k)(\delta_{jk} + \partial_{x_j} \Psi^*_k)}(x) \tilde{\nu}^*(x)}} \\
&= 2 \beta^{-1}\bra*{\intt{ \bra*{\delta_{ij} + \partial_{x_i} \Psi^*_j + \partial_{x_j} \Psi^*_i + \sum_{k=1}^d\partial_{k} \Psi^*_i \partial_{k} \Psi^*_j } (x) \tilde{\nu}^*(x)}} \\
&=2 \beta^{-1}\bra*{\intt{ \bra*{\delta_{ij} + \partial_{x_i} \Psi^*_j + \partial_{x_j} \Psi^*_i + \nabla \Psi^*_i \cdot \nabla  \Psi^*_j } (x) \tilde{\nu}^*(x)}} \\
&=2 \beta^{-1}\bra*{\intt{ \bra*{\delta_{ij} + \partial_{x_i} \Psi^*_j } (x) \tilde{\nu}^*(x)}}= 2 \bra*{A^{\eff}_*}_{ij} \, ,
\end{align}
where in the penultimate step we have used the fact that $\Psi^*$ satisfies~\eqref{limiting corrector}.

\textit{Step 3.} In the final step, we will show that the limit in law of $G_t^\eps$ as $t \to \infty$ is a Gaussian random variable. The key step involves replacing  $\dot{Y}_t^\eps$ in the expression for $G_t^\eps$ in~\eqref{intermediate variable} by $\dot{X}_t$, where $\dot{X}_t$ solves
\begin{align}\label{eq:mckeanSDEs}
\begin{cases}
\dx{\dot{X}}_t &= -\nabla V(\dot{X}_t)\dx{t} - \nabla (W \ast \tilde{\nu}^* )(\dot{X}_t)\dx{t} + \sqrt{2 \beta^{-1}} d\dot{B}_t \\
\mathrm{Law}(\dot{X}_0)&= \tilde{\nu}^* \in \cP(\T^d) \, .
\end{cases}
\end{align}
Here $\dot{X}_t$ is a stationary, ergodic process with invariant measure $\tilde{\nu}^*$ and is precisely the process $\dot{Y}_t$ started from the invariant measure  $\tilde{\nu}^*$.  We assert now (cf. Lemma~\ref{coupling existence}) that~\eqref{exponential convergence} implies that there exists a coupling of $(\dot{X}_t,\dot{Y}_t)$ (indeed a reflection coupling) such that $\sup_{\eps >0}\mathbb{E}\pra*{d_{\T^d}(\dot{X}_t,\dot{Y}_t)^2} \to 0$ as $t \to \infty$. Using this we obtain,
\begin{align}
\mathbb{E}\pra*{\bra*{G_t^\eps-t^{-1/2} \sqrt{2 \beta^{-1}} \int_0^t \bra*{\Id+ \nabla \Psi^*(\dot{X}_s)} dB_s}^2 } \leq  2 C\beta^{-1}t^{-1}\int_0^t \mathbb{E}\pra*{d_{\T^d}(\dot{X}_s,\dot{Y}_s)^2} \dx{s} \stackrel{t \to \infty}{\to} 0 \, , \label{e5} 
\end{align}
where we simply use the Ito isometry and the fact that $\Psi^*$ has a Lipschitz regular gradient. Thus we are left to analyse the following term:
\begin{align}
t^{-1/2} \sqrt{2 \beta^{-1}} \int_0^t \bra*{\Id+ \nabla \Psi^*(\dot{X}_s)} dB_s
\end{align} 
where $\dot{X}_s$ is stationary ergodic process. Additionally, we know the limiting covariance of the above term, i.e. $2 A^{\eff}_*$. We now apply the Birkhoff ergodic theorem followed by the martingale central limit theorem (cf.~\cite[Theorem 2.1]{KLO12} or~\cite[Theorem 2.1]{Whi07}) to complete the proof. The fact that the convergence is also in $d_2$ follows from the fact that the covariance matrices also converge.
\end{proof}
The above result holds for a fixed $\eps>0$, however we can improve it by using the fact that the convergence in~\eqref{exponential convergence} is uniform in $\eps>0$. A consequence of the analysis in the previous result is the following corollary:
\begin{corollary}\label{uniform eps}
Consider the process:
\begin{align}
M_t^\eps:=t^{-\frac{1}{2}} \bra*{\chi(Y_0^\eps,0) - \chi(Y_t^\eps,t)  + \int_0^t \partial_s \chi(Y_s^\eps,s) \dx{s}  
 + \sqrt{2 \beta^{-1}} \int_0^t \bra*{I+ \nabla \chi(\dot{Y}_s^\eps,s)} dB_s \, } \, ,
\end{align}
where $\chi$ is the solution of the time-dependent corrector problem~\eqref{eq:tdcorrector} and $Y_t^\eps$ solves~\eqref{eq:mckeanSDE}. Then
\begin{align}
\lim_{t \to \infty}\sup_{\eps>0}\norm*{M_t^\eps-t^{-1/2} \sqrt{2 \beta^{-1}} \int_0^t \bra*{\Id+ \nabla \Psi^*(\dot{X}_s)} dB_s}_{\Leb^2(\mathbb{P})} \to 0 \, ,
\end{align}
where $\dot{X}_t$ is the solves and is coupled to $\dot{Y}_t^\eps$ as in the proof of Lemma~\ref{mCLT}.
\end{corollary}
\begin{proof}
The proof of this result follows from the fact that the convergence in~\eqref{e1},~\eqref{eq:munuconvergence},~\eqref{e3},~\eqref{e4}, and~\eqref{e5}, are all controlled by~\eqref{exponential convergence} which is uniform in $\eps>0$.
\end{proof}

We can finally put all the pieces together and complete the proof of Theorem~\ref{thm:variabledata}
\begin{proof}[Proof of Theorem~\ref{thm:variabledata}] We would like to understand the behaviour of the trajectory $S_t^\eps \rho_0^\eps$ where $S_t^\eps$ is the solution semigroup associated to~\eqref{nonlineareq}. However, we know that $S_t^\eps \rho_0^\eps= \mathrm{Law}(\eps Y^\eps_{t/\eps^2})$, where $Y_t^\eps$ is the solution~\eqref{eq:mckeanSDE} of with initial law $\eps^d \rho^\eps_0(\eps x)$. Fix $t=1$ and set $\eps^{-2}=s$. Thus we have that
\begin{align}
\eps Y^\eps_{1/\eps^2}& = s^{-\frac{1}{2}} Y^{s^{-\frac{1}{2}}}_{s} \\
&= s^{-\frac{1}{2}} Y_0^{s^{-\frac{1}{2}}} \\
&+ s^{-\frac{1}{2}} \bra*{\chi(Y_0,0) - \chi(Y^{s^{-1/2}}_s,s)  + \int_0^s \partial_u \chi(\dot{Y}^{u^{-1/2}}_u,u) \dx{u}  
 + \sqrt{2 \beta^{-1}} \int_0^s \bra*{I+ \nabla \chi(\dot{Y}^{u^{-1/2}}_u,u)} dB_u \, }.
\end{align}
Applying Corollary~\ref{uniform eps}, we can pass to the limit $s \to \infty$ for the second term on the right hand side of the above expression. Since the convergence is in $\Leb^2(\mathbb{P})$, we can replace the second term in the limit as $s \to \infty$  as follows
\begin{align}
\lim_{s \to \infty}\mathrm{Law}( s^{-\frac{1}{2}} Y^{s^{-\frac{1}{2}}}_{s} ) =  \lim_{s \to \infty}\mathrm{Law}  \bra*{s^{-\frac{1}{2}} Y_0^{s^{-\frac{1}{2}}}   + s^{-1/2}\sqrt{2 \beta^{-1}} \int_0^s \bra*{I+ \nabla \Psi^*(\dot{X}_u)} dB_u } \, .
\end{align}
where $\dot{X}_u$ solves~\eqref{eq:mckeanSDEs}. The two random variables on the right hand side of the above expression are independent. Thus we can rewrite the above expression as
\begin{align}
\lim_{s \to \infty}\mathrm{Law}( s^{-\frac{1}{2}} Y^{s^{-\frac{1}{2}}}_{s} )= \lim_{s \to \infty}F_s * \rho_0^{s^{-1/2}} \, ,
\end{align}
where $F_s$ is the law of $s^{-1/2}\sqrt{2 \beta^{-1}} \int_0^s \bra*{I+ \nabla \Psi^*(\dot{X}_u)} dB_u$.
Since both laws converge in $d_2$, their convolution converges to the convolution of the individual limits in $d_2$ as $s \to \infty$. The limit of $F_s$ can be obtained by the martingale central limit theorem as in the proof of Lemma~\ref{mCLT} while the limit of $\rho_0^{s^{-1/2}}$ is $\rho_0^*$. Thus we have that
\begin{align}
\lim_{s \to \infty}d_2(\mathrm{Law}( s^{-\frac{1}{2}} Y^{s^{-\frac{1}{2}}}_{s} ) , \mathcal{N}(0,A_*^{\eff}) * \rho^* _0)=\lim_{s \to \infty}d_2(S_1^{s^{-\frac{1}{2}}} \rho_0^{s^{-\frac{1}{2}}} , \mathcal{N}(0,A_*^{\eff}) * \rho^* _0) =0 .
\end{align}
Rewriting the same in terms of the laws we have
\begin{align}
\lim_{\eps \to 0}d_2(S_1^\eps \rho_0^\eps , \mathcal{N}(0,A_*^{\eff}) * \rho^* _0) =0.
\end{align}
The choice of time $t=1$ was arbitrary. The same arguments can be repeated for arbitrary $t \geq 0$ to complete the proof of~\eqref{weak convergence}. Assume now that the initial data of~\eqref{eq:LinearKolmogorov}, $\rho_0^{\eps,N}$ is such that $\lim_{N \to \infty}\rho_0^{\eps,N} =X_0^\eps=\delta_{\rho_0^\eps}$. We can then apply Theorem~\ref{thma} to first assert that, for a fixed $t>0$,
\begin{align}
\lim_{N \to \infty}\rho^{\eps,N}(t)= X^\eps(t)= S_t^\eps \# X_0^\eps \, .
\end{align}
Let $X_0=\delta_{\rho_0^*}$ and consider some $\Phi \in \Lip(\cP(\R^d))$. Then we have 
\begin{align}
\int_{\cP(\R^d)} \Phi(\rho)\dx{\bra*{X(t)^\eps - S_t^*\# X_0 }}(\rho) &= \int_{\cP(\R^d)} \Phi(\rho)\dx{\bra*{S_t^\eps\#X_0^\eps - S_t^*\# X_0}}(\rho) \\
&=\Phi(S_t^\eps \rho_0^\eps)-\Phi(S_t^*\rho_0^*) \leq d_2(S_t^\eps \rho_0^\eps,S_t^* \rho_0^*) \stackrel{\eps \to 0}{\to} 0 \, ,
\end{align}
where in the last step we have simply applied~\eqref{weak convergence}, thus completing the proof of the theorem.
\end{proof}

\section{Proof of Theorem~\ref{thm: N then eps}}\label{sec:2thm}
\textit{Strategy of proof.} We first need to pass to the limit in the covariance matrix
\begin{equation}\label{eq:effAN}
    A^{\eff,N}=\beta^{-1}\int_{\big(\T^d\big)^N} (I+\nabla \Psi^N) M_N\;\dx{x} \, .
\end{equation}
To do this, we first pass to the limit limit in the Poisson equation for $\Psi^N:\big(\T^d\big)^N\to\big(\R^d\big)^N$ 
\begin{equation}\label{eq:poisson1}
\nabla\cdot(M_N \nabla \Psi^N)=-\nabla M_N,    
\end{equation}
with
\begin{equation}\label{eq:zeroavg}
\int_{\big(\T^d\big)^N} \Psi^N M_N\;\dx{x}=0.    
\end{equation}
Once we have obtained the limit of the covariance matrix, we then need to pass to the limit in the equation
$$
\partial_t \rho^{N,*}=\nabla\cdot(A^{\eff,N}\nabla\rho^{N,*}).
$$
We do this by testing against cylindrical test functions, that is to say functions that depend on a finite number of variables, which is enough to determine the limit in $\mathcal{P}(\mathcal{P}(\R^d))$.

\vspace{0.2cm}

\textit{Step 1.} We start by showing a few a priori estimates. First, we show that for every $1\le i\le Nd$ and $k\in\{1,...,N\}$ such that $i$ does not belong to the particle $k$ (i.e. $i\notin[(k-1)d+1,kd]$) the solution of the corrector problem \eqref{eq:poisson1} satisfies 
   \begin{equation}\label{aprioriPsi}
    \int_{\big(\T^d\big)^N}|\nabla\Psi_i^N|^2 M_N\;\dx{x}\le 1 \qquad\mbox{and}\qquad \int_{\big(\T^d\big)^N}|\nabla_{x_k}\Psi_i^N|^2 M_N\;\dx{x}\le \frac{1}{N-1},
\end{equation}
where 
$$
|\nabla\Psi_i^N|^2:=\sum_{j=1}^{Nd} |\partial_j \Psi_i^N|^2\qquad \mbox{and}\qquad \nabla_{x_{k}}\Psi_i^N:=\sum_{j=(k-1)d+1}^{kd}\partial_j\Psi_i^N e_{j}.
$$
Testing the $i$-th equation of \eqref{eq:poisson1} against $\Psi_i^N$, applying the Cauchy-Schwarz inequality, and using the fact that $M_N$ has mass one, we obtain
$$
 \int_{\big(\T^d\big)^N} M_N |\partial_i \Psi_i^N|^2\;\dx{x} \le \int_{\big(\T^d\big)^N}|\nabla\Psi_i^N|^2M_N\;\dx{x}= \int_{\big(\T^d\big)^N} M_N \partial_i \Psi_i^N\;\dx{x} \le  \left(\int_{\big(\T^d\big)^N} M_N |\partial_i \Psi_i^N|^2\;\dx{x}\right)^{1/2},
$$
which implies the first inequality in \eqref{aprioriPsi}. The second inequality follows due to the exchangeability of the particles. In fact, for any $k_1,\; k_2\in\{1,...,N\}$ and $i\notin [(k_1-1)d+1, k_1d]\cup [(k_2-1)d+1, k_2d]$ we have that up to exchanging the $k_1$ and $k_2$ particles (i.e. changing  variables)
$$
\nabla_{x_{k_1}}\Psi_i^N=\nabla_{x_{k_2}}\Psi_i^N.
$$
Combining this with the first inequality of \eqref{aprioriPsi} we obtain the second inequality of \eqref{aprioriPsi}.

Next, we show that there exists $C(\beta,\norm{W}_{C^1},\norm{V}_{C^1})$ such that
\begin{equation}\label{aprioriM_N}
    \|(M_N)_1\|_{C^1(\T^d)}\le C\qquad\mbox{and}\qquad C^{-1} \le M_N/M_{N-1}\le C,
\end{equation}
where $M_{N-1}$ is the Gibbs measure associated to the quotiented $(N-1)$-particle system trivially extended to $\big(\T^d\big)^N$ and $(M_N)_1$ is the first marginal of $M_N$. We start by rewriting
\begin{equation}\label{eq:1}
\begin{array}{rl}
\ds M_N&\ds =\frac{e^{-\beta\bra*{\frac{1}{2N}\sum_{i=1}^N\sum_{j\ne i} W(x_i-x_j)+\sum_{i=1}^NV(x_i))}}}{Z_N}\\
&\ds =e^{-\beta\bra*{\frac{1}{N}\sum_{j=1}^NW(x_1-x_j)-\frac{1}{2N(N-1)}\sum_{i,j}^NW(x_i-x_j)+V(x_1))}} M_{N-1}\frac{Z_{N-1}}{Z_{N}}.    
\end{array}
\end{equation}
Differentiating the previous expression with respect to $x_1$ we obtain
\begin{equation}\label{eq:2}
\begin{array}{rl}
\ds \nabla_{x_1} M_N =&\ds -\beta\left(\left(\frac{1}{N}-\frac{1}{N(N-1)}\right)\sum_{j=1}^N\nabla W(x_1-x_j)+\nabla V(x_1)\right)\\
&\ds\qquad\qquad e^{-\beta\bra*{\frac{1}{N}\sum_{j=1}^NW(x_1-x_j)-\frac{1}{2N(N-1)}\sum_{i,j}^NW(x_i-x_j)+V(x_1))}} M_{N-1}\frac{Z_{N-1}}{Z_{N}}.    
\end{array}
\end{equation}
By \eqref{eq:1} and \eqref{eq:2} and the fact that $V$ and $W$ are sufficiently regular, the desired estimates follows if we can show that $Z_{N-1}/Z_{N}$ is bounded above and below. This follows from the following estimate
$$
\begin{array}{rl}
\ds Z_N &\ds = \int_{\big(\T^d\big)^N}e^{-\beta\bra*{\frac{1}{2N}\sum_{i=1}^N\sum_{j\ne i} W(y_i-y_j)+ \sum_{i=1}^N V(y_i))}}\;\dx{y} \\
&\ds \ge e^{-\beta\bra*{\bra*{\frac{N}{2(N-1)}+1}\norm{W}_{\infty} +\norm{V}_\infty}}\int_{\big(\T^d\big)^N}e^{-\beta\bra*{\frac{1}{2(N-1)}\sum_{i=2}^N\sum_{j\ne i} W(y_i-y_j)+ \sum_{i=2}^N V(y_i)}}\;\dx{y}\\
&\ds \ge C^{-1}Z_{N-1}
\end{array}
$$
and its analogue for the reverse bound.

\vspace{0.2cm}

\textit{Step 2.} Next we show that we can in a suitable sense decompose $M_N$ by the product $(M_N)_1 M_{N-1}$. To be precise, we show that for every $x_1\in \T^d$
\begin{equation}\label{eq:decomposition}
    \lim_{N\to \infty} \int_{\big(\T^d\big)^{N-1}} \left(\frac{M_N}{M_{N-1}}-(M_N)_1\right)^2M_{N-1}\;\dx{x}_2...\dx{x_N} = 0.
\end{equation}
We notice that
$$
(M_N)_1=\int_{\big(\T^d\big)^{N-1}}M_N\;\dx{x}_2...\dx{x_N}=\int_{\big(\T^d\big)^{N-1}}\frac{M_N}{M_{N-1}}M_{N-1}\;\dx{x}_2...\dx{x_N},
$$
and we rewrite
\begin{equation}
\begin{array}{rl}
&\ds \int_{\big(\T^d\big)^{N-1}} \left(\frac{M_N}{M_{N-1}}-(M_N)_1\right)^2M_{N-1}\\&\ds=\left|\frac{Z_{N-1}}{Z_{N}}\right|^2\int_{\big(\T^d\big)^{N-1}} \left(u^N-\int_{\big(\T^d\big)^{N-1}} u^NM_{N-1}\right)^2M_{N-1}\\
&\ds = \left|\frac{Z_{N-1}}{Z_{N}}\right|^2\left(\int_{\big(\T^d\big)^{N-1}} (u^N)^2M_{N-1}-\left(\int_{\big(\T^d\big)^{N-1}} u^NM_{N-1}\right)^2\right),
\end{array}
\end{equation}
where the function $u^N:\T^d\times\big(\T^d\big)^{N-1}\to \R$ is given by
$$
u^N=\frac{Z_N}{Z_{N-1}}\frac{M_N}{M_{N-1}}=e^{-\beta\bra*{\frac{1}{N}\sum_{j=2}^NW(x_1-x_j)-\frac{1}{2N(N-1)}\sum_{i,j=2}^NW(x_i-x_j)+V(x_1))}}.
$$
Therefore, by \eqref{aprioriM_N}, we can show \eqref{eq:decomposition} by showing that
\begin{equation}\label{eq:intermsofuN}
\lim_{N\to\infty}\int_{\big(\T^d\big)^{N-1}} (u^N)^2M_{N-1}\;\dx{x}_2...\dx{x_N}-\left(\int_{\big(\T^d\big)^{N-1}} u^NM_{N-1}\;\dx{x}_2...\dx{x_N}\right)^2=0.
\end{equation}
This will follow from the chaoticity assumption on $M_N$ and a version of the Arzela--Ascoli theorem for the limit of symmetric functions, where we employ an idea that was proposed by Lions \cite{lions2007mean} in the context of mean field games (cf.~\cite{cardialaguetnotes,Gol13}). 

We show that the sequence of functions $\{u^N\}_{N\in\N}$ induces a compact sequence $\{U^N\}_{N\in\N}\subset C(\T^d\times \mathcal{P}(\T^d))$ and that \eqref{eq:intermsofuN} can be written in terms of the limit of $U^N$. We start by noticing that $u_N$ is continuous and symmetric in the variables $x_2$ through $x_N$ such that there exists $C(\beta,W,V)$ such that
\begin{equation}
\norm{u^N}_{\Leb^{\infty}((\T^d)^N)}\le C,\qquad|\nabla_{x_1} u^N|\le C \qquad\mbox{and}\qquad |\nabla_{x_j} u^N|\le \frac{C}{N} \quad x_j \neq x_1 \, .    
\end{equation}
Using the symmetry of $u^N$ and the previous bound, we can estimate
\begin{equation}
\begin{array}{rl}
\ds    |u^N(x_1,x_2,...,x_N)-u^N(y_1,...,y_N)|&\ds =|u^N(x_1,x_{\sigma(2)},...,x_{\sigma(N)})-u^N(y_1,...,y_N)|\\
& \ds\le C d_{\T^d}(x_1,y_1)+\frac{C}{N}\sum_{i=2}^N d_{\T^d}(x_{\sigma(i)},y_i),
\end{array}
\end{equation}
with $\sigma$ an arbitrary permutation of the indices $\{2,3,...,N\}$. Taking the infimum over $\sigma$, we obtain that
\begin{equation}\label{eq:moduN}
    |u^N(x_1,x_2,...,x_N)-u^N(y_1,...,y_N)|\le C\left(|x_1-y_1|+ d_1\left(\frac{1}{N-1}\sum_{i=2}^N\delta_{x_i},\frac{1}{N-1}\sum_{i=2}^N\delta_{y_i}\right)\right),
\end{equation}
where $d_1$ denotes the 1-Wasserstein distance on $\mathcal{P}(\T^d)$. For any $(x_1,\mu)\in \T^d\times\mathcal{P}(\T^d)$, we define
\begin{equation}\label{eq:defUN}
U^N(x_1,\mu):=\inf_{(z_2,..., z_N)\in \big(\T^d\big)^{N-1}} 2C d_1\left(\mu, \frac{1}{N-1}\sum_{i=2}^N \delta_{z_i}\right)+ u^N(x_1,z_2,...,z_N) \in C^0(\T^d \times \cP(\T^d)).    
\end{equation}
It follows directly form \eqref{eq:moduN} that
\begin{equation}\label{eq:UN=uN}
    U^N\left(x_1,\frac{1}{N-1}\sum_{i=2}^N\delta_{x_i}\right)=  u^N(x_1,x_2,...,x_N).
\end{equation}
Using \eqref{eq:UN=uN} we can rewrite \eqref{eq:intermsofuN} as
\begin{equation}\label{eq:intermsofUN}
\lim_{N\to\infty}\int_{\mathcal{P}(\T^d)} (U^N)^2(x_1,\rho)\;d\hat{M}_{N-1}(\rho)-\left(\int_{\mathcal{P}(\T^d)} U^N(x_1,\rho)\;d\hat{M}_{N-1}(\rho)\right)^2=0,   
\end{equation}
\noindent where $\hat{M}_{N-1}\in \mathcal{P}(\mathcal{P}(\T^d))$ is the empirical measure associated with $M_{N-1}$ as defined in Definition~\ref{def:empirical}.

Next, we show that $U^N$ is Lipschitz with respect to the 1-Wasserstein distance, i.e. 
\begin{equation}\label{eq:LipUN}
    |U^N(x_1,\mu)-U^N(y_1,\nu)|\le 2C(d_{\T^d}(x_1,y_1)+d_1(\mu,\nu)).
\end{equation}
By the definition of $U^N(y_1,\nu)$, for every $\delta>0$ there exists $(z_2,...,z_N)$ such that
$$
U^N(y_1,\nu)+\delta  \ge 2C d_1\left(\nu, \frac{1}{N-1}\sum_{i=2}^N \delta_{z_i}\right)+ u^N(y_1,z_2,...,z_N).
$$
By the definition of $U^N(x_1,\mu)$, the Lipschitz property of $u^N$ \eqref{eq:moduN}, and the triangle inequality for $d_1$, we have
$$
\begin{array}{rl}
\ds U^N(x_1,\mu) &\ds\le 2C d_1\left(\mu, \frac{1}{N-1}\sum_{i=2}^N \delta_{z_i}\right)+ u^N(x_1,z_2,...,z_N)\\
&\ds \le Cd_{\T^d}(x_1,y_1)+ 2C d_1\left(\mu, \frac{1}{N-1}\sum_{i=2}^N \delta_{z_i}\right)+ u^N(y_1,z_2,...,z_N)\\
&\ds \le Cd_{\T^d}(x_1,y_1)+ 2C \left(d_1\left(\mu, \frac{1}{N-1}\sum_{i=2}^N \delta_{z_i}\right)-d_1\left(\nu, \frac{1}{N-1}\sum_{i=2}^N \delta_{z_i}\right)\right)+ U^N(y_1,\nu)+\delta\\
&\ds \le Cd_{\T^d}(x_1,y_1)+ 2Cd_1(\mu,\nu)+ U^N(y_1,\nu)+\delta.
\end{array}
$$
Using the fact that $\delta>0$, $(x_1,\mu)$ and $(y_1,\nu)$ are arbitrary, \eqref{eq:LipUN} follows. Due to the compactness of $\T^d$, the space $\T^d\times \mathcal{P}(\T^d)$ equipped with the metric $d_{\T^d}+d_1$ is also compact. Therefore, by the Arzela--Ascoli theorem and the uniform Lipschitz bound in~\eqref{eq:LipUN}, we have that, up to subsequence, there exists $U\in C^0(\T^d\times \mathcal{P}(\T^d))$ such that
\begin{equation}
    \lim_{N\to\infty}\|U^N-U\|_{C^0(\T^d\times \mathcal{P}(\T^d))}=0.
\end{equation}

Finally, we use the assumption that $\hat{M}_{N-1}\to \delta_{\tilde{\nu}^{\min}}\in\mathcal{P}(\mathcal{P}(\T^d))$ (see Remark~\ref{rmk:chaoticity}), to obtain that, up to subsequence,
$$
\begin{array}{l}
\ds \lim_{N\to\infty}\int_{\mathcal{P}(\T^d)} (U^N)^2(x_1,\rho)\;d\hat{M}_{N-1}(\rho)-\left(\int_{\mathcal{P}(\T^d)} U^N(x_1,\rho)\;d\hat{M}_{N-1}(\rho)\right)^2\\
\qquad\qquad\ds=\int_{\mathcal{P}(\T^d)} U^2(x_1,\rho)\;d\delta_{\tilde{\nu}^{\min}}(\rho)-\left(\int_{\mathcal{P}(\T^d)} U(x_1,\rho)\;d\delta_{\tilde{\nu}^{\min}}(\rho)\right)^2 \\
\qquad\qquad\ds= U^2(x_1,\tilde{\nu}^{\min})-U^2(x_1,\tilde{\nu}^{\min})\\
\qquad\qquad\ds=0.
\end{array}
$$
As the limit is independent of the subsequence we have chosen, we obtain \eqref{eq:intermsofUN}, which implies \eqref{eq:decomposition}.

\vspace{0.2cm}

\textit{Step 3.} Now we are ready to pass to the limit in the Poisson equation \eqref{eq:poisson1}. As the dimension where the problem is posed grows, we consider test functions that depend on a finite number of variables. We take some $\varphi\in [C^1(\T^d)]^d$ and consider its trivial extension to  $\big(\T^d\big)^N$ to test the first $d$ equations in \eqref{eq:poisson1}:
\begin{equation}\label{eq:test1}
\int_{\big(\T^d\big)^N}M_N(x) \nabla \Psi^N(x): \nabla_{x_1}\varphi \;\dx{x}=\int_{\big(\T^d\big)^N} M_N(x)\nabla_{x_1}\cdot \varphi(x_1)\;\dx{x},    
\end{equation}
where $\nabla \Psi^N(x): \nabla_{x_1}\varphi$ denotes the inner product between matrices and we notice that $\nabla_{x_1} \varphi$ has non-trivial entries only for $1\le i,$ $j\le d$. Integrating the variables $x_2$ to $x_N$ in the right hand side of \eqref{eq:test1} we obtain
\begin{equation}\label{eq:test2}
\int_{\big(\T^d\big)^N} M_N(x)\nabla_{x_1}\cdot \varphi(x_1)\;\dx{x}=\int_{\T^d} (M_N)_1(x_1)\nabla_{x_1}\cdot \varphi(x_1)\;\dx{x}_1,
\end{equation}
where $(M_N)_1$ is the first marginal of $M_N$. 

For the left hand side of \eqref{eq:test1}, we notice that by \textit{Step 1.} \eqref{aprioriPsi} and \textit{Step 2.} \eqref{eq:decomposition} we can exchange $M_N$ in the integrand by the product $M_{N-1}(M_N)_1$
\begin{equation}\label{eq:test3}
\begin{array}{l}
\ds \left|\int_{\big(\T^d\big)^N}(M_N-M_1 M_{N-1}) \nabla_{x_1} \Psi^N: \nabla_{x_1}\varphi\;\dx{x}\right|\\
\qquad\ds\le \|\varphi\|_{C^1}\sum_{i=1}^d \int_{\big(\T^d\big)^N}\left(\frac{M_N}{M_{N-1}}-M_1 \right) |\nabla_{x_1} \Psi_i^N|M_{N-1}\;\dx{x}  \\
\qquad\ds\le \|\varphi\|_{C^1}\sum_{i=1}^d \|\nabla_{x_1} \Psi_i^N|\|_{L^2(M_{N-1})} \int_{\T^d}\left(\int_{\big(\T^d\big)^{N-1}} \left(\frac{M_N}{M_{N-1}}-(M_N)_1\right)^2M_{N-1}\;\dx{x}_2...\dx{x_N}\right)\;\dx{x}_1,\\
\qquad\ds=o(1)\stackrel{N\to\infty}{\to}0,
\end{array}
\end{equation}
where in the last equality we have used \eqref{eq:decomposition} and that 
$$
\sup_{x_1 \in \T^d}\left|\int_{\big(\T^d\big)^{N-1}} \left(\frac{M_N}{M_{N-1}}-(M_N)_1\right)^2M_{N-1}\;\dx{x}_2...\dx{x_N}\right|<C \, ,
$$
independently of $N$ to be able to apply Lebesgue dominated convergence to pass to the limit in the outer integral.

Hence, putting together \eqref{eq:test1}, \eqref{eq:test2} and \eqref{eq:test3} we obtain
\begin{equation}\label{eq:poisson2}
\int_{\T^d}(M_N)_1 \nabla_{x_1} \left(\int_{\big(\T^d\big)^{N-1}}\Psi^N M_{N-1}\;\dx{x}_2...\dx{x_N}\right): \nabla_{x_1} \varphi\;\dx{x}_1=-\int_{\T^d} (M_N)_1\nabla_{x_1}\cdot \varphi\;\dx{x}_1+o(1).
\end{equation}
To pass to the limit in \eqref{eq:poisson2}, we use of the a priori estimates which we proved in \textit{Step 1.}, \eqref{aprioriPsi} and \eqref{aprioriM_N}, which say that there exists $C>0$ such that for every $N\in \N$ and $i\le d$ we have  
$$
\norm*{\int_{\big(\T^d\big)^{N-1}}\Psi^N_i M_{N-1}\;\dx{x}_2...\dx{x_N}}_{H^1(\T^d)}\le C\qquad\mbox{and}\qquad \|(M_N)_1\|_{C^1(\T^d)}\le C.
$$
Note that we have used the Poincar\'e inequality on $\T^d$ to extend the gradient bound from~\eqref{aprioriPsi} to an $\SobH^1(\T^d)$ bound uniform in $N$. Moreover, by the chaoticity assumption $M_N\to \delta_{\tilde{\nu}^{\min}}$ (see Remark~\ref{rmk:chaoticity}) we can conclude that as $N\to\infty$ we have $(M_N)_1\to \tilde{\nu}^{\min}$ in $C^0(\T^d)$. Then, passing to the limit in \eqref{eq:poisson2} we can see that any weak-$H^1$ accumulation point $\Psi^{\min}\in [H^1(\T^d)]^d$ of   the function
$$
\left(\int_{\big(\T^d\big)^{N-1}}\Psi^N_1M_{N-1}\;\dx{x}_2...\dx{x_N},...,\int_{\big(\T^d\big)^{N-1}}\Psi^N_dM_{N-1}\;\dx{x}_2...\dx{x_N}\right) \, ,
$$
satisfies the equation
$$
\int_{\T^d} \tilde{\nu}^{\min} \nabla_{x_1}\Psi^{\min} : \nabla_{x_1}\varphi\;\dx{x}_1=-\int_{\T^d} \tilde{\nu}^{\min}\nabla_{x_1}\cdot \varphi\;\dx{x}_1
$$
with the condition
$$
\int_{\T^d} \Psi^{\min} \tilde{\nu}^{\min}\;\dx{x}_1=0,
$$
which follows from passing to the limit in \eqref{eq:zeroavg} in the same fashion as above. This uniquely determines the limit $\Psi^{\min}$. Therefore, up to exchanging the coordinates, we can pass to the limit in the diagonal of \eqref{eq:effAN}. That is to say for every set of indices $i$, $j$ satisfying $(k-1)d\le i,\;j\le kd$, we have
\begin{equation}\label{eq:limAeff1}
    A^{\eff,N}_{i,j}\to \int_{\T^d} \beta^{-1} \bra*{\delta_{i,j} + \partial_{\tilde{j}} \Psi_{\tilde{i}}^{\min}}\;\dx{\tilde{\nu}^{\min}},
\end{equation}
where $\tilde{i}$ and $\tilde{j}$ are respectively $i$ and $j$ modulo d. We also notice that using the a priori estimate \eqref{aprioriPsi}, we have that for every pair of indices $i$ and $j$ satisfying $(k_1-1)d\le i\le k_1d$ and $(k_2-1)d\le j\le k_2d$ with $k_1\ne k_2$
\begin{equation}\label{eq:limAeff2}
    |A^{\eff,N}_{i,j}|\le \frac{1}{N-1} \stackrel{N \to \infty}{\to}0 \,.
\end{equation}

\vspace{0.2cm}

\textit{Step 4.} Finally, we show that we can pass to the limit in the equation
\begin{equation}\label{eq:rhoN}
\pa_t \rho^{N,*}=\nabla\cdot (A^{\eff,N} \nabla\rho^{N,*})     \qquad\mbox{on $(0,\infty)\times \big(\R^d\big)^N$} \,.
\end{equation}
We consider a test function $\varphi\in C^2\big((\R^d)^{n}\big)$ and extend it trivially to $C^2\big((\R^d)^{N}\big)$. Testing \eqref{eq:rhoN} against $\varphi$ we obtain that for every $t>0$
\begin{equation}
    \int_{(\R^d)^{N}}\varphi \rho^{N}(t)\;\dx{x}-\int_{(\R^d)^{N}}\varphi \rho^{N}(0)\;\dx{x}=\int_0^t\int_{(\R^d)^{N}} \nabla\cdot (A^{\eff,N} \nabla\varphi)\rho^{N,*}(s)\;\dx{x}\dx{s}. 
\end{equation}
Next, we use that $\varphi$ only depends on the first $nd$ variables to obtain
\begin{equation}\label{eq:testn}
    \int_{(\R^d)^{n}}\varphi \rho_n^{N}(t)\;\dx{z}-\int_{(\R^d)^{n}}\varphi \rho_n^{N}(0)\;\dx{z}=\int_0^t\int_{(\R^d)^{n}} \nabla\cdot ([A^{\eff,N}]_{1\le i,j\le nd} \nabla\varphi)\rho_n^N(s)\;\dx{z}\dx{s}, 
\end{equation}
where $\rho_n^N\in\mathcal{P}((\R^d)^n)$ is the $n$-th marginal of $\rho^{N,*}$ and $[A^{\eff,N}]_{1\le i,j\le nd}\in \R^{nd\times nd}$ is the first $nd\times nd$ coordinates of $A^{\eff,N}$. Therefore, $\rho_n^N$ is a weak solution of
\begin{equation}
    \pa_t \rho^{N,*}_n=\nabla\cdot ([A^{\eff,N}]_{1\le i,j\le nd}\nabla\rho_n^N)     \qquad\mbox{on $(0,\infty)\times \big(\R^d\big)^n$.}
\end{equation}
By \textit{Step 3.}, we note that 
$$
[A^{\eff,N}]_{1\le i,j\le nd}\to [A^{\infty,\eff}_{\min}]_{1\le i,j\le nd}
$$
where $A^{\infty,\eff}_{\min}\in \R^{\infty\times\infty}$ denotes the matrix which if considered in $d\times d$ blocks is diagonal, which has the constant matrix
\begin{equation}
    A^{\eff}_{\min}=\beta^{-1}\int_{\T^d} ( I + \nabla \Psi^{\min})\;\dx{\tilde{\nu}^{\min}}\in \R^{d\times d},
\end{equation}
which is non-degenerate elliptic by Remark~\ref{rem:nondegenerate}. Therefore, for $N$ large enough $[A^{\eff,N}]_{1\le i,j\le nd}$ is uniformly elliptic and we can use standard parabolic techniques to obtain compactness of the curve $\rho^{N,*}_n$ in $C([0,T];\mathcal{P}((\R^d)^n))$. By \textit{Step 3.}, we can use \eqref{eq:limAeff1} and \eqref{eq:limAeff2} to pass to the limit \eqref{eq:testn} and obtain
\begin{equation}\label{eq:inftyn}
 \int_{(\R^d)^{n}}\varphi \rho_n^{\infty}(t)\;\dx{z}-\int_{(\R^d)^{n}}\varphi \rho_n^{\infty}(0)\;\dx{z}=\int_0^t\int_{(\R^d)^{n}} \nabla\cdot ([A^{\infty,\eff}_{\min}]_{1\le i,j\le nd} \nabla\varphi)\rho_n^\infty(s)\;\dx{z}\dx{s}.
\end{equation}

Equation \eqref{eq:inftyn} completely characterises $\lim_{N\to\infty}\rho^{N,*}_n$. In particular, we notice that
$$
\rho_n^\infty=(S_t^{\min}\#X_0)^{\otimes n},
$$
where $S_t^{\min}:\mathcal{P}(\R^d)\to\mathcal{P}(\R^d)$ is the solution semigroup associated to
$$
\partial_t\rho=\nabla\cdot(A^{\eff}_{\min}\nabla\rho)\qquad\mbox{on $(0,\infty)\times \R^d$.}
$$
As the marginals characterise $\lim_{N\to\infty}\rho^{N,*}(t)=X(t)\in\mathcal{P}(\mathcal{P}(\R^d))$ (cf. \cite[Lemma 3]{carrillo2019proof}), we obtain the desired result
$$
X(t)=S_t^{\min}\#X_0.
$$
Combining this with Theorem~\ref{diffusivelimit}, we have that, for a fixed $t>0$, the solution $\rho^{\eps,N}(t)$ of~\eqref{eq:LinearKolmogorov} satisfies
\begin{align}
\lim_{N \to \infty}\lim_{\eps \to 0} \rho^{\eps,N}(t)= \lim_{N \to \infty}\rho^{N,*}(t) = X(t)= S_t^{\min}\#X_0 \, .
\end{align}

\section{Proofs of Section~\ref{S:explicit}}\label{sec:phase}
In this section we include the proofs of Corollary~\ref{cor:2}, Corollary~\ref{cor:1} and Lemma~\ref{phase transition}.

\begin{proof}[Proof of Corollary~\ref{cor:2}]
The proof follows by combining Corollary~\ref{cor:beforephasetransition} and Lemma~\ref{Uniformexponentialconvergence}. Indeed, we can first apply Theorem~\ref{thm:variabledata} which gives us:
\begin{align}
\lim_{N \to \infty}\lim_{\eps \to 0} \rho^{\eps,N} = S_t^* \# X_0 \, .
\end{align}
However, from Lemma~\ref{Uniformexponentialconvergence}, we know that~\eqref{periodiceps=1} can have only one steady state. But Propositions~\ref{tfae} and~\ref{uniqueness} tell us that steady states must be minimisers and minimisers always exist. Thus, for $\beta \in (0,\beta_0]$, we have that $\tilde{\nu}^*=\tilde{\nu}^{\min}$, the unique minimiser. It follows that:
\begin{align}
\lim_{N \to \infty}\lim_{\eps \to 0} \rho^{\eps,N} = S_t^* \# X_0=S_t^{\min}\# X_0 \, .
\end{align}
The limit the other way around follows by applying Theorem~\ref{thm: N then eps} and using the fact that $\tilde{E}_{MF}$ has a unique minimiser.
\end{proof}

\begin{proof}[Proof of Corollary~\ref{cor:1}]
Since $\lim_{\eps \to 0} \eps^{-d}\rho_0(\eps^{-1}x)= \delta_0 \in \cP(\R^d)$, one can check that
\begin{align}
\lim_{N \to \infty} \lim_{\eps \to 0} \rho_0^{\eps,N}= \lim_{\eps \to 0}\lim_{N \to \infty} \rho_0^{\eps,N}=\delta_{\delta_0}=X_0 \, .
\end{align}
The proof of the limit $\eps \to 0$ followed by $N \to \infty$ follows by simply applying Theorem~\ref{thm: N then eps} and using the fact that for $\beta<\beta_c$, $\tilde{E}_{MF}$ has a unique minimiser $\tilde{\nu}^{\min}$(cf. Definition~\ref{pt} and Proposition~\ref{tfae}). 

For the other limit, since the initial data $\rho_0^{\eps,N}$ is rapidly varying, the corresponding initial data for~\eqref{eps=1} is precisely $\rho_0^{\otimes N}$ and is independent of $\eps>0$.  Thus we need to show that~\eqref{exponential convergence} holds for some fixed initial data $\tilde{\nu}_0 \in \cP(\T^d)$ independent of $\eps>0$. Here $\tilde{\nu}_0$ is the periodic rearrangement of $\rho_0$.
We will prove this by using the fact that if $\beta<\beta_c$, then~\eqref{periodiceps=1} has a unique steady state, namely $\tilde{\nu}_\infty \equiv 1 $. This follows simply from our definition of a phase transition and by plugging $\tilde{\nu}_\infty$
into the right hand side of~\eqref{periodiceps=1}. We divide the proof into two steps. In~\textit{Step 1}, we show that solutions of~\eqref{periodiceps=1} enjoy certain compactness properties and converge to $\tilde{\nu}_\infty$ along some time-divergent subsequences.  In~\textit{Step 2}, we will show that if $\tilde{\nu}_0$ is close to $\tilde{\nu}_\infty$ in an appropriate topology, then this convergence happens exponentially fast and along the whole trajectory, if $\beta<\beta_c$. Combining these together will then establish~\eqref{exponential convergence}.

\textit{Step 1.} By parabolic regularity theory, for any positive time $t>0$ the solution $\tilde{\nu}(t)$ of~\eqref{periodiceps=1} is smooth for any positive time. Thus we can assume without loss of generality that $\tilde{\nu}_0 \in C^\infty(\T^d)$. Furthermore, as discussed in Section~\ref{gradientflow},~\eqref{periodiceps=1} is a gradient flow of $\tilde{E}_{MF}$ with respect to the 2-Wasserstein metric $d_2$ on $\cP(\T^d)$. It follows from~\cite[Theorem 11.1.3]{ambrosio2008gradient}, that we have the following energy-dissipation identity along solutions of~\eqref{periodiceps=1}:
\begin{align}
\frac{\dx{}}{\dx{t}} \tilde{E}_{MF}[\tilde{\nu}(t)]=-D(\tilde{\nu}(t))=- \intt{\abs*{\nabla \log\frac{\tilde{\nu(t)}}{e^{-\beta W \ast \tilde{\nu}(t)}}}^2 \tilde{\nu}} \, .
\end{align}
Integrating from $0$ to $\infty$ and using the fact that the periodic mean field free energy $\tilde{E}_{MF}$ is bounded below,
we obtain:
\begin{align}
\int_0^\infty \intt{\abs*{\nabla \log\frac{\tilde{\nu(t)}}{e^{-\beta W \ast \tilde{\nu}(t)}}}^2 \tilde{\nu(t)}} \,  \dx{t} \leq C \, ,
\end{align}
for some constant $C>0$. Thus, there must exist a sequence of times $t_n \to \infty$ such that
\begin{align}
\lim_{n \to \infty} \intt{\abs*{\nabla \log\frac{\tilde{\nu(t_n)}}{e^{-\beta W \ast \tilde{\nu}(t_n)}}}^2 \tilde{\nu(t_n)}} =0 \, .
\end{align}
Since $\abs{\nabla W \ast \tilde{\nu}(t)} \leq \norm{\nabla W}_{\Leb^\infty(\T^d)}$, the above limit implies the following bound along the sequence $t_n$:
\begin{align}
\intt{\abs*{\nabla \sqrt{\tilde{\nu}(t_n)}}^2} = \intt{\abs*{\nabla \log \tilde{\nu}(t)}^2 \tilde{\nu}(t_n)}  \leq C \, .
\end{align}
Using the fact that $\tilde{\nu}(t) \in \cP(\T^d)$, we have that $\norm*{\sqrt{\tilde{\nu}(t_n)}}_{\SobH^1(\T^d)} \leq C$.  Thus, there exists a subsequence of times $t_{n_k}$ and a function $f \in \SobH^1(\T^d)$ such that
\begin{align}
\sqrt{\tilde{\nu}(t_{n_k})} \stackrel{k \to \infty}{\to} f \quad \textrm{strongly in }\Leb^2(\T^d), \textrm{ weakly in } \SobH^1(\T^d) \, .
\end{align}
 Furthermore, we have that
\begin{align}
\norm*{\tilde{\nu}(t_{n_k})- f^2}_{\Leb^1(\T^d)} &= \norm*{\bra*{\sqrt{\tilde{\nu}(t_{n_k})} + f }\bra*{\sqrt{\tilde{\nu}(t_{n_k})} - f}}_{\Leb^1(\T^d)}\\
& \leq \norm*{\sqrt{\tilde{\nu}(t_{n_k})} + f }_{\Leb^2(\T^d)} \norm*{\sqrt{\tilde{\nu}(t_{n_k})} - f }_{\Leb^2(\T^d)} \\&\leq   \norm*{ f }_{\Leb^2(\T^d)} \norm*{\sqrt{\tilde{\nu}(t_{n_k})} - f }_{\Leb^2(\T^d)} \stackrel{k \to \infty}{\to}0 \, .
\end{align}
Thus, $\intt{f^2}=1, f^2 \geq 0$, and thus $f^2 \in \cP(\T^d)$. One can also check that the dissipation is lower semicontinuous with respect to
$\Leb^1$ convergence. Thus 
\begin{align}
 \intt{\abs*{\nabla \log\frac{\tilde{f^2}}{e^{-\beta W \ast f^2}}}^2 f^2} \leq \liminf_{k \to \infty} \intt{\abs*{\nabla \log\frac{\tilde{\nu(t_{n_k})}}{e^{-\beta W \ast \tilde{\nu}(t_{n_k})}}}^2 \tilde{\nu(t_{n_k})}}=0 \,.
\end{align}
It follows then that $D(f^2)=0$ and from Proposition~\ref{tfae}, that $f^2 \in \cP(\T^d)$ is a steady state of~\eqref{periodiceps=1}. Since $\tilde{\nu}_\infty$ is the only stationary solution for $\beta<\beta_c$, it must hold that $f^2 = \tilde{\nu}_\infty$ and that
\begin{align}
\lim_{k \to \infty}\norm*{\tilde{\nu}_{t_{n_k}}-\tilde{\nu}_\infty}_{\Leb^1(\T^d)}=0 \, .
\end{align} 

\textit{Step 2.}We now use~\cite[Theorem 2.11]{CP10} which tells us that if $\beta<\beta_*:=-(\min_{k} \hat{W}(k))^{-1}$ and $\norm{\tilde{\nu}_0-\tilde{\nu}_\infty}_{\Leb^1\bra{\T^d}}< \eps_0$, then
\begin{align}
\norm{\tilde{\nu}(t)-\tilde{\nu}_\infty}_{\Leb^1\bra{\T^d}} \leq \norm{\tilde{\nu}_0-\tilde{\nu}_\infty}_{\Leb^1\bra{\T^d}}e^{-Ct} \, ,
\end{align}
for some $C>0$, $\eps_0>0$, and all $t \geq 0$.  Since we know from the previous step that $\lim_{k \to \infty}\norm{\tilde{\nu}(t_{n_k})-\tilde{\nu}_\infty} = 0$, there must exist some time $T>0$ such that $\norm{\tilde{\nu}(T)-\tilde{\nu}_\infty}_{\Leb^1\bra{\T^d}}< \eps_0$. We also know from~\cite[Proposition 5.3]{CGPS19} that $\beta_c\leq \beta_*$. Thus for all $\beta<\beta_c$, we have that 
\begin{align}
\norm{\tilde{\nu}(t)-\tilde{\nu}_\infty}_{\Leb^1\bra{\T^d}} \leq C_T e^{-C(t-T)} \, ,
\end{align}
where $C_T:= \max_{s \in [0,T]}\norm{\tilde{\nu}(s)-\tilde{\nu}_\infty}_{\Leb^1\bra{\T^d}} \leq 2$. Thus, we have shown that~\eqref{exponential convergence} holds, completing the proof of the first part of the result.

We remind the reader that $W \in \mathbf{H}_s$ means that $\hat{W}(k)\geq 0$ for all $k \in \Z^d$. For the second half of the result, we use the fact $W \in \mathbf{H}_s$ implies, by Proposition~\ref{expt}, that $\beta_c= +\infty$ and thus the result of the corollary necessarily holds for all $\beta<+\infty$ and rapidly varying initial data. We now sketch how to extend the result to all chaotic initial data. We note that by applying Duhamel's formula for the solution of~\eqref{periodiceps=1} one can show that there exists a time, say $t'=1>0$, such that for all initial data $\mathcal{H}(\tilde{\nu}(1)|\tilde{\nu}_\infty)<C$ for some fixed constant $C \geq 0$. Additionally, we can apply~\cite[Proposition 3.1]{CGPS19}, to assert that for $W \in \mathbf{H}_s$ and all $\beta<\infty$ , we have that 
\begin{align}
\mathcal{H}(\tilde{\nu}(t)|\tilde{\nu}_\infty) \leq \mathcal{H}(\tilde{\nu}(1)|\tilde{\nu}_\infty) e^{-C_1(t-1)} \, ,
\end{align}
for all $t \geq 1$. Since the relative entropy controls the $2$-Wasserstein distance, we can apply Corollary~\ref{cor:beforephasetransition} to complete the proof of the result.
\end{proof}

\begin{proof}[Proof of Lemma~\ref{phase transition}]
We know from Proposition~\ref{tfae} that steady states of the quotiented periodic system~\eqref{periodiceps=1} are equivalent to solutions of the self-consistency equation~\eqref{eq:critical point}, which we rewrite as
\begin{align}
    \tilde{\nu} =\frac{e^{-\beta (V + W\ast\tilde{\nu})}}{Z}, \quad Z= \intto{e^{-\beta (V + W\ast\tilde{\nu})}} \, .   
    \label{cp2}  
\end{align}
We also know from Proposition~\ref{uniqueness} that for $\beta$ sufficiently small the map in the above expression has a unique fixed point. Thus~\eqref{periodiceps=1} has a unique steady state for $\beta$ sufficiently small. Since minimisers of $\tilde{E}_{MF}$ exist and are always steady states (cf. Propositions~\ref{uniqueness} and~\ref{tfae}), it must also be the unique minimiser of $\tilde{E}_{MF}$. We argue further that any minimiser of $\tilde{E}_{MF}$ must be symmetric about $x=1/2$ and decreasing from $0$ to $1/2$. This follows directly from the Baernstein--Taylor inequality for spherical rearrangements of functions~\cite{BT76}. 

To investigate the problem ahead of the phase transition, we  consider ~\eqref{cp2}. Plugging our choice of $V$ and $W$ and testing against $\cos(2 \pi x)$ we can simplify this to
\begin{align}
\tilde{\nu}_1 =Z^{-1} \intt{\cos(2 \pi x) \exp\bra*{\beta\bra*{\cos(2 \pi x)(\eta + \tilde{\nu}_1 ) + \sin(2 \pi x) \tilde{\nu}_{-1}}}} \, , 
\end{align}
where $\tilde{\nu}_1= \skp{\tilde{\nu}, \cos(2 \pi x)}$ and $\tilde{\nu}_{-1}= \skp{\tilde{\nu}, \sin(2 \pi x)}$. Let us consider the problem when $\tilde{\nu}_{-1}=0$, as this corresponds to the setting when $\tilde{\nu}$ is symmetric about $x=1/2$. Simplifying further we obtain:
\begin{align}
\tilde{\nu}_1 =Z^{-1} \intt{\cos(2 \pi x) \exp\bra*{\beta\bra*{\cos(2 \pi x)(\eta + \tilde{\nu}_1 ) }}} \, . 
\end{align}
Using the fact the modified Bessel functions of the first kind can be expressed as $I_n(y)=\intto{\cos(2 \pi n x)e^{y \cos(2 \pi x)}}$, we obtain:
\begin{align}
\tilde{\nu}_1=r_0\bra*{\beta(\eta+ \tilde{\nu}_1)}
\end{align}
where $r_0(x):=I_1(x)/I_0(x)$, and $I_1$,$I_0$ are first and zeroth modified Bessel functions of the first kind. Setting $\beta(\eta + \tilde{\nu}_1)=a$ we simplify the above expression to
\begin{align}
a= \beta(\eta + r_0(a)) \, .
\label{simplified}
\end{align}
The function $r_0(a)$  has the following properties~\cite[Proposition 6.1]{CGPS19}:
\begin{align}
r_0(0)&=0 \\
\lim_{a \to \infty} r_0(a)&= 1  \quad \lim_{a \to -\infty} r_0(a)=-1 \label{r0-}\\
r_0''(a)&<0 , \quad a>0 \label{r''}
\end{align}
Note that only solutions of~\eqref{simplified} with $a\geq0$ can be minimisers of the free energy, as for $a < 0$ the solutions are increasing from $0$ to $1/2$. We argue now that~\eqref{simplified} has exactly one solution for $a>0$, for all $\beta>0$.
Consider the function $F: \R \to \R$ defined as follows
\begin{align}
F(a):= \beta(\eta + r_0(a)) -a \, .
\end{align}
We know that $F(0)= \beta \eta >0$. Furthermore for $a$ large enough and positive we have that $F(a)<0$, using~\eqref{r0-}. Thus, by the intermediate value theorem, for every fixed  $\beta>0$, we can find at least one $a^{\min}>0$ such that $F(a^{\min})=0$. Now if $F(a^{\min})=0$ for some $a^{\min}>0$, we must have that $\beta r_0'(a^{\min})<1$. If not, we would have that
\begin{align}
F(a^{\min})&= \beta \eta + \int_0^{a^{\min}}(\beta r_0'(a)-1) \dx{a} \\
& \geq \beta \eta + a^{\min}(\beta r_0'(a^{\min})-1) >0 \, ,
\end{align} 
which is a contradiction. In the last inequality we have used~\eqref{r''}. This implies that
\begin{align}
F'(a^{\min})=\beta r_0'(a^{\min})-1 <0 \, .
\end{align}
 Also
\begin{align}
F''(a)=\beta r_0''(a) <0 \, .
\end{align}
Thus once $F'(a)<0$ it remains negative for all $a>0$. It follows that $F(a^{\min})=0$ for only one $a^{\min}>0$. Since this is the only symmetric decreasing solution of~\eqref{eq:critical point} it corresponds to the unique minimiser of $\tilde{E}_{MF}$ through the expression in~\eqref{min1}. It is also must be the unique steady state obtained using the contraction argument earlier in the proof.

We will now show that for $\beta$ large enough we can find another solution of~\eqref{simplified} for $a<0$. Let $\eta=1-\delta$ for some $\delta \in (0,1)$. From~\eqref{r0-} we know that there exists some $a'<0 $ such that for all $a \leq a'$, $r_0(a)< -1+ \delta/2$. We then have that
\begin{align}
F(a')&= \beta -\beta \delta + \beta r_0(a') -a'
\\
& <-\beta \frac{\delta}{2} -a' \, . 
\end{align}
Furthermore , since $r_0(a)$ is an odd function and $\eta>0$, if $a^{\min}$ is a solution of~\eqref{simplified}, then $-a^{\min}$
cannot be a solution. It follows that $a^* \neq -a^{\min}$.
\end{proof}

\paragraph{\bf Acknowledgements:} MGD was partially supported by EPSRC grant number EP/P031587/1. RSG is funded by an Imperial College President's PhD Scholarship, partially through EPSRC Award Ref. 1676118. Part of this work was carried out at the {\emph ``Junior Trimester Programme in Kinetic Theory''} held at the Hausdorff Research Institute for Mathematics, Bonn. RSG is grateful to the institute for its hospitality. GAP was partially supported by the EPSRC through grant numbers EP/P031587/1, EP/L024926/1, and EP/L020564/1. This research was funded in part by JPMorgan Chase \& Co. Any views or opinions expressed herein are solely those of the authors listed, and may differ from the views and opinions expressed by JPMorgan Chase \& Co. or its affiliates. This material is not a product of the Research Department of J.P. Morgan Securities LLC. This material does not constitute a solicitation or offer in any jurisdiction. The authors would like to thank Martin Hairer and Felix Otto for useful discussions during the course of this work.

\begin{appendix}
\section{Coupling arguments}\label{ap:coupling}
In this section we will use coupling techniques introduced by Eberle and co-authors~\cite{Ebe11,Ebe16,EGZ19,DEGZ18} to show some necessary results for our proofs. Following the previous strategy we construct a new metric which is equivalent to the Wasserstein metric. We define the constant
\begin{align}
\kappa := \inf_{x\in\T} V''(x)+ \inf_{x\in\T}W''(x)\le 0 \, ,
\end{align} 
which gives a lower bound of the semi-convexity of the function of $V$ and $V+W\ast \tilde{\nu}^*$ on $\T$. Next, we define the following functions on $[0,1/2]$:
\begin{align}\label{eq:deff}
\psi(r)&:= \exp \bra*{ \frac{\beta\kappa r^2}{8} }  , && \Phi(r):= \int_0^r \psi(s) \dx{s} \, , \\
g(r)&:=1- \frac{c}{2}\int_0^r \Phi(s) \bra*{\psi(s)}^{-1} \dx{s}, && c:=\bra*{\int_0^{1/2} \Phi(s) \bra*{\psi(s)}^{-1} \dx{s}}^{-1}\ge \frac{\beta|\kappa|}{4\left(e^{\frac{\beta|\kappa|}{32}}-1\right)}\,.
\end{align}
We note that $g(r) \in [1/2,1]$ for all $r \in [0,1/2]$ and $\lim_{\beta\to0^+} c=1/8$. Additionally, both $\psi$ and $g$ are decreasing functions of $r$. Thus the  function $f:[0,1/2] \to [0,1/2]$ defined as
\begin{align}
f(r)&:= \int_0^r g(s)\psi(s) \dx{s} \, ,
\end{align}
is increasing and subadditive. Furthermore, we have the bounds
\begin{align}
\frac{\psi(1/2)}{2}r \leq f(r) \leq \Phi(r) \leq r \, .\label{dbounds}
\end{align}
Thus $d_f(x,y) := f(d_\T(x,y))$ defines a metric on $\T$ which is equivalent to $d_{\T}$. The main point of this construction is to obtain the following inequality
\begin{align}
f''(r) - \beta r \kappa \frac{f'(r)}{4 } \leq -\frac{c}{2} f(r)\qquad\mbox{for all $r \in [0,1/2]$.}
\label{contraction inequality}
\end{align}
This easily follows from the following computation:
\begin{align}
f''(r)- \beta r \kappa \frac{f'(r)}{4} &= \beta r\frac{\kappa}{4} f'(r) -\frac{c}{2}\Phi(r) - \beta r \kappa \frac{f'(r)}{4}=-\frac{c}{2}\Phi(r) \leq -\frac{c}{2} f(r) \, . 
\end{align}

Moreover, we define the Lipschitz functions $\varphi_r^\delta,\varphi_s^\delta: \R \to \R$ for some $\delta>0$, such that
\begin{align}\label{eq:defphirdelta}
(\varphi_r^\delta)^2(x) + (\varphi_s^\delta)^2(x)=1 \qquad \varphi_r^\delta(x) =
\begin{cases}
0  & \gamma(\abs{x}) \leq \delta/2 \\
1  & \gamma(\abs{x}) > \delta
\end{cases}
\end{align}
where the function $\gamma:\R_+ \to [0,1/2]$ maps Euclidean distances to distances on the torus
\begin{align}
\gamma(\abs{x}) := 
\begin{cases}
(\abs{x}\mod 1) & (\abs{x} \mod 1) \leq 1/2 \\
1- (\abs{x} \mod 1) & \textrm{otherwise} 
\end{cases} 
\, .
\end{align}
The introduction of the function $\gamma$ to account for the periodic setting is the main difference with the results in the literature \cite{Ebe11,Ebe16,EGZ19,DEGZ18}.

We have the following result:
\begin{lemma}\label{coupling existence}
Assume that~\eqref{exponential convergence} holds and consider the two SDEs in~\eqref{eq:mckeanSDEq} and~\eqref{eq:mckeanSDEs}. Then there exists a coupling of $(\dot{X}_t,\dot{Y}_t)$ and a metric
$d_f$ on $\T^d$ which is equivalent to $d_{\T^d}$ such that
\begin{align}
\sup_{\eps>0}\mathbb{E} \pra*{ d_{f}(\dot{X}_t,\dot{Y}_t)^2} \to 0
\end{align}
as $t \to \infty$.
\end{lemma}
\begin{proof} For convenience we write the proofs in 1 space dimension. The generalization to higher dimensions follows along similar lines. 
We start the proof by using the metric $d_f$ defined previously. We now proceed to construct the coupling between the two processes by considering the corresponding processes on $\R$, i.e. $Y_t$ and $X_t$. We assume that $\mathrm{Law}(X_0)=\nu^* \in \cP(\R)$ such that the
periodic rearrangement of $\nu^*$ is precisely $\tilde{\nu}^*$.   

Let $B_t^1$ and $B_t^2$ be two independent standard Wiener processes which are also independent of the initial conditions. We then couple the processes in a similar manner to~\cite{EGZ19} as follows
\begin{align}
&\begin{cases} 
\dx{Y}_t &= - V'(Y_t)\dx{t} - W' \ast \tilde{\nu}^\eps (t)(Y_t)\dx{t} + \sqrt{2 \beta^{-1}} \bra*{\varphi_r^\delta(E_t)dB^1_t +\varphi_s^\delta(E_t)dB_t^2} \\
\mathrm{Law}(Y_0)&= \nu_0^\eps \in \cP(\R) \, ,
\end{cases} \label{sde1}\\
&\begin{cases} 
\dx{X}_t &= -V'(X_t)\dx{t} - W' \ast \tilde{\nu}^* (X_t)\dx{t} + \sqrt{2 \beta^{-1}} \bra*{-\varphi_r^\delta(E_t)dB^1_t +\varphi_s^\delta(E_t)dB_t^2} \\
\mathrm{Law}(X_0)&= \nu^* \in \cP(\R) \, .
\end{cases}
\label{sde2}
\end{align}
where $E_t:= Y_t-X_t$ and $X_0,Y_0$ are independent. This above coupling corresponds to a combination of reflection and synchronous coupling. Note that we have suppressed the dependence on $\delta$ for the sake of notational convenience. However, in the limit as $\delta \to 0_+$ the processes $X_t$ and $Y_t$ converge $\mathbb{P}$-a.s. to corresponding limits with only reflection coupling. We also define the following function
\begin{align}
e_t:=
\begin{cases}
\frac{E_t}{\abs{E_t}} & \abs{E_t}>0 \\
0 & \textrm{otherwise} 
\end{cases}
\, .
\end{align}
Subtracting ~\eqref{sde2} from~\eqref{sde1} and using the same arguments as in~\cite{EGZ19} we obtain
\begin{align}
\dx{\abs{E_t}}&=-\bra*{ V'(Y_t) + W' \ast \tilde{\nu}^\eps (t)(Y_t) -V'(X_t) - W' \ast \tilde{\nu}^* (X_t)} (e_t) \dx{t} \\&+ 2 \sqrt{2 \beta^{-1}} \varphi_r^\delta(E_t) e_t dB_t^1 \, .
\end{align} 
Note now that the function $\R_+ \ni x \to \gamma(x)$ is a function whose derivatives are of locally bounded variation. Thus it can be expressed as the difference of two convex functions~\cite[Theorem (I)]{Har59}. We can thus apply the Meyer--Tanaka formula~\cite[Theorem 6.22]{KS91} to it, to obtain
\begin{align}
 \gamma(\abs{E_t}) &=  \gamma(\abs{E_0})-\int_0^t\gamma_\ell'(\abs{E_s})\bra*{ V'(Y_s) - W' \ast \tilde{\nu}^\eps (s)(Y_s) -V'(X_s) - W' \ast \tilde{\nu}^* (X_s)} (e_s) \dx{s} \\&+\int_0^t\gamma_\ell'(\abs{E_s}) 2 \sqrt{2 \beta^{-1}} \varphi_r^\delta(E_s) e_s dB_s^1+\int_{\R_+} \Lambda_t(a) \dx{\gamma''_-}(a) \, ,
\end{align}
where $\gamma'_\ell$ is the left derivative of $\gamma$, $\Lambda_t$ is the local time of the process $\abs{E_t}$, and $\gamma''_-$ is the negative part $\gamma''$ the distributional derivative of $\gamma$. We can throw away the positive part as $\varphi^\delta_r(0)=0$. The reader should note that $\gamma(\abs{E_t})=d_\T(\dot{X}_t,\dot{Y}_t)$, i.e. it is distance on the torus between the quotiented processes. Since the local time is an adapted non-decreasing continuous process it follows that $A_t:=\int_{\R_+} \Lambda_t(a) \dx{\gamma''_-}(a)$ is an adapted nonincreasing continuous process. 

Since $\gamma_t:=\gamma(\abs{E_t})$ is a continuous semimartingale 
we can apply Ito's formula to $f(\gamma_t)$ to obtain
\begin{align}
\dx{f}(\gamma_t) &= -f'(\gamma_t)\gamma_\ell'(\abs{E_t})\bra*{ V'(Y_t) + W' \ast \tilde{\nu}^\eps (t)(Y_t) -V'(X_t) - W' \ast \tilde{\nu}^* (X_t)} (e_t) \dx{t}  \\
&+ f'(\gamma_t) dA_t  +\gamma_\ell'(\abs{E_t}) 2 \sqrt{2 \beta^{-1}} \varphi_r^\delta(E_t) e_t dB_t^1 \\
&+ 4 f''(\gamma_t) \beta^{-1} (\varphi_r^\delta(E_t))^2 (\gamma_\ell'(\abs{E_t}))^2 \dx{t}
\end{align}
Next, we note that since $f'(x)\geq 0$ and $A_t$ is nonincreasing we have the bound 
\begin{align}
 \dx{f}(\gamma_t) &\leq -f'(\gamma_t)\gamma_\ell'(\abs{E_t})\bra*{ V'(Y_t) + W' \ast \tilde{\nu}^* (Y_t) -V'(X_t) - W' \ast \tilde{\nu}^* (X_t)} (e_t) \dx{t}  \\
&+ f'(\gamma_t)\gamma_\ell'(\abs{E_t})\bra*{ -W' \ast \tilde{\nu}^\eps(t) (Y_t)+W' \ast \tilde{\nu}^* (Y_t)}(e_t)\\
&  +\gamma_\ell'(\abs{E_t}) 2 \sqrt{2 \beta^{-1}} \varphi_r^\delta(E_t) e_t dB_t^1 \\
&+ 4 f''(\gamma_t)\beta^{-1} (\varphi_r^\delta(E_t))^2  \dx{t} \,  \\
&\leq -f'(\gamma_t)\gamma_\ell'(\abs{E_t})\bra*{ V'(Y_t) + W' \ast \tilde{\nu}^* (Y_t) -V'(X_t) - W' \ast \tilde{\nu}^* (X_t)} (e_t) \dx{t}  \label{ineq 2} \\
&+\norm{f'}_{\Leb^\infty(\T)}\|\gamma_l'\|_{\Leb^\infty(\T)}\norm{W''}_{\Leb^\infty(\T)} d_2(\tilde{\nu}^\eps(t),\tilde{\nu}^*) \dx{t}+\gamma_\ell'(\abs{E_t}) 2 \sqrt{2 \beta^{-1}} \varphi_r^\delta(E_t) e_t dB_t^1 \\
&+ 4 f''(\gamma_t)\beta^{-1} (\varphi_r^\delta(E_t))^2  \dx{t},
\end{align} 
where in the second inequality we have used the dual formulation of the $1$-Wasserstein distance. Consider now the $1$-periodic function $F:= V + W \ast \tilde{\nu}^*$, using the definition of $\gamma$ and $\kappa$ we have the inequality
\begin{align}
\gamma_\ell'(\abs{E_t})\bra*{ F'(Y_t) -F'(X_t)} (e_t) &\geq \kappa d_{\T}(\dot{X}_t,\dot{Y}_t)= \kappa \gamma_t \, .
\end{align}

\noindent Applying this
estimate to~\eqref{ineq 2} and using the fact that $f'>0$ we obtain
\begin{align}
\dx{f}(\gamma_t)  \leq& -\kappa f'(\gamma_t)\gamma_t \dx{t}+ 4 f''(\gamma_t)\beta^{-1} (\varphi_r^\delta(E_t))^2  \dx{t} \\&+\norm{W''}_{\Leb^\infty(\T)} d_1(\tilde{\nu}^\eps(t),\tilde{\nu}^*) \dx{t}+\gamma_\ell'(\abs{E_t}) 2 \sqrt{2 \beta^{-1}} \varphi_r^\delta(E_t) e_t dB_t^1 \, . \\
=&-\kappa f'(\gamma_t)\gamma_t (\varphi_r^\delta(E_t))^2 \dx{t}+ 4 f''(\gamma_t)\beta^{-1} (\varphi_r^\delta(E_t))^2  \dx{t}  + \kappa f'(\gamma_t)\gamma_t ((\varphi_r^\delta(E_t))^2-1) \dx{t}\\
& +\norm{W''}_{\Leb^\infty(\T)} d_1(\tilde{\nu}^\eps(t),\tilde{\nu}^*)\dx{t} +\gamma_\ell'(\abs{E_t}) 2 \sqrt{2 \beta^{-1}} \varphi_r^\delta(E_t) e_t dB_t^1 \, . 
\end{align}
Applying the differential inequality for $f$~\eqref{contraction inequality}, $f' \leq1$, and the definition of $\varphi_r^\delta$ \eqref{eq:defphirdelta} we obtain
\begin{align}
\dx{f}(\gamma_t) & \leq- 2c\beta^{-1}f(\gamma_t)(\varphi_r^\delta(E_t))^2 \dx{t}+ \frac{\abs{\kappa}}{2}\delta \dx{t}+\norm{W''}_{\Leb^\infty(\T)} d_1(\tilde{\nu}^\eps(t),\tilde{\nu}^*)\dx{t} \\&+\gamma_\ell'(\abs{E_t}) 2 \sqrt{2 \beta^{-1}} \varphi_r^\delta(E_t) e_t dB_t^1  \\
& \leq - 2c\beta^{-1}f(\gamma_t) \dx{t}+ 2c\beta^{-1}f(\delta) \dx{t}+ \frac{\abs{\kappa}}{2}\delta\dx{t} +\norm{W''}_{\Leb^\infty(\T)} d_1(\tilde{\nu}^\eps(t),\tilde{\nu}^*)\dx{t} \\&+\gamma_\ell'(\abs{E_t}) 2 \sqrt{2 \beta^{-1}} \varphi_r^\delta(E_t) e_t dB_t^1  \, .
\end{align}
Taking the expectation of the above expression and passing to the limit as $\delta \to 0^+$, we obtain
\begin{align}
\frac{\dx{}}{\dx{t}} \mathbb{E}\pra*{f(\gamma_t)} \leq - 2c\beta^{-1}\mathbb{E}\pra*{f(\gamma_t)} + \norm{W''}_{\Leb^\infty(\T)} d_1(\tilde{\nu}^\eps(t),\tilde{\nu}^*) \, .
\end{align}
It follows by Gronwall's Lemma that
\begin{align}\label{eq:ineqaux1}
\mathbb{E}\pra*{f(\gamma_t)} \leq e^{-2 c \beta^{-1}t } \mathbb{E}\pra*{f(\gamma_0)} +\norm{W''}_{\Leb^\infty(\T)} e^{-2 c \beta^{-1}t }  \int_0^t e^{2 c \beta^{-1}s } d_1(\tilde{\nu}(t),\tilde{\nu}^{\min}) \dx{s} \, .
\end{align}
Applying~\eqref{exponential convergence}, we have that $\mathbb{E}\pra*{d_{\T}(\dot{X}_t,\dot{Y}_t)}=\mathbb{E}\pra*{f(\gamma_t)} \to 0$ as $t \to \infty$, $\dot{X}_t$ and $\dot{Y}_t$ are the quotiented processes obtained in the limit as $\delta \to 0_+$. This completes the proof.
\end{proof}

\begin{lemma}\label{Uniformexponentialconvergence}
Given $V$ and $W$, there exists an explicit $\beta_0$ depending $V$ and $W$ such that for 
$$
\beta\le \beta_0
$$
there exists a unique minimiser and critical point, $\tilde{\nu}^{\min} \in \cP(\T^d)$, of the periodic mean field energy \eqref{periodicmeanfieldenergy} and $C_2$ depending on $\beta$, $W$ and $V$ such that
$$
d_2^2(\tilde{\nu}(t),\tilde{\nu}^{\min})\le e^{-C_2 t},
$$
where $\nu(t)$ is the solution to the periodic McKean--Vlasov equation \eqref{periodiceps=1} with arbitrary initial data $\tilde{\nu}_0 \in \cP(\T^d)$.
\end{lemma}
\begin{proof}
As done previously, we state the proof in $1$ dimension for the sake of simplicity. Clearly for $\beta$ small enough, by Proposition~\ref{uniqueness}, the periodic mean field energy has a unique minimiser, $\tilde{\nu}^{\min}$. Similar to the proof of \eqref{coupling existence}, we consider the processes on $\R$
\begin{align}
&\begin{cases} 
\dx{Y}_t &= - V'(Y_t)\dx{t} - W' \ast \tilde{\nu} (t)(Y_t)\dx{t} + \sqrt{2 \beta^{-1}} \bra*{\varphi_r^\delta(E_t)dB^1_t +\varphi_s^\delta(E_t)dB_t^2} \\
\mathrm{Law}(Y_0)&= \nu_0 \in \cP(\R) \, ,
\end{cases} \label{sde11}\\
&\begin{cases} 
\dx{X}_t &= -V'(X_t)\dx{t} - W' \ast \tilde{\nu}^{\min} (X_t)\dx{t} + \sqrt{2 \beta^{-1}} \bra*{-\varphi_r^\delta(E_t)dB^1_t +\varphi_s^\delta(E_t)dB_t^2} \\
\mathrm{Law}(X_0)&= \nu^{\min} \in \cP(\R) \, .
\end{cases} \label{sde21}
\end{align}
such that $\tilde{\nu}_0,\tilde{\nu}^{\min}$ are the periodic rearrangements of $\nu_0,\nu^{\min}$, respectively. We obtain the inequality \eqref{eq:ineqaux1}
\begin{equation}
    \mathbb{E}\pra*{f(\gamma_t)} \leq e^{-2 c \beta^{-1}t } \mathbb{E}\pra*{f(\gamma_0)} +\norm{W''}_{\Leb^\infty(\T)} e^{-2 c \beta^{-1}t }  \int_0^t e^{2 c \beta^{-1}s } d_1(\tilde{\nu}(s),\tilde{\nu}^{\min}) \dx{s},
\end{equation}
where $\gamma_t=d_\mathbb{T}(\dot{X}_t,\dot{Y}_t)$. Using the fact that $\mathrm{Law} (\dot{Y}_t)=\tilde{\nu}(t)$ and $\mathrm{Law}(\dot{X}_t)=\tilde{\nu}^{\min}$ and applying the bounds from~\eqref{dbounds}, we obtain
$$
\frac{e^{\frac{\beta k}{32}}}{2}d_1(\tilde{\nu}(t),\tilde{\nu}^{\min}) \le d_f(\tilde{\nu}(t),\tilde{\nu}^{\min})\le \mathbb{E}\pra*{f(\gamma_t)}.
$$
Combining this with the previous identity and applying the integral version of Gronwall's lemma we obtain
$$
d_1(\tilde{\nu}(t),\tilde{\nu}^{\min})\le e^{-2 c \beta^{-1}t } + t e^{-t\left(\beta^{-1}c- e^{-\frac{\beta k}{32}\norm{W''}_{\Leb^\infty(\T)}}\right)}
$$
Using the lower bound for $c$ \eqref{eq:deff}, we have the following: if
$$
\frac{|\kappa|}{4e^{\frac{\beta|\kappa|}{32}}\left(e^{\frac{\beta|\kappa|}{32}}-1\right)}\ge \norm{W''}_{\Leb^\infty(\T)},
$$
then there exists $C_2>0$ such that
$$
d_1(\tilde{\nu}(t),\tilde{\nu}^{\min})\le e^{-C_2t}.
$$
By making $\beta$ smaller than some $\beta_0$, this can be achieved.

\end{proof}

\end{appendix}

 \bibliographystyle{abbrv}
 \bibliography{biblio}

\begin{thebibliography}{10}

\bibitem{ambrosio2008gradient}
L.~Ambrosio, N.~Gigli, and G.~Savar\'{e}.
\newblock {\em Gradient flows in metric spaces and in the space of probability
  measures}.
\newblock Lectures in Mathematics ETH Z\"{u}rich. Birkh\"{a}user Verlag, Basel,
  2005.

\bibitem{BT76}
A.~Baernstein, II and B.~A. Taylor.
\newblock Spherical rearrangements, subharmonic functions, and
  {$\sp*$}-functions in {$n$}-space.
\newblock {\em Duke Math. J.}, 43(2):245--268, 1976.

\bibitem{BBC19}
J.~{Barr{\'e}}, C.~{Bernardin}, R.~{Ch{\'e}trite}, Y.~{Chopra}, and
  M.~{Mariani}.
\newblock {Gamma Convergence Approach For The Large Deviations Of The Density
  In Systems Of Interacting Diffusion Processes}.
\newblock {\em arXiv e-prints}, page arXiv:1910.04026, Oct 2019.

\bibitem{BLP11}
A.~Bensoussan, J.-L. Lions, and G.~Papanicolaou.
\newblock {\em Asymptotic analysis for periodic structures}.
\newblock AMS Chelsea Publishing, Providence, RI, 2011.
\newblock Corrected reprint of the 1978 original [MR0503330].

\bibitem{BO2018}
R.~J. Berman and M.~\"{O}nnheim.
\newblock Propagation of chaos, {W}asserstein gradient flows and toric
  {K}\"{a}hler-{E}instein metrics.
\newblock {\em Anal. PDE}, 11(6):1343--1380, 2018.

\bibitem{BGP10}
L.~Bertini, G.~Giacomin, and K.~Pakdaman.
\newblock Dynamical aspects of mean field plane rotators and the {K}uramoto
  model.
\newblock {\em J. Stat. Phys.}, 138(1-3):270--290, 2010.

\bibitem{cardialaguetnotes}
P.~Cardialaguet.
\newblock Notes on mean-field games (from pl. lions' lectures at college de
  france), 2013.

\bibitem{carrillo2019proof}
J.~A. {Carrillo}, M.~G. {Delgadino}, and G.~A. {Pavliotis}.
\newblock {A proof of the mean-field limit for $\lambda$-convex potentials by
  $\Gamma$-Convergence}.
\newblock {\em arXiv e-prints}, page arXiv:1906.04601, Jun 2019.

\bibitem{CGPS19}
J.~A. Carrillo, R.~S. Gvalani, G.~A. Pavliotis, and A.~Schlichting.
\newblock Long-time behaviour and phase transitions for the {M}ckean--{V}lasov
  equation on the torus.
\newblock {\em Arch. Ration. Mech. Anal.}, Jul 2019.

\bibitem{CP10}
L.~Chayes and V.~Panferov.
\newblock The {M}c{K}ean-{V}lasov equation in finite volume.
\newblock {\em J. Stat. Phys.}, 138(1-3):351--380, 2010.

\bibitem{dawson1983critical}
D.~A. Dawson.
\newblock Critical dynamics and fluctuations for a mean-field model of
  cooperative behavior.
\newblock {\em J. Statist. Phys.}, 31(1):29--85, 1983.

\bibitem{deFinetti}
B.~de~Finetti.
\newblock La pr\'{e}vision : ses lois logiques, ses sources subjectives.
\newblock {\em Ann. Inst. H. Poincar\'{e}}, 7(1):1--68, 1937.

\bibitem{DFGW89}
A.~De~Masi, P.~A. Ferrari, S.~Goldstein, and W.~D. Wick.
\newblock An invariance principle for reversible {M}arkov processes.
  {A}pplications to random motions in random environments.
\newblock {\em J. Statist. Phys.}, 55(3-4):787--855, 1989.

\bibitem{diaconis1980finite}
P.~Diaconis and D.~Freedman.
\newblock Finite exchangeable sequences.
\newblock {\em Ann. Probab.}, 8(4):745--764, 1980.

\bibitem{DEGZ18}
A.~{Durmus}, A.~{Eberle}, A.~{Guillin}, and R.~{Zimmer}.
\newblock {An Elementary Approach To Uniform In Time Propagation Of Chaos}.
\newblock {\em arXiv e-prints}, page arXiv:1805.11387, May 2018.

\bibitem{durmus2018elementary}
A.~Durmus, A.~Eberle, A.~Guillin, and R.~Zimmer.
\newblock An elementary approach to uniform in time propagation of chaos.
\newblock {\em arXiv preprint arXiv:1805.11387}, 2018.

\bibitem{Ebe11}
A.~Eberle.
\newblock Reflection coupling and {W}asserstein contractivity without
  convexity.
\newblock {\em C. R. Math. Acad. Sci. Paris}, 349(19-20):1101--1104, 2011.

\bibitem{Ebe16}
A.~Eberle.
\newblock Reflection couplings and contraction rates for diffusions.
\newblock {\em Probab. Theory Related Fields}, 166(3-4):851--886, 2016.

\bibitem{EGZ19}
A.~Eberle, A.~Guillin, and R.~Zimmer.
\newblock Couplings and quantitative contraction rates for {L}angevin dynamics.
\newblock {\em Ann. Probab.}, 47(4):1982--2010, 2019.

\bibitem{fernandez1997hilbertian}
B.~Fernandez and S.~M\'{e}l\'{e}ard.
\newblock A {H}ilbertian approach for fluctuations on the {M}c{K}ean-{V}lasov
  model.
\newblock {\em Stochastic Process. Appl.}, 71(1):33--53, 1997.

\bibitem{FJ17}
N.~Fournier and B.~Jourdain.
\newblock Stochastic particle approximation of the {K}eller-{S}egel equation
  and two-dimensional generalization of {B}essel processes.
\newblock {\em Ann. Appl. Probab.}, 27(5):2807--2861, 2017.

\bibitem{FV2018}
S.~Friedli and Y.~Velenik.
\newblock {\em Statistical mechanics of lattice systems}.
\newblock Cambridge University Press, Cambridge, 2018.
\newblock A concrete mathematical introduction.

\bibitem{GPW17}
J.~Garnier, G.~Papanicolaou, and T.-W. Yang.
\newblock Consensus convergence with stochastic effects.
\newblock {\em Vietnam J. Math.}, 45(1-2):51--75, 2017.

\bibitem{Gol13}
F.~{Golse}.
\newblock {On the Dynamics of Large Particle Systems in the Mean Field Limit}.
\newblock {\em arXiv e-prints}, page arXiv:1301.5494, Jan 2013.

\bibitem{GP18}
S.~N. Gomes and G.~A. Pavliotis.
\newblock Mean field limits for interacting diffusions in a two-scale
  potential.
\newblock {\em J. Nonlinear Sci.}, 28(3):905--941, 2018.

\bibitem{Har59}
P.~Hartman.
\newblock On functions representable as a difference of convex functions.
\newblock {\em Pacific J. Math.}, 9:707--713, 1959.

\bibitem{hauray2014kac}
M.~Hauray and S.~Mischler.
\newblock On {K}ac's chaos and related problems.
\newblock {\em J. Funct. Anal.}, 266(10):6055--6157, 2014.

\bibitem{HewittSavage}
E.~Hewitt and L.~J. Savage.
\newblock Symmetric measures on {C}artesian products.
\newblock {\em Trans. Amer. Math. Soc.}, 80:470--501, 1955.

\bibitem{KPP2019}
N.~{Kantas}, P.~{Parpas}, and G.~A. {Pavliotis}.
\newblock {The sharp, the flat and the shallow: Can weakly interacting agents
  learn to escape bad minima?}
\newblock {\em arXiv e-prints}, page arXiv:1905.04121, May 2019.

\bibitem{KS91}
I.~Karatzas and S.~E. Shreve.
\newblock {\em Brownian motion and stochastic calculus}, volume 113 of {\em
  Graduate Texts in Mathematics}.
\newblock Springer-Verlag, New York, second edition, 1991.

\bibitem{KV86}
C.~Kipnis and S.~R.~S. Varadhan.
\newblock Central limit theorem for additive functionals of reversible {M}arkov
  processes and applications to simple exclusions.
\newblock {\em Comm. Math. Phys.}, 104(1):1--19, 1986.

\bibitem{KLO12}
T.~Komorowski, C.~Landim, and S.~Olla.
\newblock {\em Fluctuations in {M}arkov processes}, volume 345 of {\em
  Grundlehren der Mathematischen Wissenschaften [Fundamental Principles of
  Mathematical Sciences]}.
\newblock Springer, Heidelberg, 2012.
\newblock Time symmetry and martingale approximation.

\bibitem{lions2007mean}
P.~Lions.
\newblock Mean-field games and applications.
\newblock {\em Lectures at the College de France}, 2007.

\bibitem{malrieu2003convergence}
F.~Malrieu et~al.
\newblock Convergence to equilibrium for granular media equations and their
  euler schemes.
\newblock {\em The Annals of Applied Probability}, 13(2):540--560, 2003.

\bibitem{messer1982statistical}
J.~Messer and H.~Spohn.
\newblock Statistical mechanics of the isothermal {L}ane-{E}mden equation.
\newblock {\em J. Statist. Phys.}, 29(3):561--578, 1982.

\bibitem{Oel84}
K.~Oelschl\"{a}ger.
\newblock A martingale approach to the law of large numbers for weakly
  interacting stochastic processes.
\newblock {\em Ann. Probab.}, 12(2):458--479, 1984.

\bibitem{pavliotis2008multiscale}
G.~A. Pavliotis and A.~M. Stuart.
\newblock {\em Multiscale methods}, volume~53 of {\em Texts in Applied
  Mathematics}.
\newblock Springer, New York, 2008.
\newblock Averaging and homogenization.

\bibitem{Rey2018}
J.~Reygner.
\newblock Equilibrium large deviations for mean-field systems with translation
  invariance.
\newblock {\em Ann. Appl. Probab.}, 28(5):2922--2965, 2018.

\bibitem{RV2018}
G.~M. {Rotskoff} and E.~{Vanden-Eijnden}.
\newblock {Trainability and Accuracy of Neural Networks: An Interacting
  Particle System Approach}.
\newblock {\em arXiv e-prints}, page arXiv:1805.00915, May 2018.

\bibitem{rougerie2015finetti}
N.~{Rougerie}.
\newblock {De finetti theorems, mean-field limits and bose-Einstein
  condensation}.
\newblock {\em arXiv e-prints}, page arXiv:1506.05263, Jun 2015.

\bibitem{Rue69}
D.~Ruelle.
\newblock {\em Statistical mechanics: {R}igorous results}.
\newblock W. A. Benjamin, Inc., New York-Amsterdam, 1969.

\bibitem{Ser15}
S.~Serfaty.
\newblock {\em Coulomb gases and {G}inzburg-{L}andau vortices}.
\newblock Zurich Lectures in Advanced Mathematics. European Mathematical
  Society (EMS), Z\"{u}rich, 2015.

\bibitem{Shi87}
M.~Shiino.
\newblock Dynamical behavior of stochastic systems of infinitely many coupled
  nonlinear oscillators exhibiting phase transitions of mean-field type: {H}
  theorem on asymptotic approach to equilibrium and critical slowing down of
  order-parameter fluctuations.
\newblock {\em Phys. Rev. A}, 36:2393--2412, Sep 1987.

\bibitem{Tam84}
Y.~Tamura.
\newblock On asymptotic behaviors of the solution of a nonlinear diffusion
  equation.
\newblock {\em J. Fac. Sci. Univ. Tokyo Sect. IA Math.}, 31(1):195--221, 1984.

\bibitem{Whi07}
W.~Whitt.
\newblock Proofs of the martingale {FCLT}.
\newblock {\em Probab. Surv.}, 4:268--302, 2007.

\end{thebibliography}

\end{document}